\documentclass[a4paper,10pt]{amsart}
\usepackage{amssymb,amsfonts,amsthm}
\usepackage{amstext}
\usepackage{comment}
\usepackage{amsmath}
\usepackage{mathtools}
\usepackage{hyperref}  % cria links para teoremas, lemas, ... e as referencias tambem
\usepackage{color}
\usepackage[latin1]{inputenc}
\usepackage[active]{srcltx}
\usepackage{mathrsfs} % \mathscr{B} cria  simbolos parecidos com os de sigma algebra de borel 
\usepackage{comment} % adiciona a possibilidade de usar \begin{comment} \end{comment}
\usepackage{graphicx}
\usepackage{amssymb,amsfonts,amstext, amsmath}
\usepackage[latin1]{inputenc}
\usepackage{tensor}
\usepackage{color}

\newtheorem{theorem}{Theorem}[section]
\newtheorem{corollary}[theorem]{Corollary}
\newtheorem{lemma}[theorem]{Lemma}

\newtheorem{remark}[theorem]{Remark}
\newtheorem{definition}[theorem]{Definition}

\DeclareMathAlphabet\EuScript{U}{eus}{m}{n}
\SetMathAlphabet\EuScript{bold}{U}{eus}{b}{n}

\usepackage{ dsfont }

\usepackage{chngcntr} % number the theorems and definitions on the appendix with the same  a new counting method. Must be used together with \counterwithin{thm}{section}  after \appendix

\usepackage{mathrsfs} % \mathscr{B} cria  simbolos parecidos com os de sigma algebra de borel 

\usepackage{comment} % adiciona a possibilidade de usar \begin{comment} \end{comment}

\usepackage{hyperref}  % cria links para teoremas, lemas, ... e as referencias tambem
\begin{document}
	
	\title{Hilbert space embeddings of independence  tests of several variables with radial basis functions }

\begin{abstract}
In this paper, we characterize several classes of continuous radial basis functions that can be employed to determine whether a interaction of a probability is zero or not. These functions encompass standard independence tests but also the Lancaster/Streitberg interactions, and are multivariate extensions of Bernstein functions. Addressing a gap in these two probability contexts of interactions, we introduce an indexed measure of independence that generalizes the Lancaster interaction. We present several examples of these functions derived from high-order completely monotone functions.   
\end{abstract}
\keywords{Independence tests; Lancaster and Streitberg interactions;  Radial basis  functions; Bernstein functions of several variables}

 \subjclass[2020]{ 42A82; 43A35; 44A10; 46E22; 62H15}
 	
 \author{Jean Carlo Guella}
 \address{	Universidade Estadual de Mato Grosso do Sul, Nova Andradina, Brazil}
 	 \email{jean.guella@uems.br} 	
\maketitle
	
	\tableofcontents

 \section{Introduction}

 A crucial aspect in probability and statistics, which differentiates it from measure theory,  is the concept of independence. The most classical results in the field usually requires that a probability on a Cartesian product is the product of its marginals (or in other words, that the  random variables are independent).

 However, when analyzing multivariate data we must test weather or not independence is a valid assumption. There are several known tests in the literature, but the ones connected to this text are the  Hilbert Schmidt Independence Criterion (HSIC)\cite{ Albert2022, Gretton2005, Gretton2008, petersjonasjointindp, sejdinovic2013equivalence, Poczos2012, Tjoestheim2022, Zhu2020},  Distance Covariance   \cite{ Feuerverger1993, Bakirov2006, Szekely2009,  Szekely2014, Szekely2007, Dueck2014, Janson2021, MartinezGomez2014,  Szekely2013, Yao2018} and   its generalization to several variables known as   Distance Multivariance \cite{Boettcher2018, Boettcher2019, Chakraborty2019}.

 One of our aims in this article is to provide a characterization of all radial kernels for which we can use as an independence test in all Euclidean spaces.  Precisely,   we want to characterize the   continuous  functions  $g:[0, \infty)^{n} \to \mathbb{R} $ such that for any $d \in \mathbb{N}^{n}$ is  able to discern if a discrete probability  $P $  in $\prod_{i=1}^{n}\mathbb{R}^{d_{i}}$ is equal to $\bigtimes_{i=1}^{n}P_{i}$ using a double sum (but is convenient to use an integration terminology to simplify  the expressions)
 \begin{equation}\label{objective}
 	\int_{\prod_{i=1}^{n}\mathbb{R}^{d_{i}}} \int_{\prod_{i=1}^{n}\mathbb{R}^{d_{i}}} g(\|x_{1}  - y_{1} \|^{2}, \ldots , \|x_{n}  - y_{n} \|^{2} )d[P - \bigtimes_{i=1}^{n}P_{i}](x )d[P - \bigtimes_{i=1}^{n}P_{i}]( y )>0.
 \end{equation}

 However, the set of signed measures $\{P - \bigtimes_{i=1}^{n}P_{i}, \quad P \text{ is a discrete probability}\}$ is difficult to deal as it is not a vector space. As done in \cite{guella2023} for the case $n=2$, if we restrict the functions that satisfies Equation \ref{objective} by demanding that it can differentiate if or not $P-Q=0$, provided that $P_{i}=Q_{i}$ for any $1\leq i \leq n$, we are essentially analyzing the problem on the vector space $\mathcal{M}_{2}(\prod_{i=1}^{n}\mathbb{R}^{d_{i}})$ (see Remark \ref{hanhjordanequivalence}).
 
 The concept that will allow us to understand this problem is of a  positive definite independent function of order $2$ (PDI$_{2,n}^{\infty}$), which is characterized in 	Theorem \ref{berns2sevndimpart2}  and on   Corollary \ref{SPDPDIK}  we provide a characterization for  when the inequality in Equation \ref{objective} holds ($k=2$). Interestingly, as the Corollary \ref{SPDPDIK} states,  the additional restriction that we made on Equation \ref{objective}, which simplified its analysis do not restrict the initial set of functions we were aiming at.   
 
 The interest in the functions that satisfies \ref{objective} is to obtain an all purpose independence test  on Euclidean spaces (that is, with no restrictions in the dimension).

 However, in some cases the multivariate data might not be independent, but the probability might interact with its marginals on a different way  that is relevant for the analysis. As far as we know there is no formal definition of a interaction for a probability, but the way we use in this text is if it satisfies some algebraic relation between the probability and its marginals.

 Two types of interactions have gained attention in the literature of kernel methods recently: the Streitberg \cite{streitberg1990lancaster} and the Lancaster interaction \cite{lancaster1969chi} (they are defined in Section \ref{Terminology},   and they are part of the  broader context of partition lattices \cite{NEURIPS2023_74f11936}). A natural question is then to obtain a characterization of which continuous  functions  $g:[0, \infty)^{n} \to \mathbb{R} $ satisfies that  for  any $d \in \mathbb{N}^{n}$ is  able to discern if a discrete probability  $P $  in $\prod_{i=1}^{n}\mathbb{R}^{d_{i}}$ satisfies that   $\Sigma[P]=0$ (or $\Lambda[P]=0$) if and only if  
 \begin{equation}\label{objective2}
 	\int_{\prod_{i=1}^{n}\mathbb{R}^{d_{i}}} \int_{\prod_{i=1}^{n}\mathbb{R}^{d_{i}}} g(\|x_{1}  - y_{1} \|^{2}, \ldots , \|x_{n}  - y_{n} \|^{2} )d\Sigma[P](x )d\Sigma[P](y)=0.
 \end{equation} 
 
 Similar to the independence tests above, initially, such task is difficult, but if we additionally impose that   the functions that satisfies Equation \ref{objective2} also can  differentiate if or not $P-Q=0$, provided that $P_{F}=Q_{F}$ for any $F \subset \{1, \ldots, n\}$, $|F|\leq n-1$, we are essentially analyzing the problem on the vector space $\mathcal{M}_{n}(\prod_{i=1}^{n}\mathbb{R}^{d_{i}})$ (see Remark \ref{hanhjordanequivalence}, and we prove in Lemma \ref{exm02xn} that $\Sigma[P]$ and $\Lambda[P]$ are elements of $\mathcal{M}_{n}(\prod_{i=1}^{n}\mathbb{R}^{d_{i}})$ for any $P$). 
 
 The concept that will allow us to understand this extended  problem is of a  positive definite independent function of order $n$ (PDI$_{n}^{\infty}$), which is characterized in Theorem \ref{basicradialndim} and on Corollary \ref{SPDPDIK} we provide a characterization for  when inequality \ref{objective2} holds ($k=n$). As Corollary \ref{SPDPDIK} states,  the additional restriction that we made on Equation \ref{objective2}, which simplified  its analysis do not restrict the initial set of functions we were aiming at. We emphasize that we obtained that the tests that are able to discern if or not $\Sigma[P]=0$ are the same as the ones of  $\Lambda[P]=0$, even though those two equalities have (possible) different conclusions.

 However,  by Lemma \ref{exm01xn} and   Lemma \ref{exm02xn}, we see that there is a gap between independence tests and the Lancaster/Streitberg interactions. We fill such gap in two ways: for $2< k< n$ we define in Section \ref{Terminology}    intermediate vector spaces $\mathcal{M}_{k}$, which lies   between $\mathcal{M}_{2}$  and $\mathcal{M}_{n}$,  and based on them we  define a generalization of the Lancaster interaction with an index $k$, where when $k=2$ we have an independence test and when $k=n$ we have the standard  Lancaster interaction.

 Those classes of functions must be seen as multivariate generalizations of the  Schoenberg results about  the positive definite and conditionally negative definite radial kernels on all Euclidean spaces   in \cite{schoenbradial}, which are respectively $k=0$ and $k=1$. The proof we present is not constructive, instead we rely on the fact that the representation for the conditionally negative definite case is unique (Theorem \ref{reprcondneg}), and some techniques based in \cite{Guella2020} concerning measure valued positive definite radial  kernels on Euclidean spaces.

 We conclude the text in Section \ref{sumsbased}, where we present    a  characterization for the   continuous functions $\psi: [0, \infty)  \to \mathbb{R}$, for which the kernel 
 $$
 \psi( \|x_{1}-y_{1}\|^{2}+\ldots +\|x_{n} - y_{n}\|^{2} ), \quad  x_{i}, y_{i} \in \mathbb{R}^{d_{i}}
 $$
 is positive definite independent of order $k$  for every  $d \in \mathbb{N}^{n}$, by relating  those functions with  completely monotone functions of order $k$, which there are plenty of examples in the literature such as  $\psi(t)=(-1)^{k} t^{a}$, $a \in (k-1,k)$.

 \section{Definitions}\label{Definitions}

 In this Section, we make a review of the most important results and definitions that will be required for the development of the text. Some of those results were presented and developed  in  \cite{guella2023}, where it is presented the theory of positive definite independent kernels in two variables, but to maintain a self contained text we  reintroduce them. Proof of these results  can be found in the references mentioned in the text.

 \begin{center}
 	Positive definite and conditionally negative definite kernels
 \end{center}

 A symmetric kernel $K: X \times X \to \mathbb{R}$ is called Positive Definite (PD) if for every finite quantity of distinct points $x_{1}, \ldots, x_{n} \in X$ and scalars $c_{1}, \ldots, c_{n} \in \mathbb{R}$, we have that
 $$
 \sum_{i, j =1}^{n}c_{i}c_{j} K(x_{i}, x_{j}) \geq 0.
 $$

 A symmetric kernel $\gamma: X \times X \to \mathbb{R}$ is called Conditionally Negative Definite (CND) if for every finite quantity of distinct points $x_{1}, \ldots, x_{n} \in X$ and scalars $c_{1}, \ldots, c_{n} \in \mathbb{R}$, with the restriction that $\sum_{i=1}^{n}c_{i}=0$, we have that
 $$
 \sum_{i, j =1}^{n}c_{i}c_{j} \gamma(x_{i}, x_{j}) \leq 0.
 $$

 The concept of CND kernels is intrinsically related to PD kernels, as a  symmetric kernel $\gamma: X\times X\to \mathbb{R}$ is CND if and only if  for any (or equivalently, for every) $w \in X$ the kernel
 \begin{equation}\label{Kgamma}
 	K_{\gamma}^{w}(x,y):=\gamma(x,w) + \gamma(w, y) - \gamma(x,y) - \gamma(w,w)
 \end{equation}
 is positive definite. With this result is possible to explain the relation between CND kernels and Hilbert spaces as  if   $\gamma: X \times X \to \mathbb{R}$ is CND it can  be written as 
 \begin{equation}\label{condequa}
 	\gamma(x,y)= \|h(x)- h(y)\|_{\mathcal{H}}^{2} + \gamma(x,x)/2 + \gamma(y,y)/2
 \end{equation}
 where $\mathcal{H}$ is a real Hilbert space and $h: X \to \mathcal{H}$. Another famous relation is that a  symmetric kernel $\gamma: X\times X\to \mathbb{R}$ is CND if and only if  for every $r>0$  the kernel
 \begin{equation}\label{schoenmetriccond}
 	(x,y) \in X\times X \to  e^{-r\gamma(x,y)}
 \end{equation}
 is PD. 
 
 From Equation \ref{condequa} we obtain that an CND kernel $\gamma$ is continuous if and only if is continuous in the diagonal, that is, $\gamma(z^{1}_{n},z_{n}^{2}) \to \gamma(z,z)$ whenever $z_{n}^{1}$ and $z_{n}^{2}$ converges to $z$.
 
 Those  classical results about CND kernels  can be found in Chapter $3$ at \cite{Berg1984}.

 The characterization of  the continuous CND radial kernels in all Euclidean spaces was proved in \cite{schoenbradial}, and is the following:  
 \begin{theorem}\label{reprcondneg} Let $\psi :[0, \infty)\to \mathbb{R}$ be a continuous function. The following conditions are equivalent
 	\begin{enumerate}
 		\item[$(i)$] The kernel
 		$$
 		(x,y) \in \mathbb{R}^{d}\times \mathbb{R}^{d} \to \psi(\|x-y\|^{2}) \in \mathbb{R}
 		$$
 		is CND for every $d \in \mathbb{N}$. 
 		\item[$(ii)$] The function $\psi$ can be represented as
 		$$
 		\psi(t)=\psi(0)+ \int_{[0,\infty)}(1-e^{-rt})\frac{1+r}{r}d\eta(r),
 		$$  
 		for all $t \geq 0$, where  $\eta$ is a nonnegative  measure on $\mathfrak{M}([0,\infty))$. The  representation is  unique.
 		\item[$(iii)$] The function $\psi \in C^{\infty}(0,\infty))$ and $\psi^{(1)}$ is completely monotone, that is, $(-1)^{n}\psi^{(n+1)}(t) \geq 0$, for every $n\in \mathbb{Z}_{+}$ and $t>0$.
 	\end{enumerate}
 \end{theorem}

 A continuous function $\psi:[0, \infty)\to \mathbb{R}$ that satisfies the relation $(iii)$ in  Theorem \ref{reprcondneg}  is called a Bernstein function (we do not need to assume that Bernstein functions are nonnegative), and the same theorem provides a representation for it. For more information on Bernstein functions see \cite{Schilling2012}. The value of the function $(1-e^{-rt})(1+r)/r$ at  $r=0$ is defined as the limit of $r \to 0$, that is, its value is $t$. Usually, the integral on the set $[0, \infty)$ is separated in the integral at $\{0\}$ plus the integral on the set $(0, \infty)$, we do not present it in this way as   the notation and terminology of the proofs in Section \ref{Bernsteinfunctionsofordern} are considerably simplified by using this simple modification. 
 
 The following two simple inequalities are necessary for the proof of Theorem \ref{reprcondneg}  and are useful for the development of this text 
 \begin{equation}\label{bern1ineq}
 	1 \leq (1-e^{-s})\frac{1+s}{s} \leq 2, \quad s \geq 0,
 \end{equation}
 \begin{equation}\label{bern1ineq2}
 	\min (1,t) \leq (1-e^{-rt})\frac{1+r}{r} \leq 2 \max (1,t), \quad r,t \geq 0.
 \end{equation}

 \begin{center}
 	Positive definite independent kernels
 \end{center}
 
 In  \cite{guella2023} a new type of kernel was defined, whose idea is to  generalize the existing  kernel methods approach for independence tests in two variables.

 \begin{definition}\label{PDI}Let $X$ and $Y$ be non empty sets. We say that a symmetric kernel $\mathfrak{I}: (X\times Y) \times (X\times Y) \to \mathbb{R}$ is a   Positive Definite Independent Kernel (PDI)    if for every finite quantity of distinct points $x_{1}, \ldots, x_{n} \in X$, $y_{1}, \ldots , y_{m} \in Y$  and real scalars $c_{i, k}$, with the restrictions
 	$$
 	\sum_{i=1}^{n}c_{i,k}=0 , \quad \sum_{l=1}^{m}c_{j, l}=0, 
 	$$
 	for every $1\leq k \leq m$, $1\leq j \leq n$, it satisfies
 	$$
 	\sum_{i,j=1}^{n}\sum_{k,l=1}^{m}c_{i, k}c_{j, l}\mathfrak{I}((x_{i},y_{k}),(x_{j}, y_{l})) \geq 0.
 	$$
 \end{definition}

 This definition is inspired by the fact that for arbitrary points $x_{1}, \ldots, x_{n} \in X$, $y_{1}, \ldots , y_{m} \in Y$  and real scalars $c_{i, k}$ that satisfies the restrictions in Definition \ref{PDI}, there exists   discrete probabilities $P, P^{\prime}$ in $X\times Y$ with the same marginals and an $M \geq 0$ (obtained from its Hahn-Jordan decomposition) for which 
 $$
 \left ( \sum_{i=1}^{n} \sum_{k=1}^{m}c_{i,k}\delta_{(x_{i}, y_{k})} \right )=M[P- P^{\prime}].
 $$
 Conversely, if $P, P^{\prime}$ are discrete probabilities in $X\times Y$ with the same marginals,  then there exists points $x_{1}, \ldots, x_{n} \in X$, $y_{1}, \ldots , y_{m} \in Y$  and real scalars $c_{i, k}$ that satisfies the restrictions in Definition \ref{PDI} for which the previous equality holds.

 \begin{theorem}\label{basicradial}
 	Let $g: [0, \infty)^{2} \to \mathbb{R}$ be a  continuous function  that is zero at the boundary. The following conditions are equivalent:
 	\begin{enumerate}
 		\item [$(i)$] The kernel
 		$$
 		g( \|x_{1}-y_{1}\|^{2}, \|x_{2}-y_{2}\|^{2} ), \quad  x_{1},y_{1} \in  \mathbb{R}^{d}, \quad x_{2}, y_{2} \in \mathbb{R}^{d^{\prime}}
 		$$ is PDI  for every  $d, d^{\prime }\in \mathbb{N}$.
 		\item [$(ii)$] The kernel can be represented as
 		\begin{align*}
 			g( \|x_{1}-y_{1}\|^{2}, \|x_{2}-y_{2}\|^{2})=\int_{[0,\infty)^{2}}\prod_{i=1}^{2} \frac{(1-e^{-r_{i}\|x_{i}-y_{i}\|^{2}})}{r_{i}}(1 +r_{i})d\eta(r_{1},r_{2}) 
 		\end{align*}
 		where  the measure  $\eta\in \mathfrak{M}([0,\infty)^{2})$ is nonnegative.
 		\item [$(iii)$] The function $g$ is a Bernstein function   of two variables.
 	\end{enumerate}
 \end{theorem}

 \section{Vector spaces of measures and probability interactions}\label{Terminology}	
 
 Let  $X_{i}$, $1\leq i \leq n$, be non empty sets and consider the $n-$Cartesian product $\prod_{i=1}^{n}X_{i}$, which we  denote as $\mathds{X}_{n}$.

 For  $m, n \in \mathbb{N}$ we define the set  $ \mathbb{N}_{m}^{n}:=\{1,\ldots , m\}^{n}$, which has $m^{n}$ elements, similarly we define $\mathbb{N}_{m}^{0,n}:=\{0, 1, \ldots, m\}^{n}$  which has $(m+1)^{n}$ elements.  If $x_{i}^{1}, \ldots, x_{i}^{m} \in X_{i}$, $1\leq i \leq n$, we define for $\alpha=(\alpha_{1}, \ldots, \alpha_{n}) \in \mathbb{N}_{m}^{n}$ (or $\mathbb{N}_{m}^{0,n}$) the element $x_{\alpha}:=(x_{1}^{\alpha_{1}}, \ldots, x_{n}^{\alpha_{n}})$.

 We  frequently use $\vec{1}$ as a  vector in which all entries are equal to $1$, similarly $\vec{0}$ and  $\vec{2}$, the dimension of those vectors are omitted as they are clear from the context. Also, for a subset $F \subset \{1, \ldots, n\}$ and coefficientes $\alpha, \beta \in \mathbb{N}^{n}$, we use notations such as  $x_{\alpha_{F} + \beta_{F^{c}}}$ to indicate the element in $\mathds{X}_{n}$, in which the coordinates in  $F$ are the same as the ones from $x_{\alpha}$  and  the coordinates in  $F^{c}$ are the same as the ones from $x_{\beta}$.

 Even though the results presented in Section \ref{Bernsteinfunctionsofordern},	 \ref{Partialindependencetests} and \ref{sumsbased}   are on a discrete scenario, it is convenient to use an integral terminology to simplify some expressions. For that, we define 
 
 \begin{equation*}\label{medidaarbi}
 	\mathcal{M}( \mathds{X}_{n}):=\{\text{The vector space of all discrete measures in } \mathds{X}_{n}\}. 
 \end{equation*}

 Some important subspaces of $\mathcal{M}( \mathds{X}_{n})$ for the development of this text are  
 
 \begin{equation*}\label{medidacomplmarg}
 	\mathcal{M}_{n}( \mathds{X}_{n}):=\{ \mu \in \mathcal{M}( \mathds{X}_{n}), \quad \mu (\prod_{i=1}^{n}A_{i} )=0, \text{ if } | \{i,\quad  A_{i}=X_{i}\} |\geq 1 \}, 
 \end{equation*}
 which will be important in Section \ref{Bernsteinfunctionsofordern}. Also
 \begin{equation*}\label{medidamarg}
 	\mathcal{M}_{2}( \mathds{X}_{n}):=\{\mu \in \mathcal{M}( \mathds{X}_{n}), \quad \mu (\prod_{i=1}^{n}A_{i} )=0, \text{ if } | \{i,\quad  A_{i}=X_{i}\} |\geq n-1 \}. 
 \end{equation*}
 which is  related to independence tests, in this sense we emphasize  Lemma \ref{exm01xn} which will be used to prove   Theorem \ref{berns2sevndimpart2}. More generally, for $0\leq k \leq n$ 
 \begin{equation*}\label{medidamargk}
 	\mathcal{M}_{k}( \mathds{X}_{n}):=\{\mu \in \mathcal{M}( \mathds{X}_{n}), \quad \mu (\prod_{i=1}^{n}A_{i} )=0, \text{ if } | \{i,\quad  A_{i}=X_{i}\} |\geq n-k+1 \}. 
 \end{equation*}	
 which for $2\leq k \leq n$ are a generalization of independence test, in this sense we emphasize Theorem \ref{generallancaster} and Theorem \ref{genlancastercartesianproduct} which are  used to prove Theorem \ref{bernsksevndimpart3}. Note that $\mathcal{M}_{0}( \mathds{X}_{n}):= \mathcal{M}( \mathds{X}_{n})$ and that $	\mathcal{M}_{1}( \mathds{X}_{n})$ is related to the definition of conditionally negative definite kernels in $ \mathds{X}_{n}$. They satisfy the following  inclusion relation
 
 \begin{equation}\label{inclumeas}
 	\mathcal{M}_{n}( \mathds{X}_{n}) \subset \mathcal{M}_{n-1}( \mathds{X}_{n}) \subset  \ldots \subset \mathcal{M}_{2}( \mathds{X}_{n}) \subset \mathcal{M}_{1}( \mathds{X}_{n}) \subset \mathcal{M}_{0}( \mathds{X}_{n}).  
 \end{equation}

 A technical property that we frequently use for a measure $\mu$ in $ \mathcal{M}_{k}(\mathds{X}_{n})$  when  $k\geq 1$, is if $f: \mathds{X}_{n}\to \mathbb{R}$ only depends  of $k-1$ of its $n$ variables (for instance if $f(x_{1}, \ldots, x_{n}) = g(x_{1}, \ldots, x_{k-1})$ for some $g: \prod_{i=1}^{k-1}X_{i}\to \mathbb{R}$) then 
 \begin{equation}\label{integmu0n}
 	\int_{\mathds{X}_{n}}f(x_{1}, \ldots,x_{n})d\mu(x_{1}, \ldots, x_{n})=0.
 \end{equation}
 
 By the pigeonhole principle,  if $\mu_{i}  \in \mathcal{M}(X_{i})$, $1\leq i \leq n$,  with the restriction that  $|i, \quad \mu_{i}(X_{i})=0|\geq k$, then $(\bigtimes_{i=1}^{n}\mu_{i}  )$ is an element of $\mathcal{M}_{k}(\mathds{X}_{n})$. This crucial simple property appears in Theorem \ref{basicradialndim} and  \ref{bernsksevndimpart3}. 
 
 \begin{remark}\label{hanhjordanequivalence} When $k \geq 1$,  by the  Hahn-Jordan decomposition,  if $\mu \in \mathcal{M}_{k}(\mathds{X}_{n})$  then there exists an $M \in \mathbb{R}$ and  probabilities $P$ and $P^{\prime}$	in $\mathcal{M}(\mathds{X}_{n})$  such that $P_{F}= (P^{\prime})_{F}$ for any $F \subset \{1, \ldots, n\}$ that satisfies $|F|=k-1$ and
 	$$
 	\mu = M[P - P^{\prime}].
 	$$
 	Similarly, if two probabilities $P$ and $P^{\prime}$	in $\mathcal{M}(\mathds{X}_{n})$  are such that $P_{F}= (P^{\prime})_{F}$ for any $F \subset \{1, \ldots, n\}$ that satisfies $|F|=k-1$, then $P - P^{\prime}$ is an element of $\mathcal{M}_{k}(\mathds{X}_{n})$.
 \end{remark}
 
 An important class of examples for those spaces are the following measures.

 \begin{lemma}\label{measorderk}Let $1\leq k \leq n$ and  $x_{\vec{1}}, x_{\vec{2}}  \in \mathds{X}_{n} $, the measure
 	$$
 	\mu^{n}_{k}[x_{\vec{1}}, x_{\vec{2}}]:= \delta_{x_{\vec{1}}}+  \sum_{j=0}^{k-1}(-1)^{k-j}\binom{n-j-1}{n-k}\sum_{|F|=j}\delta_{x_{\vec{1}_{F} + \vec{2}_{F^{c}} }}.
 	$$
 	is an element of  	$\mathcal{M}_{k}( \mathds{X}_{n}) $. Further,  if $L:=\{i, \quad x_{i}^{1}\neq  x_{i}^{2}\}$, then  when  $|L|<   k$ the measure $\mu^{n}_{k}[x_{\vec{1}}, x_{\vec{2}}]$ is zero, otherwise
 	$$
 	\mu^{n}_{k}[x_{\vec{1}}, x_{\vec{2}}]= 	\mu^{|L|}_{k}[x_{\vec{1}_{L}}, x_{\vec{2}_{L}}]\times \delta_{x_{\vec{1}_{L^{c}}}}, 
 	$$
 	and when $|L|> k $ we have that $\mu^{|L|}_{k}[x_{\vec{1}_{L}}, x_{\vec{2}_{L}}] \notin \mathcal{M}_{k+1}^{|L|}( \mathds{X}_{L}) $.

 \end{lemma}
 
 \begin{proof} We prove this result by induction in $n$. If $n=k$, we have that	
 	$$
 	\mu^{k}_{k}[x_{\vec{1}}, x_{\vec{2}}]:= \delta_{x_{\vec{1}}}+  \sum_{j=0}^{k-1}(-1)^{k-j}\sum_{|F|=k}\delta_{x_{\vec{1}_{F} + \vec{2}_{F^{c}} }}=  \bigtimes_{i=1}^{k}(\delta_{x_{i}^{1}} - \delta_{x_{i}^{2}}),
 	$$	
 	which is clearly a measure in $\mathcal{M}_{k}( \mathds{X}_{k}) $ which is nonzero if and only if all coordinates of $x_{\vec{1}}$ and $x_{\vec{2}}$ are different.\\	
 	For $0\leq j \leq  n+1$ we define the measure $P_{j}^{n+1}:=\sum_{|F|=j}\delta_{x_{\vec{1}_{F} + \vec{2}_{F^{c}} }} $ in $\mathcal{M}( \mathds{X}_{n+1}) $. For an  $F \subset \{1, \ldots, n+1\}$ with  $|F|= j$,  either $F$ does not contain the element $n+1$ or
 	it contains and has $j-1$ elements in $\{1, \ldots, n\}$. From the the bijection between those sets $F$ that contains the element $n+1$ and the subsets of  $\{1, \ldots, n\}$ that has $j-1$ terms, we obtain that (we define $P_{-1}^{n}$ as the zero measure)
 	\begin{equation}\label{measorderkeq1}
 		P_{j}^{n+1}(A \times A_{n+1})=P_{j}^{n}(A)\delta_{x_{n+1}^{2}}(A_{n+1}) + P_{j-1}^{n}(A)\delta_{x_{n+1}^{1}}(A_{n+1}),
 	\end{equation}
 	for every $A \subset \mathds{X}_{n},   A_{n+1} \subset X_{n+1}$. Thus, we obtain that
 	\begin{align*}
 		\mu^{n+1}_{k}[x_{\vec{1}}, x_{\vec{2}}]&(A \times X_{n+1})= \delta_{x_{\vec{1}}}(A)+  \sum_{j=0}^{k-1}(-1)^{k-j}\binom{n-j}{n+1-k}[P_{j}^{n}(A)  + P_{j-1}^{n}(A)]	\\
 		&=\delta_{x_{\vec{1}}}(A)+  \sum_{j=0}^{k-1}(-1)^{k-j}\left [\binom{n-j}{n+1-k}-\binom{n-j-1}{n+1-k}  \right ]P_{j}^{n}(A)\\
 		&=\mu^{n }_{k}[x_{\vec{1}}, x_{\vec{2}}](A).
 	\end{align*}
 	Applying this relation recursively $n-k+1$ times we obtain that
 	$$
 	\mu^{n+1}_{k}[x_{\vec{1}}, x_{\vec{2}}](A \times (\prod_{i=k+1}^{n}  X_{i} ) )= \mu^{k}_{k}[x_{\vec{1}}, x_{\vec{2}}](A)= \bigtimes_{i=1}^{k}(\delta_{x_{i}^{1}} - \delta_{x_{i}^{2}})(A),
 	$$
 	more generally, by the symmetry of this measure, we have that for any $F \subset \{1, \ldots, n+1\}$ with $|F|= k$
 	$$
 	\mu^{n+1}_{k}[x_{\vec{1}}, x_{\vec{2}}]((\prod _{i\in F}A_{i})\times (\prod _{i\in F^{c}} X_{i}))=\bigtimes_{i\in F}(\delta_{x_{i}^{1}} - \delta_{x_{i}^{2}}) (\prod _{i\in F}A_{i})
 	$$
 	thus concluding that $\mu^{n+1}_{k}[x_{\vec{1}}, x_{\vec{2}}] \in \mathcal{M}_{k}( \mathds{X}_{n+1}) $. \\
 	Now, note that if $x_{n+1}^{1}= x_{n+1}^{2}$, Equation \ref{measorderkeq1} can be rewritten as
 	$$
 	P_{j}^{n+1}(A \times A_{n+1})=[P_{j}^{n}(A) + P_{j-1}^{n}(A)] \delta_{x_{n+1}^{1}}(A_{n+1}), \quad A \subset \mathds{X}_{n}, \quad  A_{n+1} \subset X_{n+1},
 	$$
 	and then
 	\begin{align*}
 		&\mu^{n+1}_{k}[x_{\vec{1}}, x_{\vec{2}}](A \times A_{n+1})\\
 		&= \delta_{x_{\vec{1}}}(A)\delta_{x_{n+1}^{1}}(A_{n+1})-  \sum_{j=0}^{k-1}(-1)^{k-j}\binom{n-j}{n+1-k}[P_{j}^{n}(A)  + P_{j-1}^{n}(A)] \delta_{x_{n+1}^{1}}(A_{n+1})\\
 		&	=\mu^{n }_{k}[x_{\vec{1}}, x_{\vec{0}}](A) \delta_{x_{n+1}^{1}}(A_{n+1}).
 	\end{align*}
 	By the symmetry of the measures  involved and  applying the previous equality  recursively  we obtain the second part of the Lemma.\end{proof}

 We conclude this Section by presenting  additional properties for the cases $2\leq k \leq n$. The one for the case $k=2$  is very important for the equivalence between embedding independence tests in Hilbert spaces and the kernels related to  the vector space $\mathcal{M}_{2}(\mathds{X}_{n})$, presented in Theorem \ref{berns2sevndimpart2}.

 \begin{lemma}\label{exm01xn}
 	Let   $\mu_{i}  \in \mathcal{M}(X_{i})$, $1\leq i \leq n$,  with the restriction that for at least two of those measures $\mu_{i}(X_{i})=0$, then there exists an $M \geq 0$, a probability $P$ in $\mathcal{M}(\mathds{X}_{n})$  with marginals $P_{i}$ in $\mathcal{M}(X_{i})$, for which $ \left (\bigtimes_{i=1}^{n}\mu_{i}\right ) = (-1)^{n}M [ P - (\bigtimes_{i=1}^{n} P_{i})]$.	
 \end{lemma}
 
 \begin{proof}We may assume that all measures are nonzero, because otherwise we take $M=0$.  \\
 	For convenience, we assume that $\{i, \quad \mu_{i}(X_{i})=0\}= \{1, 2\}$, then  for $i\in \{1,2\}$ a Jordan decomposition of $\mu_{i}$  can be written as $\mu_{i}= b_{i}^{1}[S_{i}^{1}  - S_{i}^{2} ]$, where $b_{i}$ is positive  and $S_{i}^{1}, S_{i}^{2}$ are probabilities in $X_{i}$. Then 
 	$$
 	\mu_{1}\times \mu_{2} = B[S_{1}^{1} - S_{1}^{2}]\times [S_{2}^{1} - S_{2}^{2}]= B \left [ (S_{1}^{1}\times S_{2}^{1} + S_{1}^{2}\times S_{2}^{2}   )  -(S_{1}^{2}\times S_{2}^{1} + S_{1}^{1}\times S_{2}^{2}   )\right ] 
 	$$  
 	where $B=  b_{1}b_{2}$.  Also, for $i\geq 3$,  if $\mu_{i}= c_{i}^{1}R_{i}^{1}  - c_{i}^{2}R_{i}^{2} $ is a  Hahn-Jordan decomposition, where $c_{i}^{1}, c_{i}^{2}$ are nonnegative and $R_{i}^{1}, R_{i}^{2}$ are probabilities in $X_{i}$,   then  
 	$$
 	\bigtimes_{i=3}^{n}\mu_{i} = \sum_{\beta\in \mathbb{N}_{2}^{n-2}}(-1)^{n-2- |\beta|}c_{\beta}R_{\beta} = (-1)^{n} \left [ \sum_{|\beta| \in 2\mathbb{N}}c_{\beta}R_{\beta} - \sum_{|\beta| \in 2\mathbb{N}+1}c_{\beta}R_{\beta}\right ] 
 	$$
 	where $c_{\beta} := \prod_{i=3}^{n}c_{i}^{\beta(i-2)}$ and $R_{\beta }:= \bigtimes_{i=1}^{n-2}R_{i}^{\beta(i-2)}$. Thus
 	\begin{align*}
 		&(-1)^{n}\bigtimes_{i=1}^{n}\mu_{i}\\
 		&= B\left [ (S_{1}^{1}\times S_{2}^{1} + S_{1}^{2}\times S_{2}^{2}   )  -(S_{1}^{2}\times S_{2}^{1} + S_{1}^{1}\times S_{2}^{2}   )\right ]   \times \left [ \sum_{|\beta| \in 2\mathbb{N}}c_{\beta}R_{\beta} - \sum_{|\beta|   \in 2\mathbb{N}+1}c_{\beta}R_{\beta}\right ]  \\  
 		&= B  \left ( S_{1}^{1}\times S_{2}^{1} + S_{1}^{2}\times S_{2}^{2}    \right ) \times \left ( \sum_{|\beta| \in 2\mathbb{N}}c_{\beta}R_{\beta} \right ) + B\left (S_{1}^{2}\times S_{2}^{1} + S_{1}^{1}\times S_{2}^{2} \right ) \times \left (  \sum_{|\beta| \in 2\mathbb{N}+1}c_{\beta}R_{\beta} \right )\\
 		&  -B\left ( S_{1}^{1}\times S_{2}^{1} + S_{1}^{2}\times S_{2}^{2}  \right ) \times \left (  \sum_{|\beta| \in 2\mathbb{N}+1}c_{\beta}R_{\beta} \right )-B\left ( S_{1}^{2}\times S_{2}^{1} + S_{1}^{1}\times S_{2}^{2} \right ) \times \left (  \sum_{|\beta| \in 2\mathbb{N}}c_{\beta}R_{\beta} \right ). 
 	\end{align*}
 	Define the probability
 	$$
 	P:= \frac{1}{D}\left [\left ( S_{1}^{1}\times S_{2}^{1} + S_{1}^{2}\times S_{2}^{2}    \right ) \times \left ( \sum_{|\beta| \in 2\mathbb{N}}c_{\beta}R_{\beta} \right ) + \left (S_{1}^{2}\times S_{2}^{1} + S_{1}^{1}\times S_{2}^{2} \right ) \times \left (  \sum_{|\beta| \in 2\mathbb{N}+1}c_{\beta}R_{\beta} \right ) \right ]
 	$$
 	where $D = 2\sum_{\beta\in \mathbb{N}_{2}^{n-2}}c_{\beta}$. Then, 
 	\begin{align*}
 		P_{\{2, \ldots n\}}&= \frac{1}{D} \left [  \left (  S_{2}^{1} +  S_{2}^{2}  \right ) \times \left ( \sum_{|\beta| \in 2\mathbb{N}}c_{\beta}R_{\beta} \right ) +  \left ( S_{2}^{1} +  S_{2}^{2}  \right ) \times \left (  \sum_{|\beta| \in 2\mathbb{N}+1}c_{\beta}R_{\beta} \right )  \right ]\\
 		&= \frac{2}{D}     \left ( \frac{ S_{2}^{1} +  S_{2}^{2}}{2}   \right ) \times \left ( \sum_{\beta\in \mathbb{N}_{2}^{n-2}}c_{\beta}R_{\beta} \right )  \\
 		&=   \left ( \frac{ S_{2}^{1} +  S_{2}^{2}}{2}   \right )\times \left ( \bigtimes_{j=3 }^{n} \frac{c_{j}^{1}R_{j}^{1} + c_{j}^{2}R_{j}^{2}}{c_{j}^{1} + c_{j}^{2}}   \right ), 
 	\end{align*}
 	and similarly,
 	$$
 	P_{\{1,3 \ldots n\}}=   \left ( \frac{ S_{1}^{1} +  S_{1}^{2}}{2}   \right )\times \left ( \bigtimes_{j=3 }^{n} \frac{c_{j}^{1}R_{j}^{1} + c_{j}^{2}R_{j}^{2}}{c_{j}^{1} + c_{j}^{2}}   \right ).
 	$$		
 	To conclude, from those relations we obtain that 
 	$$
 	P_{i}= \frac{S_{i}^{1} + S_{i}^{2}}{2}, \text{ for $1\leq i \leq 2$  and } P_{i}= \frac{c_{j}^{1}R_{j}^{1} + c_{j}^{2}R_{j}^{2}}{c_{j}^{1} + c_{j}^{2}}, \text{ for $i \geq 3$},
 	$$
 	and
 	\begin{align*}
 		&\bigtimes_{j=1}^{n} P_{i}= \left(\frac{1}{4} \frac{1}{\prod_{j=3}^{n}(c_{j}^{1} +   c_{j}^{2})} \right) \left ( \sum_{\alpha \in \mathbb{N}_{2}^{2}} S_{\alpha}  \right ) \times \left (  \sum_{\beta \in \mathbb{N}_{2}^{n-2}} c_{\beta}R_{\beta}   \right )\\
 		&= \frac{1}{2D}\left ( S_{1}^{1}\times S_{2}^{1} + S_{1}^{2}\times S_{2}^{2}    \right ) \times \left ( \sum_{|\beta| \in 2\mathbb{N}}c_{\beta}R_{\beta} \right ) + \frac{1}{2D}\left (S_{1}^{2}\times S_{2}^{1} + S_{1}^{1}\times S_{2}^{2} \right ) \times \left (  \sum_{|\beta| \in 2\mathbb{N}+1}c_{\beta}R_{\beta} \right )\\
 		&  +\frac{1}{2D}\left ( S_{1}^{1}\times S_{2}^{1} + S_{1}^{2}\times S_{2}^{2}  \right ) \times \left (  \sum_{|\beta| \in 2\mathbb{N}+1}c_{\beta}R_{\beta} \right )+ \frac{1}{2D}\left ( S_{1}^{2}\times S_{2}^{1} + S_{1}^{1}\times S_{2}^{2} \right ) \times \left (  \sum_{|\beta| \in 2\mathbb{N}}c_{\beta}R_{\beta} \right ),
 	\end{align*}
 	thus
 	\begin{align*}
 		& P -\bigtimes_{j=1}^{n} P_{i}\\
 		&= \frac{1}{2D}\left [\left ( S_{1}^{1}\times S_{2}^{1} + S_{1}^{2}\times S_{2}^{2}    \right ) \times \left ( \sum_{|\beta| \in 2\mathbb{N}}c_{\beta}R_{\beta} \right ) + \left (S_{1}^{2}\times S_{2}^{1} + S_{1}^{1}\times S_{2}^{2} \right ) \times \left (  \sum_{|\beta| \in 2\mathbb{N}+1}c_{\beta}R_{\beta} \right )\right ]\\
 		& - \frac{1}{2D}\left [\left ( S_{1}^{1}\times S_{2}^{1} + S_{1}^{2}\times S_{2}^{2}  \right ) \times \left (  \sum_{|\beta| \in 2\mathbb{N}+1}c_{\beta}R_{\beta} \right )+  \left ( S_{1}^{2}\times S_{2}^{1} + S_{1}^{1}\times S_{2}^{2} \right ) \times \left (  \sum_{|\beta| \in 2\mathbb{N}}c_{\beta}R_{\beta} \right )\right ]\\
 		&= \frac{1}{2D}  \left [\left [(S_{1}^{1}\times S_{2}^{1} + S_{1}^{2}\times S_{2}^{2}   )  -(S_{1}^{2}\times S_{2}^{1} + S_{1}^{1}\times S_{2}^{2}   ) \right ]\times \left ( \sum_{|\beta| \in 2\mathbb{N}}c_{\beta}R_{\beta} - \sum_{|\beta| \in 2\mathbb{N}+1}c_{\beta}R_{\beta} \right ) \right ]\\
 		&= \frac{(-1)^{n}}{2BD}\bigtimes_{i=1}^{n}\mu_{i}.
 \end{align*}\end{proof}

 For the additional property when $n=k$, we need to define two objects, the Streitberg and the Lancaster interaction measures.

 A partition $\pi$ of the set $\{1, \ldots, n\}$ is a collection of disjoint subsets $F_{1}, \ldots,  F_{\ell}$ of $\{1, \ldots, n\}$, whose union is the entire set. In particular, we always have that $1\leq \ell \leq n$ and we sometimes use the notation $|\pi|$ to indicate $\ell$, that is, the amount of disjoint subsets in the partition $\pi$. Given a probability $P$ in $\mathcal{M}(X_{i})$  we define
 $$
 P_{\pi}:= \bigtimes_{i=1}^{\ell}P_{F_{i}}
 $$
 where $P_{F_{i}}$ is the marginal probability in $ \mathds{X}_{F_{i}}$.

 A probability is called decomposable if there exists a  partition $\pi$  with $|\pi| \geq 2$ for which $P = P_{\pi}$. When $n=2$, a probability is decomposable if and only if $P = P_{1}\times P_{2} $, and  when $n=3$   a probability is decomposable $P$ when 
 $$
 P_{123} - (P_{12}\times P_{3}) -  (P_{13}\times P_{2})- (P_{23}\times P_{1})+  2(P_{1}\times P_{2}\times P_{3}) \text{ is the zero measure}.
 $$ 
 but the converse is not true, as can be seen in Appendix C of \cite{NIPS2013_076a0c97}.

 When $n\geq 4$, a sufficient condition  for when $P$ is decomposable similar to the one of 3 variables  gets more complicated, and the characterization was done in  Proposition $2$ \cite{streitberg1990lancaster} and  is the following:

 \begin{theorem}\label{streit}The collection of real numbers $a_{\pi}:= (-1)^{|\pi|-1}(|\pi|-1)!$, indexed over the partitions of the set $\{1, \ldots, n\}$, is the only one that satisfies the following conditions
 	\begin{enumerate}
 		\item[i)] $a_{\{1,\ldots, n\}}:=1$
 		\item[ii)] For any decomposable probability $P$ defined in the Cartesian product  $\mathds{X}_{n}$ the measure
 		$$
 		\Sigma[P] := \sum_{\pi} a_{\pi} P_{\pi}
 		$$
 		is the zero measure.
 		\item[iii)] The operator $\Sigma$ is invariant if we reverse the order of the sets $X_{i}$, $1\leq i \leq n$. More precisely, if $\sigma: \{1,\ldots, n\} \to \{1,\ldots, n\} $ is a bijection then
 		$$
 		\Sigma[P^{\sigma}]= [\Sigma[P]]^{\sigma}
 		$$
 		where for a  measure $\mu$ in $\mathds{X}_{n}$  the measure $\mu^{\sigma}$ is defined in $\prod_{i=1}^{n}X_{\sigma(i)}$ by $\mu^{\sigma}(\prod_{i=1}^{n}A_{\sigma(i)} ):=\mu (\prod_{i=1}^{n}A_{i} ) $, for measurable sets $A_{j}$ of $X_{j}$.
 	\end{enumerate}
 	
 \end{theorem}

 The measure $\Sigma[P]$ is  called the Streitberg interaction of the probability $P$. It is important to  emphasize that $\Sigma[P]$ can be the zero measure for a non decomposable probability.  It can be proved that the amount of  partitions in a set with $n$ elements is the Bell number $B_{n}$, see Section $26.7$ in \cite{NIST:DLMF}, 
 which are defined as $B_{0}:=1$ and with the recurrence relation
 $$
 B_{n+1}= \sum_{j=0}^{n}\binom{n}{j}B_{j}.
 $$

 Another measure of interaction is the Lancaster, see Chapter $XII$ page $255$  in \cite{lancaster1969chi}

 \begin{equation}
 	\Lambda[P]:=	 \sum_{|F|=0}^{n} (-1)^{n-|F|} \left ( P_{F} \times \left [ \bigtimes_{j \in F^{c}}P_{j}\right ]\right ),
 \end{equation}
 
 which satisfies that $\Lambda[P]=0$ if $P$ is the product of its marginals, but similarly to the Streitberg interaction, the converse does not hold. This interaction measure inspired  the definition and the previous result about the Streitberg interaction, and as Theorem \ref{streit} states, there are decomposable probabilities for which its Lancaster interaction are nonzero.

 \begin{lemma}\label{exm02xn} For a probability $P$ in   $\mathcal{M}(\mathds{X}_{n})$, both the Streitberg interaction $\Sigma[P]$ and the Lancaster interaction $\Lambda[P]$ are elements of  $\mathcal{M}_{n}(\mathds{X}_{n})$. Further, for measures   $\mu_{i}$  in $\mathcal{M}(\mathds{X}_{n})$  such that $\mu_{i}(X_{i})=0$, $1\leq i \leq n$,  there exists an $M \geq 0$ and a probability $P$ in    $\mathcal{M}(\mathds{X}_{n})$  for which $\Sigma[P] = \Lambda[P]= M(-1)^{n}(\bigtimes_{i=1}^{n}\mu_{i}) $.
 \end{lemma}

 \begin{proof}Note that the case $n=2$ is proved in Lemma \ref{exm01xn}, because  for any probability $P$ we have that $\Sigma[P] = \Lambda[P]= P - P_{1}\times P_{2}$ and the restriction on the measures $\mu_{i}$ is the same.\\
 	Let a probability $P$ in   $\mathcal{M}(\mathds{X}_{n})$  and consider its Streitberg interaction $\Sigma[P]$.    Choose an arbitrary $z \in X_{n}$ and define  the  probability $Q$ in $\mathds{X}_{n}$ by 
 	$$
 	Q(\prod_{i=1}^{n}A_{i}):=P(\left [\prod_{i=1}^{n-1}A_{i}\right ]\times X_{n}) \delta_{z}(A_{n})= (P_{\{1, \ldots, n-1\}}\times \delta_{z})(\prod_{i=1}^{n}A_{i}). 
 	$$ 
 	In particular, $Q$ is decomposable and then by Theorem \ref{streit} $\Sigma(Q)$ is the zero measure. However, for any partition $\pi$ of $\{1, \ldots, n\}$ 
 	$$
 	Q_{\pi }(\left [\prod_{i=1}^{n-1}A_{i}\right ]\times X_{n} )=P_{\pi }(\left [\prod_{i=1}^{n-1}A_{i}\right ]\times X_{n} )
 	$$
 	thus concluding that 	$\Sigma[P]([\prod_{i=1}^{n-1}A_{i}]\times X_{n} )=\Sigma[Q]([\prod_{i=1}^{n-1}A_{i}]\times X_{n} )=0$. Due to the property $iii)$ in Theorem \ref{streit}, we conclude that $\Sigma[P] $ is an element of $\mathcal{M}_{n}(\mathds{X}_{n})$.\\
 	Now, for  the Lancaster interaction of the probability  $P$, since  either $n \in F$ or $n \in F^{c}$, $\Lambda[P]$ can be written as 
 	\begin{align*}
 		\Lambda[P]&=\sum_{|F|=0}^{n} (-1)^{n-|F|} \left ( P_{F} \times \left [ \bigtimes_{j \in F^{c}}P_{j} \right ] \right )\\
 		&=\sum_{|G|=0}^{n-1} (-1)^{n-|G|-1} \left (P_{G\cup\{n\}} \times \left [ \bigtimes_{j \in G^{c}}P_{j} \right ] \right ) + \sum_{|G|=0}^{n-1} (-1)^{n-|G|} \left ( P_{G} \times  \left [ \bigtimes_{j \in G^{c}\cup\{n\} }P_{j} \right ]\right ),
 	\end{align*} 
 	where  $G$ and its complement $G^{c}$ are always being considered as a subset of $\{1, \ldots, n-1\}$.  Thus, if $Q([\prod_{i=1}^{n-1}A_{i}]):= P([\prod_{i=1}^{n-1}A_{i}]\times X_{n} ) $
 	\begin{align*}
 		\Lambda[P]&(\left [\prod_{i=1}^{n-1}A_{i} \right ]\times X_{n} )\\
 		=&\sum_{|G|=0}^{n-1} (-1)^{n-|G|-1} \left(  Q_{G}\times \left  [ \bigtimes_{j \in G^{c}}Q_{j} \right ] \right )(\prod_{i=1}^{n-1}A_{i}) \\
 		& \quad + \sum_{|G|=0}^{n-1} (-1)^{n-|G|} \left(    Q_{G} \times \left [ \bigtimes_{j \in G^{c} }Q_{j} \right ] \right)(\prod_{i=1}^{n-1}A_{i})=0.
 	\end{align*}  
 	Now, regarding the product of the measures $\mu_{i}$, we  may assume that all of them are nonzero, because otherwise we take $M=0$.  \\
 	If $\mu_{i}= a_{i}^{1}P_{i}^{1}  - a_{i}^{2}P_{i}^{2} $ is a  Hahn-Jordan decomposition, where $a_{i}^{1}, a_{i}^{2}$ are nonnegative and $P_{i}^{1}, P_{i}^{2}$ are probabilities in $X_{i}$,  then $a_{i}^{1}=a_{i}^{2} >0$, and we write $a_{i}$ for this value.  In particular, the measure 
 	$$
 	\bigtimes_{i=1}^{n}\mu_{i} = C(\bigtimes_{i=1}^{n}[P_{i}^{1} - P_{i}^{2}])= C\sum_{\alpha\in \mathbb{N}_{2}^{n}}(-1)^{n- |\alpha|}P_{\alpha},
 	$$ 
 	where $C:= \prod_{i=1}^{n}a_{i}$ and $P_{\alpha }:= \bigtimes_{i=1}^{n}P_{i}^{\alpha(i)}$. Define the probability
 	$$
 	P:= \frac{1}{2^{n-1}}\sum_{|\alpha| \in 2\mathbb{N}}P_{\alpha},
 	$$
 	and note that the  marginal in the variables $\{1,\ldots, n-1\}$ satisfies
 	$$
 	P(\left [\prod_{i=1}^{n-1}A_{i}\right ]\times X_{n})= \frac{1}{2^{n-1}}\sum_{\alpha^{\prime} \in \mathbb{N}_{2}^{n-1}}P_{\alpha^{\prime}}(\prod_{i=1}^{n-1}A_{i})= \bigtimes_{i=1}^{n-1}\left [ \frac{P_{i}^{1} + P_{i}^{2}}{2}   \right ](A_{i}).  
 	$$
 	A similar relation occurs if we compute the  marginal on any set of size $n-1$. More generally, from these relations we directly obtain that  for any $F\subset \{1, \ldots, n\}$, with $1\leq |F|\leq n-1$ it holds that the  marginal in the variables $F$ satisfies 
 	\begin{equation}\label{exm02xneq1}
 		P_{F}= \bigtimes_{i \in F}\left [ \frac{P_{i}^{1} + P_{i}^{2}}{2}   \right ],  
 	\end{equation}
 	and from this, we obtain that at the exception of the full partition $|\pi |=1$, for any other partition $\pi $ of  $\{1, \ldots, n \}$ we have that $P_{\pi}= \bigtimes_{i=1}^{n}\left [ (P_{i}^{1} + P_{i}^{2})/2   \right ]$. Then
 	
 	$$
 	\Sigma[P]= P + \sum_{|\pi|\geq 2} (-1)^{|\pi|-1}(|\pi|-1)!\left (  \bigtimes_{i=1}^{n}\left [ \frac{P_{i}^{1} + P_{i}^{2}}{2}   \right ]  \right )= P-   \bigtimes_{i=1}^{n}\left [ \frac{P_{i}^{1} + P_{i}^{2}}{2}   \right ],
 	$$
 	$$
 	\Lambda[P]= P + \sum_{|F|=0}^{n-1} (-1)^{n-|F|} \left (  \bigtimes_{i=1}^{n}\left [ \frac{P_{i}^{1} + P_{i}^{2}}{2}   \right ]  \right )= P-  \bigtimes_{i=1}^{n}\left [ \frac{P_{i}^{1} + P_{i}^{2}}{2}   \right ],
 	$$
 	because both $\Sigma[\bigtimes_{i=1}^{n} (P_{i}^{1} + P_{i}^{2})/2   ]$ and $\Lambda[\bigtimes_{i=1}^{n}  (P_{i}^{1} + P_{i}^{2})/2   ]$ are the zero measure. To conclude, note that
 	
 	$$
 	P- \bigtimes_{i=1}^{n}\left [ \frac{P_{i}^{1} + P_{i}^{2}}{2}   \right ]= \frac{1}{2^{n }}\sum_{\alpha \in \mathbb{N}_{2}^{n}}(-1)^{|\alpha|}P_{\alpha}= \frac{(-1)^{n}}{2^{n }C}(\bigtimes_{i=1}^{n} \mu_{i}).
 	$$\end{proof}

 Lemma \ref{exm02xn} is used to provide a characterization for the radial PDI$_{n}^{\infty}$ kernels in Theorem \ref{basicradialndim}.

 Based on Lemma \ref{measorderk} and the properties in Lemma \ref{exm01xn} and Lemma \ref{exm02xn}, for discrete probabilities $P,Q \in \mathcal{M}(\mathds{X}_{n})$ we propose  the following generalization of the Lancaster interaction	
 \begin{equation}
 	\Lambda_{k}^{n}[P, Q]:=P+  \sum_{j=0}^{k-1}(-1)^{k-j}\binom{n-j-1}{n-k}\sum_{|F|=j}P_{F}\times Q_{F^{c}}.
 \end{equation}  
 
 When $Q =  \bigtimes_{i=1}^{n}P_{i}$ we simply write $	\Lambda_{k}^{n}[P]$.  Note that
 $$
 \Lambda_{n}^{n}[P]:=P+  \sum_{j=0}^{n-1}(-1)^{n-j}\sum_{|F|=j}P_{F}\times [\bigtimes_{i \in F^{c}}P_{i} ] = \Lambda[P],
 $$	
 also, $\Lambda_{k}^{n}[\delta_{x_{\vec{1}}}, \delta_{x_{\vec{2}}}]=\mu_{k}^{n}[ x_{\vec{1}}, x_{\vec{2}}] $ and
 $$
 \Lambda_{2}^{n}[P]:=P+  \sum_{j=0}^{1}(-1)^{2-j}\sum_{|F|=j}P_{F}\times [\bigtimes_{i \in F^{c}}P_{i} ] = P - \bigtimes_{i=1}^{n}P_{i}.
 $$

 \begin{theorem}\label{generallancaster} The generalized Lancaster interaction $\Lambda_{k}^{n}$ satisfies the following properties:
 	\begin{enumerate}
 		\item [$i)$] For discrete probabilities $P, Q$ in   $\mathcal{M}(\mathds{X}_{n})$,  the generalized Lancaster interaction $\Lambda_{k}^{n}[P,Q] \in \mathcal{M}_{k}(\mathds{X}_{n})$. 
 		\item [$ii)$] If for some $1\leq k \leq n-1$ we have that  $\Lambda_{k}^{n}[P]=0$ then $\Lambda_{k+1}^{n}[P]=0$.
 		\item [$iii)$] $\Lambda_{n}^{n}[P]$ is multiplicative, in the sense that if $P= P_{\pi}$ for some partition $\pi= F_{1}, \ldots, F_{\ell}$ of $\{1, \ldots, n\}$ then
 		$$
 		\Lambda_{n}^{n}[P]= \prod_{i=1}^{\ell}\Lambda_{|F_{i}|}^{|F_{i}|}[P_{F_{i}}].
 		$$	 
 	\end{enumerate}
 \end{theorem}
 
 \begin{proof} Similar to the proof of Lemma \ref{measorderk}, for an  $F \subset \{1, \ldots, n\}$ with  $|F|= j$ and $1\leq j \leq n-1$,  either $F$ does not contain the element $n$ or
 	it contains and has $j-1$ elements in $\{1, \ldots, n-1\}$. From the  bijection between those sets $F$ that contains the element $n $ and the subsets of  $\{1, \ldots, n-1\}$ that has $j-1$ terms, we obtain that 
 	\begin{equation}\label{generallancastereq1}
 		\left [ \sum_{|F|=j}P_{F}\times Q_{F^{c}} \right ](A \times X_{n})=\left [ \sum_{|L|=j-1}P_{L}\times Q_{L^{c}} \right ](A) + \left [ \sum_{|L|=j}P_{L}\times Q_{L^{c}} \right ](A),
 	\end{equation}
 	where $A \subset \mathds{X}_{n-1}$ and $L$ and its complement $L^{c}$ are subsets of  $\{1, \ldots, n-1\}$.
 	Thus, we obtain that
 	\begin{align*}
 		&\Lambda_{k}^{n}[P, Q](A \times X_{n})\\
 		&= P_{\{1, \ldots, n-1\}}+\sum_{j=0}^{k-1}(-1)^{k-j}\binom{n-j-1}{n -k} \left [ \sum_{|L|=j-1}P_{L}\times Q_{L^{c}}     +     \sum_{|L|=j}P_{L}\times Q_{L^{c}} \right ] (A)\\
 		&=P_{\{1, \ldots, n-1\}}+\sum_{j=0}^{k-1}(-1)^{k-j} \left [ \binom{n-j-1}{n -k} - \binom{n-j-2}{n -k} \right ]      \sum_{|L|=j}P_{L}\times Q_{L^{c}}  (A) \\
 		&=	\Lambda_{k}^{n-1}[P_{\{1, \ldots, n-1\}}, Q_{\{1, \ldots, n-1\}}](A).
 	\end{align*}
 	Applying this relation recursively $n-k $ times we obtain that the marginal 
 	$$
 	\left ( \Lambda_{k}^{n}[P, Q]\right  )_{\{1, \ldots, k\}}= \Lambda_{k}^{k}[P_{\{1, \ldots, k\}}, Q_{\{1, \ldots , k\}}].
 	$$
 	More generally, by the symmetry of the measures involved, we have that for any $G \subset \{1, \ldots, n\}$ with $|G|\geq k$
 	\begin{equation}\label{generallancastereq2}
 		\left ( \Lambda_{k}^{n}[P, Q]\right  )_{G}= \Lambda^{|G|}_{k}[P_{G}, Q_{G}].
 	\end{equation}
 	Thus, to conclude the first part, note that if we apply the induction argument on the case  $n=k$                                                                                       
 	\begin{align*}
 		\Lambda_{n}^{n}[P, Q](A \times X_{n})&= P_{\{1, \ldots, n-1\}} +\sum_{j=1}^{n-1}(-1)^{n-j} \left [ \sum_{|L|=j-1}P_{L}\times Q_{L^{c}}     +     \sum_{|L|=j}P_{L}\times Q_{L^{c}} \right ] (A)=0,
 	\end{align*}	
 	where $A \subset \mathds{X}_{n-1}$, similarly for the other cases, thus  $\Lambda_{k}^{n}[P, Q] \in \mathcal{M}_{k}(\mathds{X}_{n})$.\\
 	To prove relation $(ii)$, since $\Lambda_{k}^{n}[P]=0$, Equation \ref{generallancastereq2} implies that for any $G \subset \{1, \ldots, n\}$ with $|G|= k$ we have that $\Lambda_{k}^{k}[P_{G}]=0$, that is 
 	$$
 	P_{G}=- \sum_{j=0}^{k-1}(-1)^{k-j} \sum_{|H|=j, H \subset G}P_{H}\times \left [\bigtimes_{i \in G \setminus{H}}P_{i} \right ].
 	$$
 	In particular, by multiplying each previous equality by  $\bigtimes_{i \in G^{c}}P_{i} $, and summing over all possibles  $G \subset \{1, \ldots, n\}$ with $|G|= k$ we have that
 	\begin{align*}
 		\sum_{|G|=k}P_{G}\times \left [\bigtimes_{i \in G^{c}}P_{i} \right ]&=  - \sum_{|G|=k}\sum_{j=0}^{k-1}(-1)^{k-j} \sum_{|H|=j, H \subset G}P_{H}\times \left [\bigtimes_{i \in H^{c}}P_{i} \right ]\\
 		&=- \sum_{j=0}^{k-1}(-1)^{k-j} \binom{n-j}{n-k}\sum_{|F|=j}P_{F}\times \left [\bigtimes_{i \in F^{c}}P_{i} \right ],
 	\end{align*}
 	where the last equality occurs because for a fixed  $F \subset \{1, \ldots, n\}$ with $|F|\leq  k$, to analyse the term that multiplies $P_{F}\times \left [\bigtimes_{i \in F^{c}}P_{i} \right ]$, we need to compute how many sets $G \subset \{1, \ldots, n\}$ with $|G|= k$ contains the subset $F$, which is precisely the binomial term  in the equality.  To conclude, note that
 	\begin{align*}
 		\Lambda_{k+1}^{n}[P]&=P -   \sum_{|F|=k}P_{F}\times \left [\bigtimes_{i \in F^{c}}P_{i} \right ]+ \sum_{j=0}^{k-1}(-1)^{k+1-j}\binom{n-j-1}{n-k+1}\sum_{|F|=j}P_{F}\times \left [\bigtimes_{i \in F^{c}}P_{i} \right ]\\
 		&=P +  \sum_{j=0}^{k-1}(-1)^{k-j}\left [-\binom{n-j-1}{n-k+1}  + \binom{n-j}{n-k}\right ]\sum_{|F|=j}P_{F}\times \left [\bigtimes_{i \in F^{c}}P_{i}  \right ]\\
 		&	=\Lambda_{k}^{n}[P]=0.
 	\end{align*}
 	By recurrence, it is sufficient to prove relation $iii)$ on the case $\ell=2$. Suppose that $P= P_{L}\times P_{L^{c}}$ for some $L \subset \{1, \ldots, n\}$, then 
 	\begin{align*}
 		&\Lambda_{|L|}^{|L|}[P_{L}]\times \Lambda_{n-|L|}^{n-|L|}[P_{L^{c}}]\\
 		&=\sum_{|F|=0,  F \subset L}^{|L|} \sum_{|G|=0,  G \subset L^{c}}^{n-|L|}(-1)^{n-|F|- |G|}	P_{F\cap L}\times  P_{G\cap L^{c}} \times  \left [\bigtimes_{j \in F^{c}\cap L}P_{j} \right ] \times \left [\bigtimes_{j \in G^{c}\cap L^{c}}P_{j} \right ]\\
 		&=\sum_{|F|=0,  F \subset L}^{|L|} \sum_{|G|=0,  G \subset L^{c}}^{n-|L|}(-1)^{n-|F|- |G|}	P_{F\cup G } \times  \left [\bigtimes_{j \in (F\cup G)^{c}}P_{j} \right ]\\
 		&=\sum_{|H|=0}^{n} (-1)^{n-|H|}	P_{H } \times \left [\bigtimes_{j \in H^{c}}P_{j}  \right ]= \Lambda_{n}^{n}[P],
 	\end{align*}	
 	where the first equality in the last line occurs because   any subset $H \subset \{1, \ldots, n\}$	 can be uniquely written as $ F\cup G$, where $F \subset L$ and $G\cup L^{c}$.\end{proof}

 Property $i)$ in Theorem \ref{generallancaster} generalizes Lemma \ref{measorderk} while property $iii)$  explains the differences  between the Lancaster and the Streiberg interactions  (note that with it we can easily deduce that $\Lambda_{n}[P]=0$ when $P=P_{\pi}$ and one of the sets in the partition $\pi$ is a singleton). Property $ii)$ emphasizes the role that $\Lambda_{k}^{n}[P]$ is  an indexed measure of independence for $P$.
 
 We conclude this section by providing a generalization of the second part of Lemma \ref{exm01xn} and Lemma \ref{exm02xn}, which will be used in Theorem \ref{bernsksevndimpart3}.

 \begin{theorem}\label{genlancastercartesianproduct}For measures   $\mu_{i}  \in \mathcal{M}(X_{i})$, $1\leq i \leq n$,  with the restriction that  $|i, \quad \mu_{i}(X_{i})=0|\geq n+1-k$,  there exists an $M \geq 0$ and a probability $P$ in    $\mathcal{M}(\mathds{X}_{n})$  for which $\Lambda_{k}^{n}[P]= M(-1)^{n}(\bigtimes_{i=1}^{n}\mu_{i}) $.
 \end{theorem}
 \begin{proof} We  may assume that all of them are nonzero, because otherwise we take $M=0$.\\
 	Let $\mu_{i}  \in \mathcal{M}(X_{i})$, $1\leq i \leq n$,  with the restriction that  $|i, \quad \mu_{i}(X_{i})=0|\geq k$. For convenience, we assume that $\{i, \quad \mu_{i}(X_{i})=0\}= \{1, \ldots, k\}$, then similar to Lemma  \ref{exm02xn}, for $1\leq i \leq k$ a Jordan decomposition of $\mu_{i}$  can be written as  $\mu_{i}= b_{i}^{1}[S_{i}^{1}  - S_{i}^{2} ]$, where $b_{i}$ is positive  and $S_{i}^{1}, S_{i}^{2}$ are probabilities in $X_{i}$, and then 
 	$$
 	\bigtimes_{i=1}^{k}\mu_{i} = B(\bigtimes_{i=1}^{n}[S_{i}^{1} - S_{i}^{2}])= B\sum_{\alpha\in \mathbb{N}_{2}^{k}}(-1)^{k- |\alpha|}S_{\alpha}=  (-1)^{k}B\left [ \sum_{|\alpha| \in 2\mathbb{N}}S_{\alpha} - \sum_{|\alpha| \in 2\mathbb{N}+1}S_{\alpha}\right ], 
 	$$  
 	where $B= 1/ \prod_{i=1}^{k}b_{i}$.  Also, by Lemma  \ref{exm01xn}, for $i\geq k+1$,  if $\mu_{i}= c_{i}^{1}R_{i}^{1}  - c_{i}^{2}R_{i}^{2} $ is a  Hahn-Jordan decomposition, where $c_{i}^{1}, c_{i}^{2}$ are nonnegative and $R_{i}^{1}, R_{i}^{2}$ are probabilities in $X_{i}$,  we have that 
 	$$
 	\bigtimes_{i=k+1}^{n}\mu_{i} = \sum_{\beta\in \mathbb{N}_{2}^{n-k}}(-1)^{n-k- |\beta|}c_{\beta}R_{\beta} = (-1)^{n-k} \left [ \sum_{|\beta| \in 2\mathbb{N}}c_{\beta}R_{\beta} - \sum_{|\beta| \in 2\mathbb{N}+1}c_{\beta}R_{\beta}\right ] 
 	$$
 	where $c_{\beta} := \prod_{i=k+1}^{n}c_{i}^{\beta(i-k)}$ and $R_{\beta }:= \bigtimes_{i=1}^{n-k}R_{i}^{\beta(i-k)}$. Thus
 	\begin{align*}
 		&(-1)^{n}\bigtimes_{i=1}^{n}\mu_{i}=  B\left [ \sum_{|\alpha| \in 2\mathbb{N}}S_{\alpha} - \sum_{|\alpha| \in 2\mathbb{N}+1}S_{\alpha}\right ]  \times \left [ \sum_{|\beta| \in 2\mathbb{N}}c_{\beta}R_{\beta} - \sum_{|\beta|   \in 2\mathbb{N}+1}c_{\beta}R_{\beta}\right ]  \\  
 		&= B  \left (  \sum_{|\alpha| \in 2\mathbb{N}}S_{\alpha} \right ) \times \left ( \sum_{|\beta| \in 2\mathbb{N}}c_{\beta}R_{\beta} \right ) +   B \left ( \sum_{|\alpha| \in 2\mathbb{N}+1}S_{\alpha} \right ) \times \left (  \sum_{|\beta| \in 2\mathbb{N}+1}c_{\beta}R_{\beta} \right )\\
 		&  - B \left ( \sum_{|\alpha| \in 2\mathbb{N}}S_{\alpha} \right ) \times \left (  \sum_{|\beta| \in 2\mathbb{N}+1}c_{\beta}R_{\beta} \right )- B \left ( \sum_{|\alpha| \in 2\mathbb{N}+1}S_{\alpha}\right ) \times \left (  \sum_{|\beta| \in 2\mathbb{N}}c_{\beta}R_{\beta} \right ). 
 	\end{align*}
 	Define the probability
 	$$
 	P:= \frac{1}{D}\left [\left (  \sum_{|\alpha| \in 2\mathbb{N}}S_{\alpha} \right ) \times \left ( \sum_{|\beta| \in 2\mathbb{N}}c_{\beta}R_{\beta} \right ) + \left ( \sum_{|\alpha| \in 2\mathbb{N}+1}S_{\alpha} \right ) \times \left (  \sum_{|\beta| \in 2\mathbb{N}+1}c_{\beta}R_{\beta} \right ) \right ],
 	$$
 	where $D = 2^{k-1}\sum_{\beta\in \mathbb{N}_{2}^{n-k}}c_{\beta}$. Then, By the property of Equation \ref{exm02xneq1}, if $i \in \{1,\ldots,  k\}$
 	\begin{align*}
 		P_{\{1, \ldots n\}\setminus{i}}&= \frac{2^{k-1}}{D}  \left ( \bigtimes_{j=1, j \neq i }^{k} \frac{S_{j}^{1} + S_{j}^{2}}{2}   \right ) \times \left ( \sum_{|\beta| \in 2\mathbb{N}}c_{\beta}R_{\beta} \right ) \\
 		& \quad +  \frac{2^{k-1}}{D}\left ( \bigtimes_{j=1, j \neq i }^{k} \frac{S_{j}^{1} + S_{j}^{2}}{2}   \right ) \times \left (  \sum_{|\beta| \in 2\mathbb{N}+1}c_{\beta}R_{\beta} \right )  \\
 		&= \frac{2^{k-1}}{D}     \left ( \bigtimes_{j=1, j \neq i }^{k} \frac{S_{j}^{1} + S_{j}^{2}}{2}   \right ) \times \left ( \sum_{\beta\in \mathbb{N}_{2}^{n-k}}c_{\beta}R_{\beta} \right )  \\
 		&=  \left ( \bigtimes_{j=1, j \neq i }^{k} \frac{S_{j}^{1} + S_{j}^{2}}{2}   \right )\times \left ( \bigtimes_{j=k+1 }^{n} \frac{c_{j}^{1}R_{j}^{1} + c_{j}^{2}R_{j}^{2}}{c_{j}^{1} + c_{j}^{2}}   \right ). 
 	\end{align*}
 	To conclude, from this relation we obtain that 
 	$$
 	P_{i}= \frac{S_{i}^{1} + S_{i}^{2}}{2}, \text{ for $1\leq i \leq k$  and } P_{i}= \frac{c_{i}^{1}R_{i}^{1} + c_{i}^{2}R_{i}^{2}}{c_{i}^{1} + c_{i}^{2}}, \text{ for $i \geq k+1$},
 	$$
 	and from this  we obtain that  for any $F \subset \{1, \ldots, n\}$, with $|F|\leq k-1$ 
 	$$
 	P_{F}= \bigtimes_{i\in F }P_{i},  
 	$$ 
 	because there must exist a point in $\{1, \ldots, k\} \cap F^{c}$. Gathering all those relations we conclude that
 	\begin{align*}
 		\Lambda_{k}^{n}[P]&=P+  \sum_{j=0}^{k-1}(-1)^{k-j}\binom{n-j-1}{n-k}\sum_{|F|=j}P_{F}\times  \left [\bigtimes_{i \in F^{c}}P_{i} \right ]\\
 		&=P+ \sum_{j=0}^{k-1}(-1)^{k-j}\binom{n-j-1}{n-k}\sum_{|F|=j} \left [\bigtimes_{i=1}^{n} P_{i}  \right ]= P -\bigtimes_{i=1}^{n} P_{i}, 
 	\end{align*}
 	and
 	\begin{align*}
 		&\bigtimes_{j=1}^{n} P_{i}= \left(\frac{1}{2^{k}} \frac{1}{\prod_{j=k+1}^{n}(c_{j}^{1} +   c_{j}^{2})} \right) \left ( \sum_{\alpha \in \mathbb{N}_{2}^{k}} S_{\alpha}  \right ) \times \left (  \sum_{\beta \in \mathbb{N}_{2}^{n-k}} c_{\beta}R_{\beta}   \right )\\
 		&= \frac{1}{2D}\left [\left (  \sum_{|\alpha| \in 2\mathbb{N}}S_{\alpha} \right ) \times \left ( \sum_{|\beta| \in 2\mathbb{N}}c_{\beta}R_{\beta} \right ) + \left ( \sum_{|\alpha| \in 2\mathbb{N}+1}S_{\alpha} \right ) \times \left (  \sum_{|\beta| \in 2\mathbb{N}+1}c_{\beta}R_{\beta} \right ) \right ]\\
 		& + \frac{1}{2D}\left [\left (  \sum_{|\alpha| \in 2\mathbb{N}}S_{\alpha} \right ) \times \left ( \sum_{|\beta| \in 2\mathbb{N}+1}c_{\beta}R_{\beta} \right ) + \left ( \sum_{|\alpha| \in 2\mathbb{N} }S_{\alpha} \right ) \times \left (  \sum_{|\beta| \in 2\mathbb{N}+1}c_{\beta}R_{\beta} \right ) \right ] 
 	\end{align*}
 	thus
 	\begin{align*}
 		&\Lambda_{k}^{n}[P]= P -\bigtimes_{j=1}^{n} P_{i}\\
 		&= \frac{1}{2D}\left [\left (  \sum_{|\alpha| \in 2\mathbb{N}}S_{\alpha} \right ) \times \left ( \sum_{|\beta| \in 2\mathbb{N}}c_{\beta}R_{\beta} \right ) + \left ( \sum_{|\alpha| \in 2\mathbb{N}+1}S_{\alpha} \right ) \times \left (  \sum_{|\beta| \in 2\mathbb{N}+1}c_{\beta}R_{\beta} \right ) \right ]\\
 		&\quad - \frac{1}{2D}\left [\left (  \sum_{|\alpha| \in 2\mathbb{N}}S_{\alpha} \right ) \times \left ( \sum_{|\beta| \in 2\mathbb{N}+1}c_{\beta}R_{\beta} \right ) + \left ( \sum_{|\alpha| \in 2\mathbb{N} }S_{\alpha} \right ) \times \left (  \sum_{|\beta| \in 2\mathbb{N}+1}c_{\beta}R_{\beta} \right ) \right ]\\
 		&= \frac{1}{2D}  \left [\left (  \sum_{|\alpha| \in 2\mathbb{N}}S_{\alpha} - \sum_{|\alpha| \in 2\mathbb{N}+1}S_{\alpha}\right ) \times \left ( \sum_{|\beta| \in 2\mathbb{N}}c_{\beta}R_{\beta} - \sum_{|\beta| \in 2\mathbb{N}+1}c_{\beta}R_{\beta} \right ) \right ]\\
 		&=\frac{(-1)^{n}}{2BD}\bigtimes_{i=1}^{n}\mu_{i}.
 \end{align*}\end{proof}	
 
 \begin{remark} We expect that $\Lambda_{k}^{n}[P]$ also have a Lattice interpretation as the Streitberg and the Lancaster interaction  $\Lambda_{n}^{n}[P]$, see \cite{ip2004structural} , more specifically, with respect to the Lattice
 	$$
 	M_{k}(n):= \{ \text{$\pi$ is a partition of $\{1, \ldots, n\}$  and $\pi$ has at most $n-k$ non-singleton blocks} \}.
 	$$	
 	We also expect  that the Streitberg interaction have a similar generalization as $\Lambda_{k}^{n}[P]$. We obtained expressions for low values of $k$ and arbitrary $n$, but the combinatorial complexity of the  Streitberg interaction make it difficult to obtain closed formulas for the general case.  
 \end{remark}

 \begin{remark} Even though the results presented in this Section are stated for discrete probabilities, there is no need for such restriction. However, for extending these results it is convenient that either the product sigma algebra of $\mathds{X}_{n}$ or the space of probabilities have the property that
 	$$
 	P(\prod_{i=1}^{n}A_{i})= Q(\prod_{i=1}^{n}A_{i})
 	$$ 
 	for every measurable $A_{i} \subset X_{i}$ if and only if $P=Q$. For instance, this is the case for finite Radon measures, see Theorem $1.10$ page $24$  in \cite{Berg1984}.
 \end{remark}

 \section{Bernstein functions of order  $n$ in $n$ variables}\label{Bernsteinfunctionsofordern}
 
 As detailed in Section \ref{Terminology}, we use the notation $(\mathbb{R}^{d})_{n}$ as the $n-$Cartesian product of the Euclidean space $\mathbb{R}^{d}$, being then an $nd-$dimensional Euclidean space.

 \begin{definition}\label{defnpdikn} A   function $g:[0, \infty)^{n} \to \mathbb{R}$ is called   positive definite independent  of order $n$  in all Euclidean spaces (PDI$_{n}^{\infty}$)  if for every $d \in \mathbb{N}$  and $\mu \in \mathcal{M}_{n}( (\mathbb{R}^{d})_{n})$ it satisfies
 	$$
 	\int_{\mathbb{R}^{d} }	\int_{\mathbb{R}^{d} }(-1)^{n}g(\|u_{1} - v_{1}\|^{2}, \ldots,\|u_{n} - v_{n}\|^{2} )d\mu(u)d\mu(v) \geq 0.
 	$$	
 	If for every $d \in \mathbb{N}$  the previous inequality is an equality only when $\mu$ is the zero measure in $\mathcal{M}_{n}( (\mathbb{R}^{d})_{n})$, we say that $g$  is a strictly positive definite  independent function of order $n$  in all Euclidean spaces (SPDI$_{n}^{\infty}$).   
 \end{definition}

 The most important example of an  PDI$_{n}^{\infty}$ function is the fact that the   Kronecker product  of $n$ Bernstein functions in $[0, \infty)$ is PDI$_{n}^{\infty}$.  Indeed, let $g_{i}: [0, \infty)\to \mathbb{R}$, $1\leq i \leq n$, be nonzero  Bernstein functions in $[0, \infty)$ and consider its  Kronecker product
 $$
 (\times_{i=1}^{n}g_{i})(t_{1}, \ldots, t_{n}):= \prod_{i=1}^{n} g_{i}(t_{i}),
 $$
 
 Thus,  	by   Equation \ref{Kgamma} and Equation \ref{integmu0n}, by fixing  $w=(w_{1}, \ldots, w_{n})= \vec{0} \in (\mathbb{R}^{d})_{n} $   
 \begin{align*}
 	&\int_{(\mathbb{R}^{d})_{n} }	\int_{(\mathbb{R}^{d})_{n} }(-1)^{n}	(\times_{i=1}^{n}g_{i})(\|u_{1} - v_{1}\|^{2}, \ldots,\|u_{n} - v_{n}\|^{2} )d\mu(u)d\mu(v)\\
 	=& \int_{(\mathbb{R}^{d})_{n} }	\int_{(\mathbb{R}^{d})_{n} }(-1)^{n-1}	K^{\vec{0}}_{g_{n}}(u_{n}, v_{n})\left (\prod_{i=1}^{n-1} g_{i}(\|u_{i}- v_{i}\|^{2}) \right )d\mu(u)d\mu(v).       
 \end{align*}
 because the remaining $3$ terms in the definition of  $K^{\vec{0}}_{g_{n}}$  either do not depend on the $n$ variables of $u=(u_{1}, \ldots, u_{n})$ or the $n$ variables of $v=(v_{1}, \ldots, v_{n})$. Using this equality recursively we obtain that 
 $$
 \int_{(\mathbb{R}^{d})_{n} }	\int_{(\mathbb{R}^{d})_{n}}(-1)^{n} 	(\times_{i=1}^{n}g_{i})(u, v)d\mu(u)d\mu(v)= 	\int_{(\mathbb{R}^{d})_{n} }	\int_{(\mathbb{R}^{d})_{n} }	\prod_{i=1}^{n}	K^{\vec{0}}_{g_{i}}(u_{i}, v_{i})d\mu(u)d\mu(v) \geq 0
 $$	
 because the Kronecker product of PD kernels is an PD kernel as well. This equality is essentially a first  generalization  on the discrete case to several variables of Theorem $24$  in \cite{sejdinovic2013equivalence}, where it is proved the case $n=2$. This property also gives an explanation as to why the Distance Multivariance concept in \cite{Boettcher2019}  (more specifically, Theorem $3.4$)  works.

 If $g:[0, \infty)\to \mathbb{R}$ is an Bernstein function, then for every $\mu \in \mathcal{M}_{1}((\mathbb{R}^{d})_{1})$ (that is, $\mu \in \mathcal{M} (\mathbb{R}^{d})$ and $\mu(\mathbb{R}^{d})=0$) 
 $$
 \int_{\mathbb{R}^{d}}\int_{\mathbb{R}^{d}}g(\|u-v\|^{2})d\mu(u)d\mu(v) = \int_{\mathbb{R}^{d}}\int_{\mathbb{R}^{d}}G(\|u-v\|^{2})d\mu(u)d\mu(v), 
 $$
 where $G(t)= g(t)-g(0)$. Hence, we may suppose that $g$ is zero in the border of $[0, \infty)$, that is, $g(0)=0$. Next, we generalize this property for PDI$_{n}^{\infty}$ kernels. For that, we define the set 
 $$
 \partial_{n-1}^{n}:= \{t=(t_{1}, \ldots, t_{n}) \in [0, \infty)^{n}, \quad  t_{i}=0 \text{ for some } i   \}.
 $$

 \begin{lemma}\label{PDInsimpli}	Let $g: [0, \infty)^{n} \to \mathbb{R}$ be a  continuous  function and define
 	\begin{align*}
 		G(t_{\vec{1}}):= \sum_{\alpha \in (\mathbb{N}_{1}^{0})^{n}}(-1)^{n - |\alpha| }g(t_{\alpha}), \quad  t_{\vec{0}}, t_{\vec{1}} \in [0,\infty)^{n}  \text{ and } t_{\vec{0}}:= \vec{0}. 
 	\end{align*} 
 	Then, for any  $\mu \in \mathcal{M}_{n}(( \mathbb{R}^{d})_{n})$ 
 	\begin{align*} 
 		\int_{(\mathbb{R}^{d})_{n}}\int_{(\mathbb{R}^{d})_{n}}&(-1)^{n}g(\|u_{1} - v_{1}\|^{2}, \ldots,\|u_{n} - v_{n}\|^{2} )d\mu(u)d\mu(v)\\
 		&=\int_{(\mathbb{R}^{d})_{n}}\int_{(\mathbb{R}^{d})_{n}}(-1)^{n}G(\|u_{1} - v_{1}\|^{2}, \ldots,\|u_{n} - v_{n}\|^{2} )d\mu(u)d\mu(v),
 	\end{align*} 
 	hence, $g$ is PDI$_{n}^{\infty}$ if and only if  $G$ is PDI$_{n}^{\infty}$.\\
 	If  $g(t)=0$ for every $t \in \partial_{n-1}^{n}$ then $g=G$.\\
 	If $t \in \partial_{n-1}^{n}$ then $G(t)=0$.
 \end{lemma}

 \begin{proof}Note that at the exception of the term related to $\alpha =\vec{1}$ in the definition of $G$, the other terms   depend on a maximum of  $n-1$ among the $n$ variables of  $t_{\vec{1}}$. As a consequence of  Equation \ref{integmu0n}, for every $\mu \in \mathcal{M}_{n}( (\mathbb{R}^{d})_{n})$  we obtain the equality
 	\begin{align*} 
 		\int_{(\mathbb{R}^{d})_{n}}\int_{(\mathbb{R}^{d})_{n}}&(-1)^{n}g(\|u_{1} - v_{1}\|^{2}, \ldots,\|u_{n} - v_{n}\|^{2} )d\mu(u)d\mu(v)\\
 		&=\int_{(\mathbb{R}^{d})_{n}}\int_{(\mathbb{R}^{d})_{n}}(-1)^{n}G(\|u_{1} - v_{1}\|^{2}, \ldots,\|u_{n} - v_{n}\|^{2} )d\mu(u)d\mu(v).
 	\end{align*} 
 	If  $g(t)=0$ for every $t \in \partial_{n-1}^{n}$ then $g=G$ by the way we defined $G$. For the remaining property, note   that we may rewrite $G$ as
 	$$
 	G(t_{1}, \ldots, t_{n})= \int_{[0, \infty)^{n}}g(s)d\left (\bigtimes_{i=1}^{n}[\delta_{t_{i}} - \delta_{\vec{0}}]\right )(s),
 	$$
 	thus, if at least $1$ coordinate of  $t$  is  zero, the above integral is related to the zero measure, consequently  $G(t)=0$ for every $t \in \partial^{n}_{n-1}$.\end{proof}	
 
 A convenient property that a function $g$  that is  PDI$_{n}^{\infty}$ and that is zero at 	$\partial_{n-1}^{n}$ has is the fact that it is nonnegative. Indeed, for an arbitrary $t \in [0, \infty)^{n}$, choose $d=1$ and $\mu= \times_{i=1}^{n} (\delta_{t_{i}} - \delta_{0})$, then 
 $$
 0 \leq \int_{\mathbb{R}^{d} }	\int_{\mathbb{R}^{d} }(-1)^{n}g(\|u_{1} - v_{1}\|^{2}, \ldots,\|u_{n} - v_{n}\|^{2} )d\mu(u)d\mu(v)=2^{n}g(t).
 $$

 Now, or objective is to characterize the PDI$_{n}^{\infty}$ functions.  First, we review and generalize a few results concerning completely monotone functions of several variables.

 \begin{definition}
 	A function $h:(0, \infty)^{n}\to \mathbb{R}$ is   completely monotone with $n$ variables  if $h \in C^{\infty}((0, \infty)^{n})$ and $(-1)^{|\alpha|}\partial^{\alpha}h(t) \geq 0$, for every $\alpha \in \mathbb{Z}_{+}^{n}$ and $t\in (0,\infty)^{n}$. 
 \end{definition}	
 
 Similar to the Hausdorff-Bernstein-Widder Theorem on completely monotone functions (one variable), the following equivalence holds, Section $4.2$ in \cite{Bochner2005}:
 
 \begin{theorem}\label{Bochnercomplsevndim} A  function $g:(0, \infty)^{n} \to \mathbb{R}$ is  completely monotone  with $n$ variables if and only if it can be represented as
 	$$
 	h(t)=\int_{[0,\infty)^{n}}e^{- r   \cdot  t  }d\eta(r), 
 	$$
 	where $\eta$ is a Borel nonnegative measure (possibly unbounded) on $[0,\infty)^{n}$. Further, the representation is unique.
 \end{theorem}

 \begin{definition}
 	A function $g:(0, \infty)^{n} \to \mathbb{R}$ is  called a  Bernstein function  of order $n$ in $(0, \infty)^{n}$ if $g \in C^{\infty}((0, \infty)^{n})$ and $\partial^{\vec{1}}g(t)$ is a completely monotone function  with $n$ variables, where  $\vec{1}=(1,1,\ldots, 1) \in \mathbb{N}^{n}$. 
 \end{definition}

 On the next Theorem we provide a representation  for some of those functions.
 
 \begin{theorem}\label{bernssevndim} A continuous function $g:[0, \infty)^{n} \to \mathbb{R}$ that satisfies $g(t)=0$ for every $t \in \partial_{n-1}^{n}$ is  a Bernstein function  of order $n$ in $(0, \infty)^{n}$ if and only if it can be represented as
 	\begin{align*}
 		g(t)=  \int_{[0,\infty)^{n}}\prod_{i=1}^{n}(1-e^{-r_{i}t_{i}})\frac{1 +r_{i}}{r_{i}}d\eta(r),
 	\end{align*}
 	where  the measure $\eta \in \mathfrak{M}([0,\infty)^{n})$ is nonnegative. Further, the representation is unique.
 \end{theorem}
 
 \begin{proof}Suppose that the  function $g$ admits the integral representation, then 
 	\begin{align*}
 		[\partial^{\vec{1}}g](t_{1}, \ldots, t_{n})=& \int_{[0,\infty)^{n}}\prod_{i=1}^{n}\left [e^{-r_{i}t_{i}}(1+r_{i})\right ]d\eta(r) = \int_{[0,\infty)^{n}}e^{-r\cdot t} \left [\prod_{i=1}^{n} (1+r_{i}) \right ]d\eta(r),
 	\end{align*}
 	which is clearly completely monotone by Theorem \ref{Bochnercomplsevndim}. The uniqueness of the integral representation for $g$ follows by the fact that completely monotone functions are also uniquely representable.\\
 	For the converse, since  $\partial^{\vec{1}}g$ is a completely monotone function,  Theorem \ref{Bochnercomplsevndim} implies that 
 	\begin{align*}
 		[\partial^{\vec{1}}g](t_{1}, \ldots , t_{n}) &= \int_{[0, \infty)^{n}} e^{-r\cdot t}d\sigma(r).
 	\end{align*}
 	By the fundamental Theorem of calculus in $n$ variables, if  $t^{1}_{1},\ldots, t^{1}_{n}, t^{2}_{1}, \ldots,  t^{2}_{n} >0$, with $t_{i}^{2} \geq t_{i}^{1}$, then
 	$$
 	\int_{\prod_{i=1}^{n}[t_{i}^{1},t^{2}_{i}]}[\partial^{\vec{1}}g](s_{1}, \ldots, s_{n})d(s_{1}, \ldots, s_{n})=\sum_{\alpha  \in \mathbb{N}_{2}^{n}} (-1)^{|\alpha|} g(t_{\alpha}).  
 	$$
 	By letting all $t_{i}^{1}\to 0$ and using the fact that $g$ is continuous on $[0, \infty)^{n}$, $\partial^{\vec{1}}g$ is a nonnegative function  and that $g$ is zero in $\partial_{n-1}^{n}$, we get that
 	$$
 	g(t_{1}^{2}, \ldots , t_{n}^{2})= \int_{\prod_{i=1}^{n}[0,t_{i}^{2}]}[\partial^{\vec{1}}g](s_{1}, \ldots, s_{n})d(s_{1}, \ldots, s_{n}).
 	$$
 	But 
 	\begin{align*}
 		\int_{\prod_{i=1}^{n}[0,t_{i}^{2}]}[\partial^{\vec{1}}g](s)ds&= \int_{\prod_{i=1}^{n}[0,t_{i}^{2}]}\left [\int_{[0, \infty)^{n}} e^{-r\cdot s}d\sigma(r)\right ]ds\\
 		&=\int_{[0, \infty)^{n}}\left [\int_{\prod_{i=1}^{n}[0,t_{i}^{2}]} e^{-r\cdot s}ds\right ]d\sigma(r)\\
 		&=\int_{[0, \infty)^{n}}\left [\prod_{i=1}^{n}\frac{1-e^{-r_{i}t_{i}^{2}}}{r_{i}}\right ]d\sigma(r).
 	\end{align*}
 	Similar to the proof of Theorem $3.2$ in \cite{Schilling2012}, the measure $d\eta(r):=(\prod_{i=1}^{n}1/(1+r_{i}))d\sigma(r)
 	$ is finite because by  Equation \ref{bern1ineq}
 	$$
 	\int_{[0, \infty)^{n}} \prod_{i=1}^{n}\frac{1}{1+r_{i}}d\sigma(r) \leq \int_{[0, \infty)^{n}} \prod_{i=1}^{n}\frac{1-e^{-r_{i}}}{r_{i}}d\sigma(r)=g(1,\ldots, 1), 
 	$$
 	w hich conludes the proof of the representation. \end{proof}

 Next results provides the characterization of all  PDI$_{n}^{\infty}$ functions, being an extension of the classical Schoenberg result of Theorem  \ref{reprcondneg}.

 \begin{theorem}\label{basicradialndim}
 	Let $g: [0, \infty)^{n} \to \mathbb{R}$ be a  continuous function such that $g(t)=0$ for every $t \in \partial_{n-1}^{n}$. The following conditions are equivalent:
 	\begin{enumerate}
 		\item [$(i)$] For any $d\in \mathbb{N}$ and discrete  measures  $\mu_{i}$ in $\mathbb{R}^{d}$ such that $\mu_{i}(\mathbb{R}^{d})=0$, $1\leq i \leq n$,  it holds that
 		$$
 		\int_{(\mathbb{R}^{d})_{n}}\int_{ (\mathbb{R}^{d})_{n}}(-1)^{n}g(\|x_{1}-y_{1}\|^{2}, \ldots, \|x_{n} - y_{n}\|^{2})d[\bigtimes_{i=1}^{n}\mu_{i}](x)d[ \bigtimes_{i=1}^{n}\mu_{i}](y)\geq 0.
 		$$
 		\item [$(ii)$] For any $d\in \mathbb{N}$ and discrete  probability  $P$ in $(\mathbb{R}^{d})_{n}$   its Lancaster interaction $\Lambda[P]$ satisfies that
 		$$
 		\int_{(\mathbb{R}^{d})_{n}}\int_{ (\mathbb{R}^{d})_{n}}(-1)^{n}g(\|x_{1}-y_{1}\|^{2}, \ldots, \|x_{n} - y_{n}\|^{2})d[\Lambda[P]](x)d[\Lambda[P]](y)\geq 0.
 		$$
 		\item [$(ii^{\prime})$] For any $d\in \mathbb{N}$ and discrete  probability  $P$ in $(\mathbb{R}^{d})_{n}$   its Streitberg interaction $\Sigma[P]$ satisfies that
 		$$
 		\int_{(\mathbb{R}^{d})_{n}}\int_{ (\mathbb{R}^{d})_{n}}(-1)^{n}g(\|x_{1}-y_{1}\|^{2}, \ldots, \|x_{n} - y_{n}\|^{2})d[\Sigma[P]](x)d[\Sigma (P)](y)\geq 0.
 		$$
 		\item [$(iii)$] The  function $g$ is  PDI$_{n}^{\infty}$.
 		\item [$(iv)$] The function $g$ can be represented as
 		\begin{align*}
 			g(t)=\int_{[0,\infty)^{n}}\prod_{i=1}^{n}(1-e^{-r_{i}t_{i}})\frac{1 +r_{i}}{r_{i}}d\eta(r)
 		\end{align*}
 		where  the measure $\eta \in \mathfrak{M}([0,\infty)^{n})$ is nonnegative. The representation is unique.
 		\item [$(v)$] The function $g $  is a Bernstein function   of  order $n$.
 	\end{enumerate}
 \end{theorem}

 \begin{proof} The equivalence between $(iv)$ and $(v)$ is proved in Theorem \ref{bernssevndim}.\\ 
 	Relation $(iv)$ implies relation $(iii)$ because for any $\mu \in \mathcal{M}_{n}((\mathbb{R}^{d})_{n})$
 	\begin{align*}
 		\int_{(\mathbb{R}^{d})_{n}}&\int_{(\mathbb{R}^{d})_{n}} (-1)^{n}\left [\int_{[0,\infty)^{n}}\prod_{i=1}^{n}(1-e^{-r_{i}\|x_{i} - y_{i}\|^{2}})\frac{1 +r_{i}}{r_{i}}d\eta(r) \right ] d\mu(x)d\mu(y)\\
 		&=\int_{[0,\infty)^{n}} \left [\int_{(\mathbb{R}^{d})_{n}}\int_{(\mathbb{R}^{d})_{n}}(-1)^{n} \prod_{i=1}^{n}(1-e^{-r_{i}\|x_{i} - y_{i}\|^{2}})\frac{1 +r_{i}}{r_{i}} d\mu(x)d\mu(y)  \right ]d\eta(r)\geq 0 \\
 	\end{align*}
 	as the  Kronecker product of $n$ Bernstein functions is PDI$_{n}^{\infty}$.\\
 	Relation $(iii)$ implies relations $(ii)$ and $(ii^{\prime})$ because by the first assertion in Lemma \ref{exm02xn}, both $\Sigma[P]$ and $\Lambda[P]$ are elements of $\mathcal{M}_{n}((\mathbb{R}^{d})_{n})$. \\
 	Both relations $(ii)$ and $(ii^{\prime})$  implies relation $(i)$ by the second assertion in Lemma \ref{exm02xn}.\\ 
 	To conclude, we prove that relation $(i)$ implies relation $(iv)$, and the proof is done by induction in $n$. When $n=1$, relations $(i)$, $(ii)$, $(ii^{\prime})$ and $(iii)$ are the same as relation $(i)$  in Theorem \ref{reprcondneg}, the same occurs for relation $(iv)$ and $(v)$ which are respectively relation $(ii)$ and $(iii)$  in Theorem \ref{reprcondneg}.\\
 	Now, we assume that the $6$ equivalences in Theorem \ref{basicradialndim} holds for an arbitrary $n\in \mathbb{N}$ and we prove for $n+1$ variables.\\
 	For any $d\in \mathbb{N}$ and  $s \in [0, \infty)$ it holds that the function
 	$$
 	t \in [0, \infty)^{n} \to g(t,s)
 	$$
 	is continuous, is zero when $t \in \partial_{n-1}^{n}$ and satisfies relation $(i)$ of Theorem \ref{basicradialndim} with $n$ variables,  by taking $\mu_{n+1}:= (\delta_{x_{n+1}} - \delta_{\vec{0}})/2$, where $\|x_{n+1}\|^{2}= s$. Thus, there exists an unique  nonnegative measure  $\eta_{s} \in \mathfrak{M}([0, \infty)^{n})$ for which
 	$$
 	g(t,s)= \int_{[0,\infty)^{n}}\prod_{i=1}^{n}(1-e^{-r_{i}t_{i}})\frac{1 +r_{i}}{r_{i}}d\eta_{s}(r). 
 	$$
 	We affirm that  the function 
 	$$
 	s \in [0, \infty) \to \eta_{s}(A) \in \mathbb{R}
 	$$ 
 	is continuous and  CND for every $d \in \mathbb{N}$ and   every $A \in \mathscr{B}([0, \infty)^{n})$.\\
 	To prove that the  radial kernel is CND in all Euclidean spaces, let $d \in \mathbb{N}$,  $u_{1}, \ldots , u_{l} \in \mathbb{R}^{d}$ and scalars $c_{1}, \ldots , c_{l} \in \mathbb{R} $ with the restriction that $\sum_{\theta=1}^{l}c_{\theta}=0$, by  choosing $\lambda :=\sum_{\theta =1}^{l}c_{\theta}\delta_{u_{\theta}}$, we also obtain that the function
 	$$
 	t \in [0, \infty)^{n} \to   -\sum_{\theta,\vartheta=1}^{l}c_{\theta}c_{\vartheta}g(t, \|u_{\theta}- u_{\vartheta}\|^{2}),
 	$$   
 	satisfies relation $(i)$ of Theorem \ref{basicradialndim} with $n$ variables. Since the representation in Theorem \ref{basicradialndim} is unique, we obtain that 
 	$$
 	\sum_{\theta,\vartheta=1}^{l}c_{\theta}c_{\vartheta}\eta_{\|u_{\theta}- u_{\vartheta}\|^{2}} 
 	$$
 	is a nonpositive measure, which proves our first claim.\\
 	It is continuous because by the comment made after Equation \ref{schoenmetriccond}, it is sufficient to prove continuity in the diagonal, which occurs because by  Equation \ref{bern1ineq} 
 	\begin{align*}
 		0 &\leq  \eta_{\|u-v\|^{2}}(A) \leq \eta_{\|u-v\|^{2}}([0, \infty)^{n} )\leq \int_{[0, \infty)^{n}} \prod_{i=1}^{n}(1-e^{-r_{i}})\frac{1 +r_{i}}{r_{i}}d\eta_{\|u-v\|^{2}}(r)\\
 		& = g( 1, \ldots , 1, \|u-v\|^{2}),
 	\end{align*}
 	and  $g(1, \ldots , 1, 0)=0$.
 	Note also that  $g(t,0)$ is the zero function,  and because the representation is unique,  $\eta_{0}$ is the zero measure.\\
 	From those relations, by Theorem \ref{reprcondneg}, we have that 
 	$$
 	\eta_{s}(A)=\int_{[0,\infty)}(1-e^{-r_{n+1}s})\frac{1 +r_{n+1}}{r_{n+1}}d\eta_{A}(r_{n+1})
 	$$
 	where  $\eta_{A}$ is a nonnegative measure in $\mathfrak{M}([0, \infty))$. Note that  $\eta_{\emptyset}$ is the zero measure. Since  $\eta_{s}$ is a measure we have that if $(A_{n})_{n \in \mathbb{N}}$ is a disjoint sequence of Borel measurable sets in $[0, \infty)^{n}$, then
 	\begin{align*}
 		& \int_{[0,\infty)}(1-e^{-sr_{n+1}})\frac{1 +r_{n+1}}{r_{n+1}}d\eta_{\cup_{k \in \mathbb{N}}A_{k}}(r_{n+1})= \eta_{s}(\cup_{k \in \mathbb{N}}A_{k}) \\
 		&= \sum_{k \in \mathbb{N}}\eta_{s}(A_{k})=\sum_{k \in \mathbb{N}}\int_{[0,\infty)}(1-e^{-sr_{n+1}})\frac{1 +r_{n+1}}{r_{n+1}}d\eta_{A_{k}}(r_{n+1})\\
 		&= \int_{[0,\infty)}(1-e^{-sr_{n+1}})\frac{1 +r_{n+1}}{r_{n+1}}d\left [\sum_{k \in \mathbb{N}} \eta_{A_{k}}\right ](r_{n+1}),
 	\end{align*}
 	since the representation is unique we obtain that $
 	\eta_{\cup_{n \in \mathbb{N}}A_{n}} =\sum_{n \in \mathbb{N}} \eta_{A_{n}}$.  Hence, the function $A \times B \to \eta_{A}(B)$ is a nonnegative bimeasure, which by Theorem $1.10$ in \cite{Berg1984} there exists a nonnegative measure $\eta \in \mathfrak{M}([0,\infty)^{n}\times [0,\infty))$ such that $\eta(A\times B) = \eta_{A}(B)$, for every $A \in \mathscr{B}([0,\infty)^{n})$ and  $B \in \mathscr{B}([0,\infty))$. \\
 	Gathering all this information we obtain that
 	\begin{align*}
 		g( t_{1}&, \ldots , t_{n}, t_{n+1})=  \int_{[0,\infty)^{n}}\prod_{i=1}^{n}(1-e^{-r_{i}t_{i}})\frac{1 +r_{i}}{r_{i}}d\eta_{t_{n+1}}(r_{1}, \ldots, r_{n})\\
 		&=\int_{[0,\infty)^{n+1}}\left [\prod_{i=1}^{n}(1-e^{-r_{i}t_{i}})\frac{1 +r_{i}}{r_{i}}\right ](1-e^{-t_{n+1}r_{n+1}})\frac{1 +r_{n+1}}{r_{n+1}}d\eta(r_{1}, \ldots, r_{n}, r_{n+1})\\
 		&=\int_{[0,\infty)^{n+1}}\prod_{i=1}^{n+1}(1-e^{-r_{i}t_{i}})\frac{1 +r_{i}}{r_{i}}d\eta(r), 
 	\end{align*}
 	which concludes that $(i)$ implies $(iv)$.\end{proof}

 From the following simple inequality				
 \begin{equation}\label{ineqexp}
 	(1-e^{-sa})\leq  \max \left (1, \frac{a}{b} \right ) (1-e^{-sb}), \quad s \in [0, \infty), \quad a,b> 0, 
 \end{equation}
 we obtain that for any function $g$ that satisfies Theorem \ref{basicradialndim}
 \begin{equation}\label{consineqexp}
 	g(t_{\vec{1}})\leq  \left [\prod_{i=1}^{n} \max (1, t^{1}_{i}/t^{2}_{i} )\right ]g(t_{\vec{2}}), \quad t_{\vec{1}}, t_{\vec{2}} \in (0, \infty)^{n}, 	
 \end{equation}
 that $g$ is increasing in the sense that  $g(t_{\vec{2}}) \geq  g(t_{\vec{1}}) $ if $t_{\vec{2}} - t_{\vec{1}} \in [0, \infty)^{n}$ and that 
 \begin{equation}\label{consineqexp2}
 	g(t)\leq  g(\vec{1})\prod_{i=1}^{n} (1+ t_{i} ), \quad t    \in [0, \infty)^{n}. 	
 \end{equation}	
 Since 
 $$
 \frac{(1-e^{-s(a+b)})}{s}\leq \frac{(1-e^{-sa})}{s} +\frac{(1-e^{-sb})}{s}, \quad a,b,s \in [0, \infty)
 $$
 we obtain that for every $t_{\vec{1}}, t_{\vec{2}} \in [0, \infty)^{n}$
 \begin{equation}\label{convexexp}
 	g(t_{\vec{1}} + t_{\vec{2}}) \leq  \sum_{\alpha \in \mathbb{N}_{2}^{n}}g(t_{\alpha}).  	
 \end{equation}

 We conclude this Section with two results in which the case $n=1$ is shown in the proof of Theorem \ref{basicradialndim}, and the general case  will simplify the proof  for the characterization of   PDI$_{k,n}^{\infty}$ functions in Theorem \ref{bernsksevndimpart3}.
 
 \begin{corollary}\label{continuity0} Let $g: [0, \infty)^{n} \to \mathbb{R}$ be an PDI$_{ n}^{\infty}$  function such that $g(t)=0$ for every $t \in \partial_{n-1}^{n}$. Then, $g$ is continuous if and only if $g$ is continuous on $\partial_{n-1}^{n}$, that is, for any  $s \in \partial_{n-1}^{n}$ we must have that $\lim_{t \to s}g(t)=g(s)=0$. 
 \end{corollary}
 
 \begin{proof}The proof is done by induction, where the case $n=1$ is a direct consequence of the comment made after Equation \ref{schoenmetriccond}.\\
 	Now, suppose that the result is valid for $n$ variables and we prove for $n+1$. Similar to the proof of Theorem \ref{basicradialndim}, for any fixed $s \in [0, \infty)$ the function
 	$$
 	t=(t_{1}, \ldots, t_{n}) \in [0, \infty)^{n}\to g(t,s) \in \mathbb{R}
 	$$
 	satisfies the restrictions of the statement of the Corollary with $n$ variables, thus it is continuous. If we continue with the same arguments of Theorem \ref{basicradialndim}  we obtain the continuity of $g$.  \end{proof}

 \begin{lemma}\label{techsimple} Let $X$ be a Hausdorff space and $\eta_{t} \in \mathfrak{M}(X)$, $t \in [0, \infty)^{k}$, be a collection of measures such that $t \in [0, \infty)^{k} \to \eta_{t}(X)$ is continuous   and $\eta_{t}$ is the zero measure whenever $t \in \partial_{k-1}^{k}$. Then, for every  $A \in \mathscr{B}(X)$  the function $t \in [0, \infty)^{k} \to \eta_{t}(A)$  defines  an PDI$_{k}^{\infty}$ radial kernel in all Euclidean spaces if, and only if,  there exists a nonnegative measure  $\eta \in \mathfrak{M}(X \times [0, \infty)^{k})$ such that
 	$$
 	\eta_{t}(A)= \int_{A\times [0, \infty)^{k}}\left [\prod_{i=1}^{k}(1-e^{-r_{i}t_{i}})\frac{1 +r_{i}}{r_{i}}\right ]    d\eta(x,r), \quad A \in \mathscr{B}(X).
 	$$
 	Also, for any   bounded measurable function $h: X \to \mathbb{R}$, it holds that
 	$$
 	\int_{X}h(x)d\eta_{t}(x)=  \int_{X\times [0, \infty)^{k}}h(x)\left [\prod_{i=1}^{k}(1-e^{-r_{i}t_{i}})\frac{1 +r_{i}}{r_{i}}\right ]     d\eta(x,r).
 	$$

 \end{lemma}
 
 \begin{proof} If there exists the  nonnegative measure  $\eta \in \mathfrak{M}(X \times [0, \infty)^{k})$ for which the equation is satisfied, then for every $d \in \mathbb{R}^{d}$ and $\mu \in \mathcal{M}_{k}((\mathbb{R}^{d})_{k}) $ we have that for any $A \in \mathscr{B}(X)$
 	\begin{align*}
 		&	\int_{(\mathbb{R}^{d})_{k}} \int_{(\mathbb{R}^{d})_{k}} (-1)^{k}\eta_{\|x_{1} - y_{1}\|^{2}, \ldots, \|x_{k} - y_{k}\|^{2}}(A)d\mu(x)d\mu(y)\\
 		& = \int_{A\times [0, \infty)^{k}} \left [	\int_{(\mathbb{R}^{d})_{k}} \int_{(\mathbb{R}^{d})_{k}}(-1)^{k}\prod_{i=1}^{k}( 1-e^{-\|x_{i} - y_{i}\|^{2}r_{i}})\frac{1 +r_{i}}{r_{i}}d\mu(x)d\mu(y) \right ]   d\eta(x,r) \geq 0.
 	\end{align*}	
 	For the converse, note that the measure $\eta_{t}$ is nonnegative for every $t \in [0, \infty)^{k}$. \\ 
 	The function $t \in [0, \infty)^{k} \to \eta_{t}(A)$  is continuous for every $A \in \mathscr{B}(X)$. Indeed, since  $0 \leq \eta_{t}(A)\leq \eta_{t}(X)$, the continuity of  $t \in [0, \infty)^{k} \to \eta_{t}(X)$ implies that if  $t_{\vec{1}} \to t_{\vec{2}}$ and $t_{\vec{2}}$ is an element of $\partial_{k-1}^{k}$, then $\eta_{t_{\vec{1}}}(A) \to \eta_{t_{\vec{2}}}(A)=0$, Lemma \ref{continuity0} concludes the continuity of $\eta_{t}(A)$. \\
 	By Theorem \ref{basicradialndim}, we have that  for any $A \in \mathscr{B}(X)$ there exists an unique nonnegative measure $\eta_{A}$ in $\mathfrak{M}([0, \infty)^{k})$ such that
 	$$
 	\eta_{t}(A)=\int_{[0,\infty)^{k}}\left [\prod_{i=1}^{k}(1-e^{-r_{i}t_{i}})\frac{1 +r_{i}}{r_{i}}\right ]    d\eta_{A}(r), \quad t \in [0, \infty)^{k}.
 	$$
 	The remaining arguments are the same as the ones presented in the final steps of the proof of Theorem \ref{bernssevndim}, and thus  omitted.\end{proof}

 \section{PD  radial kernels in several variables}\label{PDseveralvariablessec}
 
 In this short Section, we provide a  characterization of the positive definite  radial kernels with  $n$ variables in a Kronecker product of  Euclidean spaces  $\mathbb{R}^{d_{1}} \times \ldots \times\mathbb{R}^{d_{n}}$ for any  $d=(d_{1}, \ldots, d_{n}) \in \mathbb{N}^{n}$. The case where $d$  is fixed can be found in \cite{fernandez2003flexible} and also in  \cite{AlonsoMalaver2015}.

 We do not claim that such characterization is new in the literature, however,  we present a proof for it   because relation $(i)$ in Theorem \ref{PDseveralvariables} is a   crucial part of the proof of the radial PDI$_{k, n}^{\infty}$  functions  obtained in Theorem \ref{bernsksevndimpart3}, and we do not expect it to be explicit stated elsewhere.

 \begin{theorem}\label{PDseveralvariables} Let $g : [0, \infty)^{n} \to \mathbb{R}$ be a continuous function. The following conditions are equivalent:
 	\begin{enumerate}
 		\item [$(i)$] For any $d\in \mathbb{N}$ and discrete  measures  $\mu_{i}$ in $\mathbb{R}^{d}$, $1\leq i \leq n$,  it holds that
 		$$
 		\int_{(\mathbb{R}^{d})_{n}}\int_{ (\mathbb{R}^{d})_{n}}g(\|x_{1}-y_{1}\|^{2}, \ldots, \|x_{n} - y_{n}\|^{2})d[\bigtimes_{i=1}^{n}\mu_{i}](x)d[ \bigtimes_{i=1}^{n}\mu_{i}](y)\geq 0.
 		$$
 		\item [$(ii)$]  The kernel  	
 		$$
 		g(\|x_{1}-y_{1}\|^{2}, \ldots, \|x_{n} - y_{n}\|^{2}), \quad x_{i},y_{i} \in  \mathbb{R}^{d}
 		$$ is PD  for every  $d\in \mathbb{N}$.
 		\item [$(iii)$]  The function  can  be represented as
 		\begin{align*}
 			g(t)  = \int_{[0, \infty)^{n}}  e^{-r\cdot t}d\eta(r_{1},\ldots,  r_{n})
 		\end{align*}
 		where  the measure $\eta \in \mathfrak{M}([0,\infty)^{n}$ is nonnegative. Further, the representation is unique.
 		\item[$(iv)$]The function $g$  is completely monotone in $(0, \infty)^{n}$.

 	\end{enumerate}	
 \end{theorem}

 \begin{proof}  Relation $(iii)$ implies relation $(iv)$ by Theorem  \ref{Bochnercomplsevndim}.\\
 	Relation $(iv)$ implies relation $(iii)$, because by Theorem \ref{Bochnercomplsevndim} the equality holds for $t \in (0,\infty)^{n}$, using the Monotone Convergence Theorem and the continuity of $g$, we obtain the other cases, including $t=\vec{0}$, which implies that $\eta$ is finite. \\
 	Relation $(iii)$ implies relation $(ii)$ by direct verification. \\
 	Relation $(i)$ is a special case of relation $(ii)$.\\
 	To conclude, we prove that relation $(i)$ implies relation $(iii)$ by induction on $n$, where the case initial case is  the classical result due to Schoenberg and  proved in \cite{schoenbradial}.\\
 	Indeed, let $s \in [0, \infty)$, be a fixed number and let $x_{n+1}$ and $y_{n+1} $ be arbitrary vectors in $ \mathbb{R}^{d}$ such that $\|x_{n+1} - y_{n+1}\|^{2}=s$. The functions
 	$$
 	(t_{1}, \ldots, t_{n}) \to -g(t_{1}, \ldots, t_{n}, s ) + g(t_{1}, \ldots, t_{n},0 ) \in \mathbb{R}, \quad  (t_{1}, \ldots, t_{n})  \to g(t_{1}, \ldots, t_{n}, 0 ) \in \mathbb{R}
 	$$
 	satisfy relation $(i)$ with $n$ variables, because we took respectively $\mu_{n+1} $ equals to $(\delta_{x_{n+1}} - \delta_{y_{n+1}})/2$ and $\delta_{x_{n+1}}$.
 	Hence, those two functions admits an unique  integral representation of relation  $(iii)$ with $n$ variables, so   there exists an unique  measure $\eta_{s} \in \mathfrak{M}([0,\infty)^{n})$ such that
 	$$
 	g(t_{1}, \ldots, t_{n}, s )= \int_{[0, \infty)^{n}}  e^{-r\cdot t}d\eta_{s}(r_{1},\ldots,  r_{n}).
 	$$
 	Now, pick an arbitrary discrete measure $\mu_{n+1}:= \sum_{i=1}^{m}c_{i}z_{n+1}^{i}   $
 	in $\mathbb{R}^{d}$, by the same argument  as before, the function
 	$$
 	(t_{1}, \ldots, t_{n}) \to \sum_{i,j=1}^{m}c_{i}c_{j}g(t_{1}, \ldots, t_{n}, \|z_{n+1}^{i} - z_{n+1}^{j}\|^{2} )\in \mathbb{R},
 	$$
 	satisfy relation $(i)$ with $n$ variables, and since the integral representation in relation  $(iii)$ with $n$ variables is unique, we get that the following   measure in $\mathfrak{M}([0,\infty)^{n})$
 	\begin{equation}\label{PDseveralvariableseq2}
 		\sum_{i,j=1}^{m}c_{i}c_{j}\eta_{\|z_{n+1}^{i} - z_{n+1}^{j}\|^{2}} \quad  \text { is always nonnegative.}
 	\end{equation}
 	Now, we prove that for any $A \in \mathscr{B}([0, \infty)^{n})$, the function
 	$$
 	s\in [0, \infty) \to \eta_{s}(A),
 	$$
 	is continuous. Indeed, as it defines a positive definite radial kernel in all Euclidean spaces, by the comment made after Equation   \ref{schoenmetriccond}  (because $-\eta_{\|x_{n+1} - y_{n+1}\|^{2}}(A) + \eta_{0}(A)$ is an CND kernel that is zero in the diagonal), to prove it is continuous it is sufficient to prove the continuity at $s=0$, and for that we note that  by Equation \ref{PDseveralvariableseq2} the measure  $\eta_{0} - \eta_{s}$ is   nonnegative, and then
 	$$
 	0\leq \eta_{0}(A) - \eta_{s}(A) \leq \eta_{0}([0, \infty)^{n}) - \eta_{s}([0, \infty)^{n}) = g( 0, \ldots , 0, 0) - g( 0, \ldots , 0, s).
 	$$
 	By  the Schoenberg characterization  in \cite{schoenbradial}, we have that  for any $A \in \mathscr{B}([0, \infty)^{n})$ there is an unique nonnegative measure $\eta_{A}$ in $\mathfrak{M}([0, \infty))$ such that
 	$$
 	\eta_{s}(A)=\int_{[0,\infty)}e^{-r_{n+1}s}d\eta_{A}(r_{n+1}).
 	$$
 	Note that  $\eta_{\emptyset}$ is the zero measure. Since  $\eta_{s}$ is a measure we have that if $(A_{k})_{k \in \mathbb{N}}$ is a disjoint sequence of Borel measurable sets in $[0, \infty)^{n}$, then
 	\begin{align*}
 		& \int_{[0,\infty)^{n}}e^{-r_{n+1}s}d\eta_{(\cup_{k \in \mathbb{N}}A_{k})}(r_{n+1})= \eta_{s}(\cup_{n \in \mathbb{N}}A_{n}) \\
 		&= \sum_{k \in \mathbb{N}}\eta_{s}(A_{k})=\sum_{k \in \mathbb{N}}\int_{[0,\infty)^{n}}e^{-r_{n+1}s}d\eta_{A_{k}}(r_{n+1})\\
 		&= \int_{[0,\infty)^{n}}e^{-r_{n+1}s}d\left [\sum_{k \in \mathbb{N}} \eta_{A_{k}}\right ](r_{n+1}),
 	\end{align*}
 	because  the representation is unique we obtain that $
 	\eta_{(\cup_{k \in \mathbb{N}}A_{k})} =\sum_{k \in \mathbb{N}} \eta_{A_{k}}$.  The function $A \times B \to \eta_{A}(B)$ is then a nonnegative bimeasure, which by Theorem $1.10$ in \cite{Berg1984}, there exists a nonnegative measure $\eta \in \mathfrak{M}([0,\infty)^{n}\times [0,\infty))$ such that $\eta(A\times B) = \eta_{A}(B)$, for every $A \in \mathscr{B}([0,\infty)^{n})$ and  $B \in \mathscr{B}([0,\infty))$. \\
 	Gathering all this information we obtain that
 	\begin{align*}
 		g( t_{1}, \ldots ,t_{n},t_{n+1})&=  \int_{[0,\infty)^{n}}e^{-\sum_{i=1}^{n}r_{i}t_{i}}d\eta_{t_{n+1}}(r_{1}, \ldots, r_{n})\\
 		&=\int_{[0,\infty)^{n+1}}e^{-\sum_{i=1}^{n}r_{i}t_{i}}e^{-r_{n+1}t_{n+1}}d\eta(r_{1}, \ldots, r_{n}, r_{n+1})\\
 		&=\int_{[0,\infty)^{n+1}}e^{-r\cdot t}d\eta(r), 
 	\end{align*}
 	which concludes that $(i)$ implies $(iii)$. \end{proof}

 The equivalence between relation $(ii)$ and $(iii)$ in Theorem \ref{PDseveralvariables} can also be proved as a corollary of Proposition $4.7$ page $115$ in \cite{Berg1984}.

 \section{Bernstein functions of order $k$ in $n$ variables}\label{Partialindependencetests}
 As presented in Section \ref{Terminology}, throughout this Section we always assume that $0\leq k \leq n$, as the functions being analyzed are intrinsically related with the sets $\mathcal{M}_{k}( (\mathbb{R}^{d})_{n})$.

 \begin{definition}\label{defnpdikkn} A   function $g:[0, \infty)^{n} \to \mathbb{R}$ is called   positive definite independent  of order $k$  in all Euclidean spaces (PDI$_{k,n}^{\infty}$)  if for every $d \in \mathbb{N}$  and $\mu \in \mathcal{M}_{k}( (\mathbb{R}^{d})_{n})$ it satisfies
 	$$
 	\int_{\mathbb{R}^{d} }	\int_{\mathbb{R}^{d} }(-1)^{k}g(\|u_{1} - v_{1}\|^{2}, \ldots,\|u_{n} - v_{n}\|^{2} )d\mu(u)d\mu(v) \geq 0.
 	$$	
 	If for every $d \in \mathbb{N}$  the previous inequality is an equality only when $\mu$ is the zero measure in $\mathcal{M}_{k}( (\mathbb{R}^{d})_{n})$, we say that $g$  is a strictly positive definite  independent function of order $k$  in all Euclidean spaces (SPDI$_{k,n}^{\infty}$).   
 \end{definition}

 Next, we provide a version of Lemma  \ref{PDInsimpli} for PDI$_{k,n}^{\infty}$ functions. For that, we define the sets 
 $$
 \partial_{k-1}^{n}:=\{r \in [0, \infty)^{n}, \quad |\{i,\quad  r_{i}>0  \} | < k \}
 $$
 They satisfy the inclusion property $\partial_{j}^{n} \subset \partial_{j+1}^{n}$ where  $\partial_{-1}^{n}= \emptyset$, $\partial_{0}^{n}= \{\vec{0}\}$, $\partial_{n-1}^{n}= [0, \infty)^{n}\setminus{(0, \infty)^{n}}$ and $ \partial_{n}^{n}=[0, \infty)^{n}$.

 \begin{lemma}\label{PDInsimplik}	Let $g: [0, \infty)^{n} \to \mathbb{R}$ be a  continuous  function and define
 	\begin{align*}
 		G(t):= \int_{[0, \infty)^{n}}g(s)d\mu_{k}^{n}[t,\vec{0}](s) = g(t)+  \sum_{j=0}^{k-1}(-1)^{k-j}\binom{n-j-1}{n-k}\sum_{|F|=j}g(t_{F}), 
 	\end{align*} 
 	where $t_{F}= \sum_{i\in F}t_{i}e_{i}$. Then, for any  $\mu \in \mathcal{M}_{k}(( \mathbb{R}^{d})_{n})$ 
 	\begin{align*} 
 		\int_{(\mathbb{R}^{d})_{n}}\int_{(\mathbb{R}^{d})_{n}}&(-1)^{k}g(\|u_{1} - v_{1}\|^{2}, \ldots,\|u_{n} - v_{n}\|^{2} )d\mu(u)d\mu(v)\\
 		&=\int_{(\mathbb{R}^{d})_{n}}\int_{(\mathbb{R}^{d})_{n}}(-1)^{k}G(\|u_{1} - v_{1}\|^{2}, \ldots,\|u_{n} - v_{n}\|^{2} )d\mu(u)d\mu(v),
 	\end{align*} 
 	hence, $g$ is PDI$_{k,n}^{\infty}$ if and only if  $G$ is PDI$_{k,n}^{\infty}$.\\
 	If  $g(t)=0$ for every $t \in \partial_{k-1}^{n}$ then $g=G$.\\
 	If $t \in \partial_{k-1}^{n}$ then $G(t)=0$.
 \end{lemma}

 \begin{proof} Note that at the exception of the term $g(t)$ in the definition of $G$, the other terms   depend on a maximum of  $k-1$ among the $n$ variables of  $t$. As a consequence of  Equation \ref{integmu0n}, for every $\mu \in \mathcal{M}_{k}( (\mathbb{R}^{d})_{n})$  we obtain the equality
 	\begin{align*} 
 		\int_{(\mathbb{R}^{d})_{n}}\int_{(\mathbb{R}^{d})_{n}}&(-1)^{k}g(\|u_{1} - v_{1}\|^{2}, \ldots,\|u_{n} - v_{n}\|^{2} )d\mu(u)d\mu(v)\\
 		&=\int_{(\mathbb{R}^{d})_{n}}\int_{(\mathbb{R}^{d})_{n}}(-1)^{k}G(\|u_{1} - v_{1}\|^{2}, \ldots,\|u_{n} - v_{n}\|^{2} )d\mu(u)d\mu(v).
 	\end{align*} 
 	If  $g(t)=0$ for every $t \in \partial_{k-1}^{n}$ then $g=G$ by the way we defined $G$. For the remaining property is a direct consequence of Lemma \ref{measorderk}. \end{proof}	
 
 Now, our objective is to characterize the PDI$_{k,n}^{\infty}$ functions, but before that,  we present some properties and inequalities for the elementary symmetric polynomials which will be necessary. We obtain those relations using only well known properties for them and   additional information   can be found in  \cite{macdonald1998symmetric}.

 \subsection{Inequalities for elementary symmetric polynomials}\label{elementarysymmetricpolynomials}

 The elementary symmetric polynomials $p^{n}_{k}$, with $0\leq k\leq n$, are the functions
 $$
 p^{n}_{k}(r_{1}, \ldots, r_{n}):= \sum_{1 \leq i_{1}  < \ldots <  i_{k} \leq n}r_{i_{1}}\ldots r_{i_{k}}, 
 $$
 and  $p^{n}_{0}:= 1$. It is widely known that $p^{n}_{k}(\vec{1})= \binom{n}{k}$ and the  generating function formula
 \begin{equation}\label{genformelepol}
 	\prod_{i=1}^{n}(\lambda + r_{i})= \sum_{k=0}^{n}\lambda^{n-k}p_{k}^{n}(r), \quad \lambda \in \mathbb{R},  \quad  r \in \mathbb{R}^{n}.
 \end{equation}

 An important relation obtained from simple combinatorics  is that
 \begin{equation}\label{inducksympol}
 	p_{j}^{n+1}(r, r_{n+1})= p_{j}^{n}(r) + r_{n+1}p_{j-1}^{n}(r), \quad 0\leq j< n+1, \quad r \in \mathbb{R}^{n}, r_{n+1} \in \mathbb{R}.
 \end{equation}

 Now, for $n, k \in \mathbb{N}$, where $n\geq k \geq  0$  we define the function 
 \begin{align*}
 	H_{k}^{n}(r)&:= p^{n}_{n}(r) +  (-1)^{1}\binom{n-k}{n-k}p^{n}_{k-1}(r) + \ldots +(-1)^{k}\binom{n-2}{n-k}p_{1}^{n}(r) + (-1)^{k+1}\binom{n-1}{n-k}\\
 	&= p^{n}_{n}(r) + \sum_{j=0}^{k-1} (-1)^{k-j}\binom{n-j-1}{n-k}p^{n}_{j}(r)\\
 	&= \int_{[0, \infty)^{n}}p_{n}^{n}(s)d\mu_{k}^{n}[r, \vec{1}](s). 
 \end{align*}
 
 Note  the similarity with the measure $	\mu^{n}_{k}[x_{\vec{1}}, x_{\vec{2}}] $ defined in Lemma \ref{measorderk}.\\
 When  $k= n$ we have used this function applied with an  exponential in each coordinate to obtain the multivariable generalizations of Schoenberg's results, as by  Equation \ref{genformelepol}
 $$
 (-1)^{n}H_{n}^{n}(r)=(-1)^{n}p^{n}_{n}(r) + \sum_{j=0}^{n-1} (-1)^{ j } p^{n}_{j}(r)=\prod_{i=1}^{n}(1-r_{i}) =p_{n}^{n}(\vec{1}-r), 
 $$
 and this function appeared in Theorem \ref{basicradialndim}.

 \begin{lemma}\label{ineqorderk} Let $n \geq 2$ and $0 <  k \leq  n$, then for every  $a  \in [0,1]^{n}$
 	\begin{equation}\label{ineqorderk2eq1}
 		0\leq \binom{n}{k}^{-1}p_{k}^{n}(\vec{1}-a) \leq (-1)^{k} H_{k}^{n}(a) \leq p_{k}^{n}(\vec{1}-a). 
 	\end{equation} 
 	Also, 
 	\begin{equation}\label{ineqorderk2eq2}
 		0\leq p_{k}^{n}(s_{1}t_{1}, \ldots, s_{n}t_{n})\leq  p_{k}^{n}(s) p_{k}^{n}(t), \quad t,s \in  [0,\infty)^{n}
 	\end{equation}	
 \end{lemma}

 \begin{proof}   First, we prove Equation \ref{ineqorderk2eq1} for an  arbitrary $n \geq 1$  and $k=1$, which is the following 
 	$$
 	0 \leq \frac{\sum_{i=1}^{n} (1 - a_{i})}{n} \leq 1 - \prod_{i=1}^{n}a_{i} \leq  \sum_{i=1}^{n} (1-a_{i}), \quad a\in [0, 1]^{n}.
 	$$
 	The positivity of the first inequality is immediate. The second inequality  is a direct consequence that for any $1\leq i \leq n$ we have that $1 - a_{i} \leq 1 - \prod_{j=1}^{n}a_{j}$. For  the third inequality,   the case $n=2$ holds because can be rewritten as the inequality $(1-a_{1})(1-a_{2})\geq 0$.   For the general case,  given  fixed  $a_{1}, \ldots a_{n } \in [0,1]$, the  linear function $a_{n+1} \in \mathbb{R} \to T(a_{n+1}):= - 1 + \prod_{i=1}^{n+1}a_{i} + \sum_{i=1}^{n+1} (1-a_{i})  $ is decreasing, then  for $a_{n+1} \in [0,1]$ it holds that $T(a_{n+1})\geq T(1)= -1 + \prod_{i=1}^{n }a_{i} + \sum_{i=1}^{n } (1-a_{i})   \geq 0$ by induction in $n$.\\ 
 	Now, for a fixed $k\geq 2$, the proof is done by induction in $n$, where the case $n=k$ is a direct consequence of the equality $(-1)^{n}H_{n}^{n}(r)=  p_{n}^{n}(\vec{1}-r)$  and we may also suppose that  it holds for any $k^{\prime}$ and $n^{\prime}$ such that $1 \leq k^{\prime}< k$ and $n^{\prime}\geq k^{\prime}$.\\
 	Due to Equation \ref{inducksympol}, the following polynomials are equal
 	\begin{equation}\label{inducksympolHKN}
 		\begin{split}
 			H_{k}^{n+1}(a, a_{n+1})&= \sum_{j=0}^{k-1} (-1)^{k-j}\binom{n-j}{n+1-k} p_{j}^{n}(a) \\
 			& \quad + a_{n+1}[p_{n}^{n}(a) + \sum_{j=1}^{k-1} (-1)^{k-j }\binom{n-j}{n+1-k}p_{j-1}^{n}(a)]\\
 			&= \sum_{j=0}^{k-1} (-1)^{k-j }\binom{n-j}{n+1-k} p_{j}^{n}(a) + a_{n+1}H_{k-1}^{n}(a).
 		\end{split}
 	\end{equation}
 	Hence, for a fixed $a \in [0, 1]^{n}$ the linear function  $a_{n+1} \in \mathbb{R} \to (-1)^{k}H_{k}^{n+1}(a, a_{n+1})$ is decreasing, because by the induction in $k$ we have that  $(-1)^{k}H_{k-1}^{n}(a)\leq 0$. However, 
 	\begin{align*}
 		&H_{k}^{n+1}(a,1)\\
 		&= p_{n}^{n}(a) + \sum_{j=0}^{k-1} (-1)^{k-j }\binom{n-j}{n+1-k} p_{j}^{n}(a) +  \sum_{j=1}^{k-1} (-1)^{k-j }\binom{n-j}{n+1-k}p_{j-1}^{n}(a) \\
 		&= p_{n}^{n}(a) -p_{k-1}^{n}(a)+  \sum_{j=0}^{k-2}(-1)^{k-j }\left [ \binom{n-j}{n+1-k} - \binom{n-j-1}{n+1-k} \right ]p_{j}^{n}(a)  = H_{k}^{n}(a),
 	\end{align*}
 	hence, by the induction in $n$ we have that $(-1)^{k}H_{k}^{n+1}(a, a_{n+1})\geq (-1)^{k}H_{k}^{n+1}(a,1)=(-1)^{k}H_{k}^{n}(a) $ when $a_{n+1}\in [0,1]$. Applying this relation recursively $n-k+1$ times we obtain that for $a_{1}, \ldots, a_{n+1} \in [0,1]$
 	$$
 	(-1)^{k}H_{k}^{n+1}(a_{1}, \ldots,a_{n+1}  ) \geq (-1)^{k}H_{k}^{k}(a_{1}, \ldots,a_{k}) =\prod_{i=1}^{k}(1-a_{i}) \geq 0, 
 	$$
 	but, since the functions involved are symmetric, it also holds that
 	$$
 	(-1)^{k}H_{k}^{n+1}(a_{1}, \ldots,a_{n+1}  ) \geq \prod_{i\in F}(1-a_{i}) \geq 0, \quad |F|=k, F \subset \{1, \ldots, n+1\}, 
 	$$
 	which implies  the first and second inequality in Equation \ref{ineqorderk2eq1}, because by summing all the above inequalities we obtain that 
 	$$
 	\binom{n+1}{k}(-1)^{k}H_{k}^{n+1}(a_{1}, \ldots,a_{n+1}  ) \geq \sum_{|F|=k}\prod_{i\in F}(1-a_{i}) = p_{k}^{n+1}(\vec{1} - a) \geq 0.
 	$$
 	For the third inequality in Equation \ref{ineqorderk2eq1}, using the same relations used for the other inequalities, we have that the following polynomials are equal
 	\begin{align*}
 		p_{k}^{n+1}&(\vec{1}-a, 1 - a_{n+1})-(-1)^{k} H_{k}^{n+1}(a, a_{n+1})\\
 		&=  p_{k}^{n}(\vec{1}-a) + p_{k-1}^{n}(\vec{1}-a)- \sum_{j=0}^{k-1} (-1)^{ j }\binom{n-j}{n+1-k} p_{j}^{n}(a)\\
 		& \quad \quad  +a_{n+1}( (-1)^{k+1}H_{k-1}^{n }(a) - p^{n}_{k-1}(\vec{1}-a)).
 	\end{align*}
 	In particular, for a fixed $a \in [0, 1]^{n}$ the term that multiplies $a_{n+1}$, that is, $(-1)^{k+1}H_{k-1}^{n }(a) - p^{n}_{k-1}(\vec{1}-a)$ is negative by the induction in $k$, thus the linear function  $a_{n+1} \in \mathbb{R} \to  p_{k}^{n+1}(\vec{1}-a, 1 - a_{n+1})-(-1)^{k} H_{k}^{n+1}(a, a_{n+1})$  is decreasing. Hence,  for $a_{n+1} \in [0,1]$ we have that 
 	\begin{align*}
 		p_{k}^{n+1}(\vec{1}-a, 1 - a_{n+1})-(-1)^{k} H_{k}^{n+1}(a, a_{n+1})&\geq  p_{k}^{n+1}(\vec{1}-a, 0)-(-1)^{k} H_{k}^{n+1}(a, 1)\\
 		&=  p_{k}^{n}(\vec{1}-a)- (-1)^{k}H_{k}^{n}(a),
 	\end{align*}
 	and the last term in the previous inequality  is nonnegative  by the induction in $n$, proving the inequality $(-1)^{k} H_{k}^{n+1}(a, a_{n+1})\leq p_{k}^{n+1}(\vec{1}-a, 1 - a_{n+1})$.\\
 	The proof for the remaining inequality is simpler, as  
 	\begin{align*}
 		p_{k}^{n}&(s_{1}t_{1}, \ldots, s_{n}s_{n})=	\sum_{1 \leq i_{1}  < \ldots <  i_{k} \leq n}s_{i_{1}}\ldots s_{i_{k}}t_{i_{1}}\ldots t_{i_{k}}   \\
 		&\leq \left (\sum_{1 \leq i_{1}  < \ldots <  i_{k}\leq n}s_{i_{1}}\ldots s_{i_{k}} \right ) \left (\sum_{1 \leq i_{1}  < \ldots <  i_{k} \leq n}t_{i_{1}}\ldots t_{i_{k}} \right ) = 	p_{k}^{n}(s)p_{k}^{n}(t).
 \end{align*} \end{proof}

 Now, we prove some inequalities for the additional term in the integral representation of Theorem \ref{bernsksevndim}, which in case $n=k$ at Theorem \ref{basicradialndim}  is 
 $$
 \prod_{i=1}^{n}\frac{1+ r_{i}}{r_{i}}=\frac{1+ \ldots  + p_{n}^{n}(r) }{p_{n}^{n}(r)}= \frac{p_{n}^{n}(\vec{1} + r)}{p_{n}^{n}(r)}.
 $$

 First, note that by 	Equation \ref{genformelepol}, we obtain that
 $$
 \sum_{k=0}^{n}\lambda^{n-k}p_{k}^{n}( \vec{1}+ r)= \prod_{i=1}^{n}(\lambda +1 + r_{i})= \sum_{k=0}^{n}(\lambda+1)^{n-k}p_{k}^{n}( r),  \quad  \lambda \in \mathbb{R}, \quad r \in \mathbb{R}^{n}
 $$
 which implies the following representation
 \begin{equation}\label{pknr+1}
 	p_{k}^{n}(\vec{1} +r)= \binom{n}{n-k}p_{0}^{n}(r) + \ldots + \binom{n-k}{n-k}p_{k}^{n}(r)= \sum_{j=0}^{k}\binom{n-j}{n-k}p_{j}^{n}(r), \quad 0\leq k\leq  n.
 \end{equation}

 Due to Equation \ref{pknr+1}, we have that
 
 \begin{equation} \label{eqpkn+1pkn}
 	\sum_{j=0}^{k} p^{n}_{j}(r) \leq p^{n}_{k}(\vec{1} + r) \leq  \binom{n}{k} \sum_{j=0}^{k} p^{n}_{j}(r), \quad r \in [0, \infty)^{n}.
 \end{equation}

 \begin{lemma} Let  $0< k \leq n$, then for  every  $ F\subset \{1, \ldots, n\}  $ with $|F|= k$
 	\begin{equation}\label{ineqdifforderkradial0}
 		\left [\sum_{j=0}^{k} p^{n}_{j}(r) \right ]\prod_{i \in F}r_{i} \leq p^{n}_{k}(r)\prod_{i \in F}(1+r_{i}), \quad r \in [0, \infty)^{n}.
 	\end{equation}
 	Consequently, for for  every		$ L\subset \{1, \ldots, n\}  $ with $l:=|L|\leq k$
 	\begin{equation}\label{ineqdifforderkradial}
 		p^{n}_{k}(\vec{1} + r)p_{k-l}^{n-l}(r_{L^{c}})\left [\prod_{i \in L}r_{i} \right ] \leq \binom{n}{k} p^{n}_{k}(r)p_{k-l}^{n-l}(\vec{1} + r_{L^{c}}) \left [\prod_{i \in L}(1+r_{i})\right ],  \quad r \in [0, \infty)^{n}, 
 	\end{equation}
 	also,
 	\begin{equation}\label{ineqdifforderkradial2}
 		p^{n}_{k}\left (\frac{r_{1}}{1+r_{1}},\ldots, \frac{r_{n}}{1+r_{n}} \right )\leq \binom{n}{k}^{2}\frac{ p^{n}_{k}(r)} {p^{n}_{k}(\vec{1} + r)}, \quad  r \in [0, \infty)^{n}.  
 	\end{equation}

 \end{lemma}	
 
 \begin{proof} 	Under the hypothesis of $k,n, j$ and $F$, for any given   $1\leq i_{1}  < \ldots < i_{j} \leq n$ we may choose $1\leq \ell_{1} < \ldots < \ell_{k} \leq n $ and  an $\mathcal{F} \subset F$ such that $|\mathcal{F}|=j$   that satisfies
 	$$
 	[r_{i_{1}} \ldots r_{i_{j}}] \prod_{i \in F}r_{i} = [r_{\ell_{1}} \ldots r_{\ell_{k}} ] \prod_{i \in \mathcal{F}}r_{i}, \quad r \in \mathbb{R}^{n}.
 	$$
 	This association is injective in the sense that for a distinct sequence  of $j$ terms in $\{1, \dots, n\}$, either we have a  distinct  sequence  of $k$ terms in $\{1, \dots, n\}$ or we have a distinct subset   $\mathcal{F}$ of $F$ with $j$ terms, consequently
 	$$
 	p^{n}_{j}(r)\prod_{i \in F}r_{i} \leq p^{n}_{k}(r) \sum_{|\mathcal{F}|=j , \mathcal{F} \subset F } \left [\prod_{i \in \mathcal{F}}r_{i} \right ], \quad r \in [0, \infty)^{n}, \quad 0 \leq j \leq k.
 	$$
 	If we sum all values of $j$ in the previous inequality and use that
 	$$
 	\prod_{i \in F}(1+r_{i}) = \sum_{j=0}^{k}\sum_{|\mathcal{F}|=j,\mathcal{F} \subset F } \left [\prod_{i \in \mathcal{F}}r_{i} \right ], \quad r \in \mathbb{R}^{n}
 	$$
 	where the case $\mathcal{F}=\emptyset$  is the constant $1$, we obtain Equation \ref{ineqdifforderkradial0}.\\
 	To obtain Equation \ref{ineqdifforderkradial}, for a fixed 	$ L\subset \{1, \ldots, n\}  $ with $l:=|L|\leq k$,  we sum the   $\binom{n-l}{k-l}$ possibilities of an  $ L^{\prime}\subset \{1, \ldots, n\} \setminus L  $ with $|L^{\prime}|=k-l$ by taking $F= L\cup L^{\prime}$ in Equation \ref{ineqdifforderkradial0}, obtaining  that  for any $r \in [0, \infty)^{n}$
 	\begin{align*}
 		\left [\sum_{j=0}^{k} p^{n}_{j}(r) \right ]p_{k-l}^{n-l}(r_{L^{c}})\prod_{i \in L}r_{i} &\leq p^{n}_{k}(r)p_{k-l}^{n-l}(\vec{1} + r_{L^{c}})\prod_{i \in L}(1+r_{i})
 	\end{align*}
 	and Equation \ref{eqpkn+1pkn} concludes the argument.\\
 	To obtain Equation \ref{ineqdifforderkradial2}, by Equation \ref{ineqdifforderkradial0} we have that  for every $ F\subset \{1, \ldots, n\}  $ with $|F|= k$
 	$$
 	\prod_{i \in F}\left (\frac{r_{i}}{1+r_{i}} \right )  \left [\sum_{j=0}^{k} p^{n}_{j}(r) \right ]\leq p^{n}_{k}(r), \quad   r \in [0, \infty)^{n},
 	$$
 	if we sum the $\binom{n}{k}$ possibilities for the possible subsets $F$  in this inequality, we obtain that
 	$$	
 	p^{n}_{k}\left (\frac{r_{1}}{1+r_{1}},\ldots, \frac{r_{n}}{1+r_{n}} \right )\left [\sum_{j=0}^{k} p^{n}_{j}(r) \right ] \leq  \binom{n}{k}p^{n}_{k}(r),  \quad r \in [0, \infty)^{n},  
 	$$
 	and Equation \ref{eqpkn+1pkn} concludes the argument.\end{proof}	
 
 We often use the constant $2^{2n}$ instead of $\binom{n}{k}^{2}$ when using  Equation \ref{ineqdifforderkradial2} to simplify the expressions.
 
 The fraction $p_{k}^{n}(\vec{1} + r)/ p_{k}^{n}(r)$ is part of the integral representation  in Theorem \ref{bernsksevndimpart3}, however, we emphasize that we can also use $p_{k}^{n}(c + r)/ p_{k}^{n}(r)$ for any fixed $c \in (0, \infty)^{n}$, or more generally 	$\sum_{j=0}^{k} b_{j}p^{n}_{j}(r)$, for fixed positive $b_{j}$, $0\leq j \leq k$. We prefer to use  $p_{k}^{n}(\vec{1} + r)/ p_{k}^{n}(r)$ to simplify the expressions.

 \subsection{Radial PDI functions of order $k$ in infinite dimensions}\label{RadialPDIKinfinitedimensions}

 We define the following substitute for the exponential function
 \begin{equation}\label{exponentialkn}
 	E_{k}^{n}(s):=  H_{k}^{n}(e^{-s_{1}}, \ldots, e^{-s_{n}}), \quad s \in \mathbb{R}^{n}, 
 \end{equation}
 in the sense that by Lemma \ref{measorderk}
 \begin{equation}\label{exponentialknmeas}
 	E_{k}^{n}(s)= \int_{\mathbb{R}^{n}}\left (\prod_{i=1}^{n}e^{-r_{i}}\right )d\mu_{k}^{n}[\delta_{s}, \delta_{\vec{0}}](r).
 \end{equation}
 
 For $r,t \in \mathbb{R}^{n}$ we use  the entrywise multiplication $r	\odot t:=(r_{1}t_{1}, \ldots, r_{n}t_{n}) \in \mathbb{R}^{n}$.

 Note that if $n=k$ the function  $E_{k}^{k}(r	\odot t)= \prod_{i=1}^{n}(e^{-r_{i}t_{i}} - 1)$ appears in Theorem \ref{basicradialndim}.

 Our aim in this Section is to characterize the continuous functions $g:[0, \infty)^{n}\to \mathbb{R}$ that are  PDI$_{k, n}^{\infty}$.

 \begin{definition}For $0\leq k\leq n $, a  function $g:(0, \infty)^{n} \to \mathbb{R}$ is  called a  Bernstein function  of order $k$,   if $g \in C^{\infty}((0, \infty)^{n})$ and the $\binom{n}{k}$ functions $[\partial^{\vec{1}_{F}}]g$  are completely monotone for every $|F|=k$. 
 \end{definition}

 We present two classes of examples for Bernstein functions  of order $k$   in $(0, \infty)^{n}$. For the first one,   let $\eta \in \mathfrak{M}([0,\infty)^{n}\setminus{\partial_{k}^{n}})$ be a  nonnegative measure, we affirm that  
 \begin{equation}\label{berskex}
 	g(t):= \int_{[0, \infty)^{n}\setminus{\partial_{k}^{n}}} 	(-1)^{k}	E_{k}^{n}( r\odot t)  \frac{p^{n}_{k}(r + \vec{1})}{p^{n}_{k}(r)}d\eta(r),
 \end{equation}
 is a well defined Bernstein function  of order $k$ in $(0, \infty)^{n}$ that is continuous in $[0, \infty)^{n}$. Indeed, by Equation \ref{ineqorderk2eq1} the integrand is a nonnegative function and 
 $$
 (-1)^{k}	E_{k}^{n}( r\odot t)  \leq p_{k}^{n}(1-e^{-r_{1}t_{1}}, \ldots, 1-e^{-r_{n}t_{n}}).
 $$
 Using in order the  Equations \ref{bern1ineq2}, \ref{ineqorderk2eq2}  and  \ref{ineqdifforderkradial2}    we get that
 \begin{align*}
 	p_{k}^{n}(1-e^{-r_{1}t_{1}},& \ldots, 1-e^{-r_{n}t_{n}}) \leq  p_{k}^{n} (2\max(1, t_{1})r_{1}/(1+r_{1}), \ldots,  2\max(1, t_{n})r_{n}/(1+r_{n}) )\\
 	& \leq 2^{k} p_{k}^{n}(\max(1, t_{1}), \ldots, \max(1, t_{n}))p_{k}^{n} (r_{1}/(1+r_{1}), \ldots,  r_{n}/(1+r_{n}) )\\
 	&\leq  2^{2n + k} p_{k}^{n}(\vec{1} + t) \frac{p_{k}^{n}(r)}{p_{k}^{n}( \vec{1} + r)}, 
 \end{align*} 
 
 hence
 
 \begin{equation}\label{ex1compordkinn}
 	0\leq (-1)^{k}E_{k}^{n}(r \odot t) \frac{p^{n}_{k}(\vec{1}+ r)}{p^{n}_{k}(r)} \leq  2^{2n +k}  p_{k}^{n}(\vec{1} + t).
 \end{equation}
 
 With this inequality we obtain that $g$ is well defined and continuous in $[0, \infty)^{n}$. Also, from  Equation \ref{ineqorderk2eq1} we obtain that $g(t)=0$ when $t \in \partial_{k-1}^{n}$.

 For the differentiability, a consequence of  Equation \ref{inducksympolHKN} is 
 $$
 \partial_{t_{n}}E_{k}^{n}(r_{1}t_{1}, \ldots, r_{n}t_{n})= -r_{n}e^{-r_{n}t_{n}}E_{k-1}^{n-1}(r_{1}t_{1}, \ldots, r_{n-1}t_{n-1}). 
 $$
 From  Equation  \ref{ex1compordkinn} we get that
 \begin{align*}
 	0\leq &(-1)^{k-1}E_{k-1}^{n-1}(r_{1}t_{1}, \ldots, r_{n-1}t_{n-1})\\
 	&\leq 2^{2n+k-3}p_{k-1}^{n-1}(1+ t_{1}, \ldots, 1+t_{n-1})\frac{p_{k-1}^{n-1}(r_{1}, \ldots, r_{n-1})}{p_{k-1}^{n-1}(1+r_{1}, \ldots,1+ r_{n-1})}  ,  
 \end{align*}	
 and by using  Equation  \ref{ineqdifforderkradial} in the case $L=\{n\}$
 \begin{align*}
 	0&\leq r_{n}e^{-r_{n}t_{n}}(-1)^{k-1}E_{k-1}^{n-1}(r_{1}t_{1}, \ldots, r_{n-1}t_{n-1}) \frac{p^{n}_{k}(\vec{1}+r)}{p^{n}_{k}(r)}\\
 	&\leq 2^{2n+k-3}  p_{k-1}^{n-1}(1+ t_{1}, \ldots, 1+t_{n-1})r_{n}e^{-r_{n}t_{n}}\frac{p_{k-1}^{n-1}(r_{1}, \ldots, r_{n-1})}{p_{k-1}^{n-1}(1+r_{1}, \ldots,1+ r_{n-1})} \frac{p^{n}_{k}(\vec{1}+r)}{p^{n}_{k}(r)}\\
 	&\leq 2^{2n+k-3} \binom{n}{k} p_{k-1}^{n-1}(1+ t_{1}, \ldots, 1+t_{n-1})(1+r_{n})e^{-r_{n}t_{n}},
 \end{align*}
 which is locally $\eta$ integrable for any $t \in (0, \infty)^{n}$, thus 
 $$
 [\partial^{e_{n}}g](t_{1},\ldots, t_{n})= \int_{[0, \infty)^{n}\setminus{\partial_{k}^{n}}}  r_{n}e^{-r_{n}t_{n}}(-1)^{k-1}E_{k-1}^{n-1}(r_{1}t_{1}, \ldots, r_{n-1}t_{n-1}) \frac{1+ p^{n}_{k}(r)}{p^{n}_{k}(r)}d\eta(r).	
 $$
 More generally, for any $L \subset \{1,\ldots, n\}$ with $|L|=l\leq k$, we have that 
 \begin{equation}\label{firstexamplepdiLform}
 	[\partial^{\vec{1}_{L}}g](t)=\int_{[0, \infty)^{n}\setminus{\partial_{k}^{n}}}  \left [\prod_{i \in L}r_{i}e^{-r_{i}t_{i}} \right ] (-1)^{k-l}E_{k-l}^{n-l} (r_{L^{c}} \odot t_{L^{c}}) \frac{1+p^{n}_{k}(r)}{p^{n}_{k}(r)}d\eta(r).	
 \end{equation}
 This occurs because applying Equation \ref{inducksympol}  recursively,  we obtain the  following expression for  the derivatives in the variable $t$
 \begin{equation}\label{firstexamplepdiL}
 	\partial^{\vec{1}_{L}}E_{k}^{n}(r_{1}t_{1}, \ldots, r_{n}t_{n})=\left [\prod_{i \in L}r_{i}e^{-r_{i}t_{i}} \right ] (-1)^{l} E_{k-l}^{n-l} (r_{L^{c}} \odot t_{L^{c}}), 
 \end{equation}	
 then by using in order the Equations \ref{ex1compordkinn}  and \ref{ineqdifforderkradial}
 \begin{align*}
 	0 &\leq \left[\prod_{i \in L}r_{i}e^{-r_{i}t_{i}} \right ](-1)^{k-l} E_{k-l}^{n-l} (r_{L^{c}} \odot t_{L^{c}})\frac{p^{n}_{k}(\vec{1}+r)}{p^{n}_{k}(r)} \\
 	&\leq 2^{2n+k-3l}p_{k-l}^{n-l}(\vec{1} + t_{L^{c}})\left [\prod_{i \in L}r_{i}e^{-r_{i}t_{i}} \right ]\frac{p^{n-l}_{k-l}(r_{L^{c}})}{p^{n-l}_{k-l}(\vec{1} + r_{L^{c}})}  \frac{p^{n}_{k}(\vec{1} + r)}{p^{n}_{k}(r)}\\
 	&\leq 2^{2n+k-3l}\binom{n }{k}p_{k-l}^{n-l}(\vec{1} + t_{L^{c}})\left [\prod_{i \in L}(1+r_{i})e^{-r_{i}t_{i}} \right ],
 \end{align*}
 which is locally  $\eta$ integrable for any $t \in (0, \infty)^{n}$.
 In particular, when $|L|=k$
 $$
 [\partial^{\vec{1}_{L}}g](t_{1},\ldots, t_{n})= \int_{[0, \infty)^{n}\setminus{\partial_{k}^{n}}}  e^{-r\cdot t}   \left [\prod_{i \in L}r_{i} \right ] \frac{p^{n}_{k}(\vec{1}+r)}{p^{n}_{k}(r)}d\eta(r), 
 $$
 which is clearly completely monotone in $(0, \infty)^{n}$ by Theorem \ref{Bochnercomplsevndim}. It is interesting to point out that $\partial^{\vec{1}_{L}}g$ is a nonnegative function for every $|L|\leq k$

 For a  second class of examples for Bernstein functions of order $k$ in $(0, \infty)^{n}$, let $F \subset \{1, \ldots, n\}$, $|F|=k$, and suppose that $\psi^{F}: [0, \infty)^{k} \to \mathbb{R}$ is a continuous Bernstein function of order $k$ in $(0,\infty)^{k}$ that is zero in $\partial_{k-1}^{k}$. Then 
 $$
 (t_{1}, \ldots, t_{n}) \to \psi^{F}(t_{F}), 
 $$
 is a Bernstein function of order $k$ in $(0,\infty)^{n}$. Indeed, the previous function is an element of $C^{\infty}((0, \infty)^{n})$, and for any $L\subset \{1, \ldots, n\}$, $|L|=k$, we have that $\partial^{\vec{1}_{L}}\psi^{F}=0$ if $L \neq F$  and    the function  $\partial^{\vec{1}_{F}}\psi^{F}$ is completely monotone by the hypothesis.

 \begin{theorem}\label{bernsksevndim} Let $n> k\geq 1 $ $g:[0, \infty)^{n} \to \mathbb{R}$ be a continuous function such that  $g(t)=0$ for every $t\in \partial_{k-1}^{n}$. Then  the function $g$    is  a Bernstein function of order $k$  in $(0, \infty)^{n}$  if and only if it can be represented as
 	\begin{align*}   
 		g(t)&
 		=  \sum_{|F|=k }\psi^{F}(t_{F})+\int_{[0, \infty)^{n}\setminus{\partial_{k}^{n}}}  	(-1)^{k}	E_{k}^{n}( r\odot t)  \frac{p^{n}_{k}(\vec{1}+r)}{p^{n}_{k}(r)}d\eta(r)
 	\end{align*}
 	where  the measure $\eta \in \mathfrak{M}([0,\infty)^{n}\setminus{\partial_{k}^{n}})$ is nonnegative and  the functions $\psi^{F}: [0, \infty)^{k} \to \mathbb{R}$  are continuous Bernstein functions of order $k$  in $(0, \infty)^{k}$ that are zero in $\partial_{k-1}^{k}$. Further, the representation is unique.
 \end{theorem}
 
 \begin{proof}The converse follows by the comments made before the statement of the Theorem.\\ 
 	Now, suppose that $g$    is  a Bernstein function of order $k$  in $(0, \infty)^{n}$, then by Theorem \ref{Bochnercomplsevndim} for any $F \subset \{1, \ldots, n\}$, $|F|=k$, there exists a nonnegative measure $\eta_{F}$, not necessarily finite, but for which
 	$$
 	[\partial^{\vec{1}_{F}}g](t)= \int_{[0, \infty)^{n}}e^{-r\cdot t} d\eta_{F}(r). 
 	$$
 	Now, let arbitrary disjoint subsets $L, M, N $ of $\{1, \ldots , n\}$  for which $|L\cup M|=|L\cup N|=k$. Since  $\partial^{\vec{1}_{M}}[\partial^{\vec{1}_{L \cup N}}g]=\partial^{\vec{1}_{N}}[\partial^{\vec{1}_{L \cup M}}g]$  and the representation in Theorem \ref{Bochnercomplsevndim} is unique, we get that
 	\begin{equation}\label{eq1bernsksevndim}
 		\left [\prod_{i\in M}r_{i} \right ] d\eta_{L \cup N}= \left [\prod_{i\in N}r_{i} \right ]d\eta_{L \cup M}, 
 	\end{equation}
 	multipliying both sides with $\prod_{i\in L}r_{i}$, we reach to the $\binom{n}{k}[\binom{n}{k}-1]/2$ (without repetition) equalities
 	\begin{equation}\label{eq1bernsksevndim2}
 		\left [\prod_{i\in \mathcal{F}}r_{i} \right ]d\eta_{F}= \left 	[\prod_{i\in F}r_{i} \right ] d\eta_{\mathcal{F}}, \quad |F|= |\mathcal{F}|=k
 	\end{equation}
 	by taking $L:=F\cap \mathcal{F}$, $M:= \mathcal{F}\setminus{L}$ and $N:= F\setminus{L}$.\\
 	For $W \subset \{1, \ldots, n\}$, we define the regions 
 	$$
 	A^{+}_{W}:= \{ r \in [0, \infty)^{n}, \quad r_{i}>0 \text{ for all } i \in W \text{ and } r_{i}=0  \text{ for all } i \notin W\},$$
 	where $A_{\emptyset}^{+}=\{0\}$. We affirm that if $|W|\leq k $ and $W$ is not a subset of $F$ then $\eta_{F}(A^{+}_{W})=0$. Indeed, under these hypothesis the sets $W \setminus{F}$ and $F\setminus{W}$ are nonempty, take arbitrary nonempty $M \subset W \setminus{F} $ and $N\subset F\setminus{W}$ with the only restriction that $|M|=|N|$ and define $L:= F\setminus{N}$. Then $L, M, N$ are disjoint subsets of $\{1, \ldots, n\}$ such that $|L\cup M|=|L\cup N|= |F|=k$ and by Equation  \ref{eq1bernsksevndim} we obtain that 
 	$$
 	\int_{A^{+}_{W}}\left [\prod_{i\in M}r_{i} \right ] d\eta_{F}(r)= \int_{A^{+}_{W}} \left [\prod_{i\in N}r_{i} \right ]d\eta_{L \cup M}(r)=0
 	$$
 	because $\prod_{i\in N}r_{i} $ is the zero function in $A^{+}_{W}$,	hence, $\eta_{F}(A^{+}_{W})=0$.  Since
 	$$
 	\partial_{k}^{n}= \bigcup_{|W|\leq k}A^{+}_{W}
 	$$
 	and
 	$$ 
 	A_{F}:=\{ r \in [0, \infty)^{n}, \quad r_{i}=0 \text{ for all } i  \notin F\}= \bigcup_{W \subset F}A^{+}_{W}
 	$$
 	we obtain that the measure $\eta_{F}$  in $\partial_{k}^{n}$ is supported on  its subset $A_{F}$. Analogously,  since
 	$$
 	[0, \infty)^{n} \setminus{\partial_{k}^{n}}= \bigcup_{|W|\geq k+1}A^{+}_{W}
 	$$
 	and 
 	$$ 
 	A^{F}:=\{ r \in [0, \infty)^{n} \setminus{\partial_{k}^{n}}, \quad r_{i}>0 \text{ for all } i  \in F\}= \bigcup_{|W| \geq k+1, F \subset  W}A^{+}_{W},                  
 	$$
 	and  defining the sets $L,M,N$  similarly, we obtain that if $|W|\geq  k+1 $ and $W$ does not contain $F$ then $\eta_{F}(A^{+}_{W})=0$, hence  the measure $\eta_{F}$  in $[0, \infty)^{n} \setminus{\partial_{k}^{n}}$ is supported on  its subset $A^{F}$. Thus, the  measure $\eta_{F}$  in $[0, \infty)^{n}$ is supported on  its subset $A_{F}\cup A^{F}$.\\
 	Now, we  define the nonnegative measure  $\eta$ in $[0, \infty)^{n}\setminus{\partial_{k}^{n}} $,  where in each region $A^{F}$ is defined as 
 	$$
 	d\eta(r):= \frac{p^{n}_{k}(r)}{p^{n}_{k}(\vec{1} +r) }\frac{1}{\prod_{i \in F}r_{i}}d\eta_{F} (r), \quad r \in A^{F}.  
 	$$ 
 	Due to Equation \ref{eq1bernsksevndim2}, the definition of $\eta$ in the region $A^{F}$ and the one in $A^{\mathcal{F}}$ are the same on the set  $A^{F}\cap A^{\mathcal{F}}$, hence,  $\eta$ is well defined  in $[0, \infty)^{n}\setminus{\partial_{k}^{n}} $.\\
 	Now, our aim is to prove that $\eta$ is a finite measure. To reach this conclusion,  we use  a similar  procedure in the first steps of the converse of the proof of Theorem \ref{bernssevndim}. Indeed, for any $|F|=k$, by the Fundamental Theorem of calculus applied  in the variables with index in $F$, the function  
 	$$
 	g_{F}(t):= \sum_{L \subset F}(-1)^{|L|}g(t - t_{L})
 	$$
 	where $t_{L}:= \sum_{i\in L}t_{i}e_{i} \in [0, \infty)^{n} $, admits the following integral representation
 	\begin{align*}
 		g_{F}(t)& = \int_{\prod_{i\in F}[0,t_{i}]}[\partial^{\vec{1}_{F}}g](s_{F},t_{F^{c}})d(s_{F}) =\int_{[0, \infty)^{n}}\left [\prod_{i \in F}\frac{1-e^{-r_{i}t_{i}}}{r_{i}}\right ]e^{-r_{F^{c}}\cdot t_{F^{c}}}d\eta_{F}(r),
 	\end{align*}
 	which implies that    the measure $(\prod_{ i \in F}1/(1+r_{i}))d\eta_{F}(r)
 	$ is finite in $[0, \infty)^{n}$, because by  Equation \ref{bern1ineq}, the hypothesis that $g(t)=0$ for every $t\in \partial_{k-1}^{n}$  and
 	\begin{align*}
 		\int_{[0, \infty)^{n}} \prod_{i\in F}\frac{1}{1+r_{i}}d\eta_{F}(r) &\leq \int_{[0, \infty)^{n}} \prod_{i\in F}\frac{1-e^{-r_{i}}}{r_{i}}d\eta_{F}(r)=g_{F}(\vec{1}_{F})=g(\vec{1}_{F}). 
 	\end{align*}
 	To conclude that the measure $\eta$ is finite, we define the new regions  for $|W|=k$
 	$$
 	B^{W}:= \{ r \in  [0, \infty)^{n}\setminus{\partial_{k}^{n}}, \quad \min \{r_{i}, \quad i \in W\} \geq \max \{r_{j}, \quad j \in W^{c} \} \},
 	$$
 	whose union is also $[0, \infty)^{n}\setminus{\partial_{k}^{n}}$. Since   $B^{F} \subset A^{F}$, we have that for every $r \in B^{F}$  
 	$$
 	0\leq \frac{p^{n}_{k}(r)}{p^{n}_{k}(\vec{1}+r) }\frac{ \prod_{ i \in F}(1+r_{i})}{\prod_{ i \in F}r_{i}}\leq  \frac{p^{n}_{k}(\vec{1})\prod_{ i \in F}r_{i}}{p^{n}_{k}(\vec{1}+r) }\frac{ \prod_{ i \in F}(1+r_{i})}{\prod_{ i \in F}r_{i}}\leq p^{n}_{k}(\vec{1}),
 	$$
 	thus, we conclude that $\eta$ is finite, because 
 	$$
 	\eta(B^{F}) =\int_{ B^{F} }\frac{p^{n}_{k}(r)}{p^{n}_{k}(\vec{1} +r) }\frac{1}{\prod_{i \in F}r_{i}}d\eta_{F} (r) \leq p^{n}_{k}(\vec{1})\int_{ B^{F} }\frac{1}{\prod_{i \in F}(1+r_{i})}d\eta_{F} (r)< \infty. 
 	$$
 	as	the measure $(\prod_{ i \in F}1/(1+r_{i}))d\eta_{F}(r)$ is finite in $[0, \infty)^{n}$.\\ 
 	We define the measure $\lambda _{F}$   in $[0, \infty)^{k}$ using  the canonical injection of $[0, \infty)^{k}$ into $A_{F}$  applied to  the measure  $(\prod_{\ell \in F}1/(1+r_{\ell}))d\eta_{F}(r)$.  Note that $\lambda _{F}$ is finite and nonnegative because the measure $(\prod_{ i \in F}1/(1+r_{i}))d\eta_{F}(r)
 	$ is finite and nonnegative in $[0, \infty)^{n}$. 
 	Now, from the first part of the proof, the function
 	\begin{align*}
 		h(t)&:= \sum_{|F|=k} \int_{[0, \infty)^{k}} \prod_{i \in F} \left [(1-e^{-r_{i}t_{i}})\frac{(1+r_{i})}{r_{i}} \right ]d\lambda _{F}(r_{F})  + \int_{ [0, \infty)^{n}\setminus{\partial_{k}^{n}} }  	(-1)^{k}	E_{k}^{n}( r\odot t)   \frac{1+p^{n}_{k}(r)}{p^{n}_{k}(r)}d\eta(r)
 	\end{align*}
 	is a well defined Bernstein function of order $k$ in $(0, \infty)^{n}$ that is zero in $\partial_{k-1}^{n}$. By separating it in the regions we have that for any $\mathcal{F} \subset \{1, \ldots, n\}$ with $|\mathcal{F}|=k$
 	\begin{align*}
 		[\partial^{\vec{1}_{\mathcal{F}}}h](t)&= \int_{[0, \infty)^{k}}e^{-r_{\mathcal{F}}\cdot t_{\mathcal{F}} }\left [ \prod_{i \in \mathcal{F}} (1+r_{i})   \right ] d\lambda_{\mathcal{F}}(r_{\mathcal{F}}) +\int_{[0, \infty)^{n}\setminus{\partial_{k}^{n}}}e^{-r\cdot t} \left [ \prod_{i \in \mathcal{F}} r_{i}\right ] \frac{p^{n}_{k}(\vec{1} + r)}{p^{n}_{k}(r)}d\eta(r)  \\
 		&=  \int_{A_{\mathcal{F}}}e^{-r_{\mathcal{F}}\cdot t_{\mathcal{F}} } d\eta_{\mathcal{F}}(r) +\int_{A^{\mathcal{F}}} e^{-r\cdot t} \left [ \prod_{i \in \mathcal{F}} r_{i}\right ] \frac{p^{n}_{k}(\vec{1} +r)}{p^{n}_{k}(r)}d\eta(r)\\
 		&=\int_{A_{\mathcal{F}} \cup A^{\mathcal{F}}}e^{-r\cdot t} d\eta_{\mathcal{F}}(r)=\int_{[0, \infty)^{n}}e^{-r\cdot t} d\eta_{\mathcal{F}}(r).
 	\end{align*}  
 	Hence, we proved that for any any $\mathcal{F} \subset \{1, \ldots, n\}$ with $|\mathcal{F}|=k$  and $t \in (0, \infty)^{n}$
 	$$
 	[\partial^{\vec{1}_{\mathcal{F}}}h](t)= \int_{[0, \infty)^{n}}e^{-r\cdot t}d\eta_{\mathcal{F}}(r)= [\partial^{\vec{1}_{\mathcal{F}}}g](t).
 	$$
 	To conclude the proof, note that if $p(t):= h(t)-g(t)$, then $p$ is zero at $\partial_{k-1}^{n}$ and satisfies the functional equations
 	$$
 	p_{F}(t):= \sum_{L \subset F} (-1)^{|L|}p(t -t_{L})= \int_{\bigtimes_{i=1}^{n}[0, t_{i}]}\partial^{\vec{1}_{\mathcal{F}}}[p](s)ds=  0, \quad |F|=k, \quad t \in [0, \infty)^{n}.
 	$$
 	Since $p_{F}(t_{F})= p(t_{F})$, we obtain that  $p$ is zero at $\partial_{k}^{n}$. The rest of the proof follows by recursion, as if  $p$ is zero at $\partial_{\ell}^{n}$, then for any subset $O\subset \{1 \ldots, n\}$ with $|O|=\ell+1$ we have that $p_{F}(t_{O})= p(t_{O})$ whenever $F \subset O$. \\ 
 	The representation is unique, because all information of the measures $\eta$ and  $\lambda_{F}$  came from the measures $\eta_{F}$, which are unique due to Theorem \ref{Bochnercomplsevndim}, hence the representation is also unique in this context. \end{proof}

 \begin{remark}\label{rembernkn}
 	The integrand  in Equation \ref{berskex} cannot be continuously extended to $[0, \infty)^{n} \setminus{\partial_{k-2}^{n}}$, when $n> k \geq 1$. Indeed, let $r=(r_{1}, \ldots, r_{k-1}, r_{k}, 0, \ldots, 0) \in (0, \infty)^{n}$, then 
 	$$
 	\lim_{r_{k} \to 0}(-1)^{k}E_{k}^{n}( r\odot t)  \frac{1+p^{n}_{k}(r)}{p^{n}_{k}(r)}= t_{k}\frac{1+ \prod_{i=1}^{k-1}r_{i}}{\prod_{i=1}^{k-1}r_{i}} \prod_{i=1} ^{k-1}(1-e^{-r_{i}t_{i}}).  
 	$$
 	On the other hand, for $r^{\prime}=(r_{1}, \ldots, r_{k-1},0,r_{k+1}, 0, \ldots, 0) \in (0, \infty)^{n}$, then 
 	$$
 	\lim_{r_{k+1} \to 0} (-1)^{k}E_{k}^{n}(r^{\prime}\odot t)  \frac{1+p^{n}_{k}(r^{\prime})}{p^{n}_{k}(r^{\prime})}=t_{k+1}\frac{1+ \prod_{i=1}^{k-1}r_{i}}{\prod_{i=1}^{k-1}r_{i}} \prod_{i=1} ^{k-1}(1-e^{-r_{i}t_{i}}). 
 	$$
 	However, the integrand in Equation \ref{berskex} can be extended to $[0, \infty)^{n} \setminus{\partial_{k-1}^{n}}$, but we use the version in  Equation \ref{berskex} as the expression for the representation  of Bernstein functions of order $k$  in $(0, \infty)^{n}$ are considerably simplified. For the same reason, we are not considering the case $n=k$ in Theorem  \ref{bernsksevndim} , as the representation in Theorem \ref{bernssevndim} is more simpler then the approach of   extending to $[0, \infty)^{n} \setminus{\partial_{k-1}^{k}}$  the Equation \ref{berskex}.   
 \end{remark}

 Next result provides the most important result in this text, the characterization of the PDI$_{k,n}^{\infty}$ functions. 
 
 \begin{theorem}\label{bernsksevndimpart3} Let $n> k\geq 1 $, $g:[0, \infty)^{n} \to \mathbb{R}$ be a continuous function such that  $g(t)=0$ for every $t\in \partial_{k-1}^{n}$. The following conditions are equivalent:
 	\begin{enumerate}
 		\item [$(i)$] For any $d\in \mathbb{N}$ and discrete  measures  $\mu_{i}$ in $\mathbb{R}^{d}$, $1\leq i \leq n$, and with the restriction that $|i, \quad \mu_{i}(\mathbb{R}^{d})=0| \geq k $,  it holds that
 		$$
 		\int_{(\mathbb{R}^{d})_{n}}\int_{ (\mathbb{R}^{d})_{n}}(-1)^{k}g(\|x_{1}-y_{1}\|^{2}, \ldots, \|x_{n} - y_{n}\|^{2})d[\bigtimes_{i=1}^{n}\mu_{i}](x)d[ \bigtimes_{i=1}^{n}\mu_{i}](y)\geq 0.
 		$$
 		\item [$(ii)$] For any $d\in \mathbb{N}$ and discrete probability $P$ in $(\mathbb{R}^{d})_{n}$,  it holds that
 		$$
 		\int_{(\mathbb{R}^{d})_{n}}\int_{ (\mathbb{R}^{d})_{n}}g(\|x_{1}-y_{1}\|^{2}, \ldots, \|x_{n} - y_{n}\|^{2})d[\Lambda_{k}^{n}[P] ](x)d[\Lambda_{k}^{n}[P]](y)\geq 0.
 		$$
 		\item [$(iii)$]  The function $g$ is PDI$_{k,n}^{\infty}$.
 		\item [$(iv)$]  The function $g$ can  be represented as
 		\begin{align*}
 			g(t) &=  \sum_{|F|=k } \psi^{F}(t_{F}) +\int_{[0, \infty)^{n}\setminus{\partial_{k}^{n}}}	(-1)^{k}	E_{k}^{n}( r \odot t)  \frac{p^{n}_{k}(r + \vec{1})}{p^{n}_{k}(r)}d\eta(r)
 		\end{align*}
 		where  the measure $\eta \in \mathfrak{M}([0,\infty)^{n}\setminus{\partial_{k}^{n}})$ is  nonnegative and  the functions $\psi^{F}: [0, \infty)^{k} \to \mathbb{R}$ are continuous Bernstein functions of order $k$ in $(0, \infty)^{k}$ that are zero on the set  $\partial_{k-1}^{k}$.  Further, the representation is unique.
 		\item[$(v)$]The function $g$  is a Bernstein function of order $k$ in $(0, \infty)^{n}$.
 	\end{enumerate}	
 \end{theorem}

 \begin{proof}The equivalence between  relation $(iv)$ and $(v)$ is proved in Theorem \ref{bernsksevndim}.\\
 	Relation $(iv)$  implies relation $(iii)$ because  due to Equation \ref{exponentialknmeas} and Equation \ref{integmu0n}, for any $\mu \in  \mathcal{M}_{k}((\mathbb{R}^{d})_{n})$
 	\begin{align*}
 		&\int_{(\mathbb{R}^{d})_{n}}\int_{ (\mathbb{R}^{d})_{n}}(-1)^{k}g(\|x_{1}-y_{1}\|^{2}, \ldots, \|x_{n} - y_{n}\|^{2})d\mu(x)d \mu(y)\\
 		&= \int_{[0, \infty)^{n}\setminus{\partial_{k}^{n}}}  \left [\int_{(\mathbb{R}^{d})_{n}}\int_{ (\mathbb{R}^{d})_{n}}e^{-\sum_{i=1}^{n}r_{i}\|x_{i}-y_{i}\|^{2}} d\mu(x)d \mu(y) \right ]\frac{p^{n}_{k}(r + \vec{1})}{p^{n}_{k}(r)}d\eta(r) \\
 		&+ 	\sum_{|F|=k } \int_{(\mathbb{R}^{d})_{n}}\int_{ (\mathbb{R}^{d})_{n}} (-1)^{k}\psi^{F}((\|x_{1}-y_{1}\|^{2}, \ldots, \|x_{n} - y_{n}\|^{2})_{F})d\mu_{F}(x_{F}) d\mu_{F}(y_{F}) \geq 0 
 	\end{align*}		
 	where $\mu_{F}(\prod_{i \in F}A_{i}):= \mu (\prod_{i=1}^{n}A_{i})$, where $A_{i}=\mathbb{R}^{d}$ for $i \notin F$ and this measure belongs to $ \mathcal{M}_{k}((\mathbb{R}^{d})_{k})$. 	 	
 	Relation $(iii)$ implies relation $(ii)$ by the property $i)$ in Theorem \ref{generallancaster}.
 	Relation $(ii)$ implies relation $(i)$ by Lemma \ref{genlancastercartesianproduct}.	\\	
 	We conclude the proof by showing that relation $(i)$ implies relation $(v)$. Indeed, for a fixed $k\geq 1$ the proof  is done by induction on $n$, where the initial case is when $n=k$ proved in Theorem \ref{basicradialndim}. Then, we assume that the  result holds for all possible values of $n^{\prime}< n$ with the same fixed $k$, more precisely a consequence of this fact,   that  the functions 
 	$$
 	(t_{1}, \ldots, t_{n}) \to g(t_{L} ) \in \mathbb{R}, \quad L\subset \{1, \ldots, n\}, \quad |L|< n
 	$$
 	are elements of  $C^{\infty}((0, \infty)^{n})$. By the chain rule, it is sufficient to prove that  $u \in [0, \infty)^{|L|} \to g(u_{L} ) $ is an element of $C^{\infty}((0, \infty)^{|L|}))$. When $|L|<k$ these are the zero functions.  When $|L|=k$, by taking  $\mu_{i}:=\delta_{\vec{0}}$ for $i \notin L$ in the hypothesis of relation $(i)$, we obtain that the function $u\in [0, \infty)^{k} \to g(u_{L} ) $ satisfy the requirements of relation $(i)$  in Theorem  \ref{basicradialndim}. Similarly, if  $k < |L| < n$, by taking  $\mu_{i}:=\delta_{\vec{0}}$ for $i \notin L$ in the hypothesis of relation $(i)$, we obtain that the function $u\in [0, \infty)^{|L|} \to g(u_{L} ) $ satisfy the requirements of relation $(i)$ in Theorem  \ref{bernsksevndimpart3} with $|L|$ variables.\\
 	Now, our final objective is to prove that for any  $F\subset \{1, \ldots, n\}$ with  $|F|=k$, the function 
 	$$
 	g_{F}(t):= \sum_{L \subset F}(-1)^{|L|}g(t - t_{L})
 	$$
 	admits  the following integral representation
 	\begin{align*}
 		g_{F}(t)&  =\int_{[0, \infty)^{n}}\left [\prod_{i \in F}(1-e^{-r_{i}t_{i}})\frac{1+r_{i}}{r_{i}}\right ]e^{-r_{F^{c}}\cdot t_{F^{c}}}d\eta_{F}(r),
 	\end{align*}
 	for some nonnegative $\eta_{F} \in \mathfrak{M}([0, \infty)^{n})$. Note that once proved this representation,   we obtain that the function $	g_{F} $ is an element of $C^{\infty}((0, \infty)^{n }))$  and
 	$$
 	[\partial^{\vec{1}_{F}}g_{F}](t)= \int_{[0, \infty)^{n}}\left [\prod_{i \in F}e^{-r_{i}t_{i}}(1+r_{i})\right ]e^{-r_{F^{c}}\cdot t_{F^{c}}}d\eta_{F}(r) 
 	$$ 
 	which is a completely monotone function by Theorem \ref{Bochnercomplsevndim}.  This concludes the proof as we obtain that $g$ is an element of $C^{\infty}((0, \infty)^{n }))$  (because all other functions in the definition of $g_{F}$  are   elements of $C^{\infty}((0, \infty)^{n }))$) and that $g$   is a Bernstein function of order $k$ in $(0, \infty)^{n}$ because $[\partial^{\vec{1}_{F}}g ](t)= [\partial^{\vec{1}_{F}}g_{F}](t)$.\\
 	We prove the integral representation for the case  $F= \{n-k+1, \ldots, n\}$, being the others an adaptation in the notation. We make a slightly change in the notation for $g_{F}$ to have a clear distinction between the role of variables in $F$ and  the ones in $F^{c}$  (a similar approach is done in the proof of Theorem \ref{PDseveralvariables})
 	$$
 	g_{F}(t, s):= \sum_{L \subset F}(-1)^{k-|L|}g(t, s_{L}), \quad t=(t_{1}, \ldots, t_{n-k}) \in [0, \infty)^{n-k}, s=(s_{1}, \ldots, s_{k}) \in   [0, \infty)^{k}.
 	$$
 	For any fixed $s \in   [0, \infty)^{k}$, the function $ t \in [0, \infty)^{n-k}\to g_{F}(t, s) $ is continuous and  satisfies the requirements for relation $(i)$ in Theorem \ref{PDseveralvariables}, by fixing the $k$ measures  $\mu_{n-k+i}:=[\delta_{x_{n-k+i}} -  \delta_{y_{n-k+i}}]/2$, for arbitrary $x_{n-k+i}, y_{n-k+i} \in \mathbb{R}^{d}$ such that $\|x_{n-k+i} - y_{n-k+i}\|^{2}=s_{i}$, for  $1\leq i \leq k$. Thus, by the same Theorem, for any fixed $s \in [0, \infty)^{k}$ there is an unique  nonnegative measure $\eta_{s} \in \mathfrak{M}([0, \infty)^{n-k})$ for which
 	$$
 	g_{F}(t, s) = \int_{[0, \infty)^{n-k}}e^{-r\cdot t}d\eta_{s}(r).
 	$$
 	We affirm that the collection of measures  $\eta_{s}$,  $s \in [0, \infty)^{k}$,  satisfies all the conditions of Lemma \ref{techsimple}. Indeed,  note that $\eta_{s}$ is the zero measure for any $s \in \partial_{k-1}^{k}$, because     for any $s \in [0, \infty)^{k}$
 	$$
 	g(\vec{0}, s)= g_{F}(\vec{0}, s)= \int_{[0, \infty)^{n-k}}d\eta_{s}(r)= \eta_{s}([0, \infty)^{n-k}),   
 	$$
 	and $(\vec{0}, s) \in   \partial_{k-1}^{n}$ if $s \in \partial_{k-1}^{k}$. This equality also implies the continuity of the function $s \to \eta_{s}([0, \infty)^{n-k})$.\\
 	Using Theorem \ref{basicradialndim} and the hypothesis, for any $d \in \mathbb{N}$ and $\mu\in  \mathcal{M}_{k}((\mathbb{R}^{d})_{k})$, the function
 	$$
 	t  \to \int_{(\mathbb{R}^{d})_{k}}\int_{ (\mathbb{R}^{d})_{k}}(-1)^{k}g(t, \|x_{n+1-k}- y_{n+1-k} \|^{2}, \ldots, \|x_{n}- y_{n} \|^{2})d\mu(x)d \mu(y)
 	$$
 	which by Equation \ref{integmu0n} can be rewritten as
 	$$
 	\int_{(\mathbb{R}^{d})_{k}}\int_{ (\mathbb{R}^{d})_{k}}(-1)^{k}g_{F}(t, \|x_{n+1-k}- y_{n+1-k} \|^{2}, \ldots, \|x_{n}- y_{n} \|^{2})d\mu(x)d \mu(y)
 	$$
 	is continuous and satisfies relation $(i)$ in Theorem \ref{PDseveralvariables}.  Since the representation in Theorem  \ref{PDseveralvariables} is unique, we obtain that for any $A \in \mathscr{B}([0, \infty )^{n-k})$ the function $t \in [0, \infty)^{k} \to \eta_{t}(A)$  defines  an PDI$_{k}^{\infty}$ radial kernel in all Euclidean spaces, Lemma \ref{techsimple} concludes that the desired integral representation is valid. \end{proof}

 As a direct consequence of the previous Theorem, we obtain that on a function $g$ that satisfies the requirements and the equivalences in Theorem \ref{bernsksevndimpart3}, its growth is delimited by all the  values of the function with $k$ variables.

 \begin{corollary}\label{growthpdiknvar} Let $n> k\geq 1 $ $g:[0, \infty)^{n} \to \mathbb{R}$ be a continuous function such that  $g(t)=0$ for every $t\in \partial_{k-1}^{n}$ an that satisfies the equivalences in Theorem \ref{bernsksevndimpart3}. Then,  $g$ nonnegative and is increasing, in the sense that  $g(t_{\vec{2}}) \geq  g(t_{\vec{1}}) $ if $t_{\vec{2}} - t_{\vec{1}} \in [0, \infty)^{n}$. Also, it holds that
 	$$
 	\binom{n}{k}^{-1} \sum_{|F|=k}g(t_{F}) \leq g(t) \leq  \sum_{|F|=k}g(t_{F}), \quad t \in [0, \infty)^{n}.
 	$$

 \end{corollary}

 \begin{proof}As a consequence of Equation \ref{firstexamplepdiLform}, the function  $\partial^{e_{i}}g$ is a  nonnegative in $(0, \infty)^{n}$ for every $1\leq i \leq n$, thus $g$ is an increasing function when restricted to the domain  $(0, \infty)^{n}$.  By the  continuity of $g$ in $[0, \infty)^{n}$, we obtain that it is also increasing   in the whole domain $[0, \infty)^{n}$. In particular, this implies that $g$ is nonnegative as $g(t)\geq g(\vec{0})=0$.\\
 	For the second part, since 	
 	$$
 	g(t_{F}) = \psi^{F}(t_{F}) +  \int_{[0, \infty)^{n}\setminus{\partial_{k}^{n}}} \left [\prod_{i \in F}(1-e^{-r_{i}t_{i}}) \right ] \frac{p^{n}_{k}(\vec{1} + r)}{p^{n}_{k}(r)}d\eta(r_{1},\ldots,  r_{n}),
 	$$	
 	Equation \ref{ineqorderk2eq1}  implies that
 	\begin{align*}   
 		&\sum_{ |F|=k } g(t_{F})
 		=  \sum_{|F|=k } \psi^{F}(t_{F})+ \int_{[0, \infty) ^{n}\setminus{\partial_{k}^{n}}}  \sum_{ |F|=k }\left [\prod_{i \in F}(1-e^{-r_{i}t_{i}}) \right ] \frac{p^{n}_{k}(\vec{1} + r)}{p^{n}_{k}(r)}d\eta(r_{1},\ldots,  r_{n}) \\
 		&=\sum_{|F|=k } \psi^{F}(t_{F})+ \int_{[0, \infty) ^{n}\setminus{\partial_{k}^{n}}}  p_{k}^{n}(1-e^{-r_{1}t_{1}}, \ldots,1-e^{-r_{n}t_{n}} ) \frac{p^{n}_{k}(\vec{1} + r)}{p^{n}_{k}(r)}d\eta(r_{1},\ldots,  r_{n}) \\
 		&\leq  \binom{n}{k} \left [  \sum_{|F|=k } \psi^{F}(t_{F})+ \int_{[0, \infty) ^{n}\setminus{\partial_{k}^{n}}}  (-1)^{k}E_{k}^{n}(r\odot t) \frac{p^{n}_{k}(\vec{1} + r)}{p^{n}_{k}(r)}d\eta(r_{1},\ldots,  r_{n})   \right ] = \binom{n}{k} g(t).
 	\end{align*}
 	The other inequality also follows from Equation \ref{ineqorderk2eq1} as
 	
 	\begin{align*}   
 		g(t) & \leq   	\sum_{|F|=k } \psi^{F}(t_{F}) + \int_{[0, \infty) ^{n}\setminus{\partial_{k}^{n}}}  p_{k}^{n}(1-e^{-r_{1}t_{1}}, \ldots,1-e^{-r_{n}t_{n}} ) \frac{p^{n}_{k}(\vec{1} + r)}{p^{n}_{k}(r)}d\eta(r_{1},\ldots,  r_{n})  \\
 		&=\sum_{ |F|=k } g(t_{F}).
 \end{align*}\end{proof}

 As $\Lambda_{2}^{n}[P]= P - \bigtimes_{i=1}^{n}P_{i}$, that is, of a independence test, we make the statement of Theorem \ref{bernsksevndimpart3} explicit for the case $k=2$ above.

 \begin{theorem}\label{berns2sevndimpart2}  Let $n \geq 2$ and $g : [0, \infty)^{n} \to \mathbb{R}$ be a continuous function such that $g(t)=0$ for every $t \in \partial^{n}_{1}$. The following conditions are equivalent:
 	\begin{enumerate}
 		\item [$(i)$] For any $d\in \mathbb{N}$ and discrete  measures  $\mu_{i}$ in $\mathbb{R}^{d}$, $1\leq i \leq n$, and with the restriction that for at least two of those measures $\mu_{i}(\mathbb{R}^{d})=0$,  it holds that
 		$$
 		\int_{(\mathbb{R}^{d})_{n}}\int_{(\mathbb{R}^{d})_{n}}g(\|x_{1}-y_{1}\|^{2}, \ldots, \|x_{n} - y_{n}\|^{2})d[\bigtimes_{i=1}^{n}\mu_{i}](x)d[ \bigtimes_{i=1}^{n}\mu_{i}](y)\geq 0.
 		$$
 		\item [$(ii)$] For any $d\in \mathbb{N}$ and discrete probability $P$ in $(\mathbb{R}^{d})_{n}$, with marginals $P_{i}$ in $\mathbb{R}^{d}$, it holds that
 		$$
 		\int_{(\mathbb{R}^{d})_{n}}\int_{ (\mathbb{R}^{d})_{n}}g(\|x_{1}-y_{1}\|^{2}, \ldots, \|x_{n} - y_{n}\|^{2})d[P - \bigtimes_{i=1}^{n}P_{i}](x)d[P - \bigtimes_{i=1}^{n}P_{i}](y)\geq 0.
 		$$
 		\item [$(iii)$]  The function $g$ is PDI$_{2, n}^{\infty}$.
 		\item [$(iv)$]  The function  $g$ can  be represented as
 		\begin{align*}
 			g(t)  &=  \sum_{i <  j } \psi^{i,j}(t_{i},t_{j}) +\int_{[0, \infty)^{n}\setminus{\partial_{2}^{n}}}  \left (e^{-r\cdot t} - \sum_{i=1}^{n}e^{-r_{i}t_{i}}+n-1 \right ) \frac{p^{n}_{2}(\vec{1} + r)}{p^{n}_{2}(r)}d\eta(r_{1},\ldots,  r_{n})
 		\end{align*}
 		where  the measure $\eta \in \mathfrak{M}([0,\infty)^{n}\setminus{\partial_{2}^{n}})$ is  nonnegative and  the functions $\psi^{i,j}: [0, \infty)^{2} \to \mathbb{R}$ are Bernstein functions of order $2$ that are zero in $\partial_{1}^{2}$.  Further, the representation is unique.
 		\item[$(v)$]The function $g$  is a Bernstein function of order $2$ in $(0, \infty)^{n}$.
 	\end{enumerate}	
 \end{theorem}

 The case $k=1$ in Theorem \ref{bernsksevndim}  is partially proved in \cite{Mirotin2013rohtua}, with a slightly different presentation, which is related to the observations made in Remark  \ref{rembernkn}, which together with Theorem \ref{bernsksevndimpart3} is the following result. 	
 
 \begin{corollary}\label{CNDseveralvariables} Let $g : [0, \infty)^{n} \to \mathbb{R}$ be a continuous function. The following conditions are equivalent:
 	\begin{enumerate}
 		\item [$(i)$] For any $d\in \mathbb{N}$ and discrete  measures  $\mu_{i}$ in $\mathbb{R}^{d}$, $1\leq i \leq n$, and with the restriction that at least one of those measures $\mu_{i}(\mathbb{R}^{d})=0$  it holds that
 		$$
 		\int_{(\mathbb{R}^{d})_{n}}\int_{ (\mathbb{R}^{d})_{n}}g(\|x_{1}-y_{1}\|^{2}, \ldots, \|x_{n} - y_{n}\|^{2})d[\bigtimes_{i=1}^{n}\mu_{i}](x)d[ \bigtimes_{i=1}^{n}\mu_{i}](y)\leq 0.
 		$$
 		\item [$(ii)$]  The kernel  	
 		$$
 		g(\|x_{1}-y_{1}\|^{2}, \ldots, \|x_{n} - y_{n}\|^{2}), \quad x_{i},y_{i} \in  \mathbb{R}^{d}
 		$$ is CND  for every  $d\in \mathbb{N}$.
 		\item [$(iii)$]  The function  can  be represented as
 		\begin{align*}
 			g(t) -g(0) =   \sum_{i=1}^{n}a_{i}t_{i}+ \int_{[0, \infty)^{n}\setminus{\{0\}}}  (1-e^{-r\cdot t})\frac{1+ \sum_{i=1}^{n}r_{i}}{\sum_{i=1}^{n}r_{i}}d\eta(r_{1},\ldots,  r_{n})
 		\end{align*}
 		where  the measure $\eta \in \mathfrak{M}([0,\infty)^{n}$ and the scalars $a_{i}$ are nonnegative. Further, the representation is unique.
 		\item[$(iv)$]The function $g$  is a Bernstein function of order $1$ in $(0, \infty)^{n}$.
 	\end{enumerate}	
 	
 \end{corollary}

 Next result describe when a kernel in Theorem \ref{basicradialndim}, Theorem \ref{PDseveralvariables} and  Theorem \ref{bernsksevndimpart3} is SPDI$_{k,n}^{\infty}$. We also obtain a crucial property, that in case $k=2$  the kernel is SPDI$_{2,n}^{\infty}$	if and only if it is and independence test for discrete probabilities with a double integration. For other values of $k$ is equivalent at describing if or not  $\Lambda_{k}^{n}[P]=0$, and on the case  $n=k$, we also show  that a  kernel is  SPDI$_{k}^{\infty}$	if and only if it can be used as a characterization for when the Streitberg interaction (but also as mentioned before the  Lancaster interaction) of a discrete probability is or not zero.

 \begin{corollary}\label{SPDPDIK} Let $n\geq k \geq 0$ and   $g:[0, \infty)^{n} \to \mathbb{R}$ be a continuous and PDI$_{k,n}^{\infty}$ function such that  $g(t)=0$ for every $t\in \partial_{k-1}^{n}$. Then, the function  is SPDI$_{k,n}^{\infty}$ if and only if the measure $\eta$ in its representation satisfy  $\eta((0,\infty)^{n})>0$.\\ 
 	Further, if $n\geq k\geq 2$ the kernel is SPDI$_{k,n}^{\infty}$ if and only if for every   discrete probability  $P$ in  $(\mathbb{R}^{d})_{n}$ 
 	$$
 	\int_{(\mathbb{R}^{d})_{n}}\int_{ (\mathbb{R}^{d})_{n}}g(\|x_{1}-y_{1}\|^{2}, \ldots, \|x_{n} - y_{n}\|^{2})d[\Lambda_{k}^{n}[P] ](x)d[\Lambda_{k}^{n}[P]](y)=0.
 	$$	
 	if and only if $\Lambda_{k}^{n}[P]=0$.\\
 	Also, if $k=n$ the kernel is SPDI$_{n,n}^{\infty}$ if and only if for every probability discrete probability  $P$ in  $(\mathbb{R}^{d})_{n}$, its Streitberg interaction $\Sigma[P]$ satisfy 
 	$$
 	\int_{(\mathbb{R}^{d})_{n}}\int_{(\mathbb{R}^{d})_{n}}(-1)^{n}g(\|x_{1}-y_{1}\|^{2}, \ldots, \|x_{n} - y_{n}\|^{2})d[\Sigma[P]](x)d[\Sigma[P]](y)=0.
 	$$ 
 	if and only if $\Sigma[P]=0$. 	
 \end{corollary}

 \begin{proof} If  $\eta((0,\infty)^{n})>0$, for any  nonzero $\mu \in \mathcal{M}_{k}((\mathbb{R}^{d})_{n})$ it holds that
 	\begin{align*}
 		&\int_{(\mathbb{R}^{d})_{n}}\int_{ (\mathbb{R}^{d})_{n}}(-1)^{k}g(\|x_{1}-y_{1}\|^{2}, \ldots, \|x_{n} - y_{n}\|^{2})d\mu(x)d\mu(y)\\
 		& \geq \int_{(0, \infty)^{n}} \left [  	\int_{(\mathbb{R}^{d})_{n}}\int_{ (\mathbb{R}^{d})_{n}}e^{-r_{1}\|x_{1}-y_{1}\|^{2} - \ldots -\|x_{n} - y_{n}\|^{2}r_{n}  }d\mu(x)d\mu(y) \right ] \frac{p_{k}^{n}(\vec{1} + r)}{p_{k}^{n}(r)}d\eta(r)>0
 	\end{align*}		
 	because the Kronecker product of SPD kernels (in this case, the Gaussians) is SPD.\\
 	For the converse, we pick nonzero discrete measures $\mu_{i} $ in 	$\mathbb{R}^{d}$, $1\leq i \leq n$,  for which $\mu_{i}(\mathbb{R}^{d})=0$, then $\mu := (\bigtimes_{i=1}^{n}\mu_{i}) \in \mathcal{M}_{n}((\mathbb{R}^{d})_{n}) \subset \mathcal{M}_{k}((\mathbb{R}^{d})_{n})$, and  by Equation \ref{integmu0n} we have that
 	\begin{align*}
 		&0< \int_{(\mathbb{R}^{d})_{n}}\int_{ (\mathbb{R}^{d})_{n}}(-1)^{k}g(\|x_{1}-y_{1}\|^{2}, \ldots, \|x_{n} - y_{n}\|^{2})d\mu(x)d\mu(y)\\
 		& = \int_{(0, \infty)^{n}} \left [  	\int_{(\mathbb{R}^{d})_{n}}\int_{ (\mathbb{R}^{d})_{n}}e^{-r_{1}\|x_{1}-y_{1}\|^{2} - \ldots -r_{n}\|x_{n} - y_{n}\|^{2}  }d\mu(x)d\mu(y) \right ] \frac{p_{k}^{n}(\vec{1} + r)}{p_{k}^{n}(r)}d\eta(r),
 	\end{align*}
 	as the inner double integral is positive for every $r \in (0, \infty)$, we must have that $\eta((0,\infty)^{n})>0$. \\
 	The same proof for the converse relation implies the other  assumptions by Lemma   \ref{exm02xn} and Lemma \ref{genlancastercartesianproduct}. \end{proof}

 \section{PDI functions  based on sums.}\label{sumsbased}

 In \cite{Guo1993}   it is proved a generalization of Theorem \ref{reprcondneg}, as  it describes, for any $\ell \in \mathbb{Z}_{+}$,  the set of  continuous functions  $\psi:[0, \infty) \to \mathbb{R}$, that  satisfies 
 $$ 
 \sum_{i,j=1}^{n}c_{i}c_{j} \psi(\|x_{i}-x_{j}\|^{2})\geq 0  
 $$
 for every points $x_{1}, \ldots, x_{n} \in \mathbb{R}^{d}$ and scalars $c_{1}, \ldots, c_{n}$, under no restriction on the dimension $d$, but such that  
 $$ 
 \sum_{i=1}^{n}c_{i}p(x_{i})=0, \quad p \text{ is a polynomial of degree less  or equal to } \ell-1.
 $$
 If the function $\psi$ satisfies this requirement we say that the kernel $\psi(\|x-y\|^{2})$ is Conditionally Positive Definite of Order $\ell$ (CPD$_{\ell}$) , see Chapter 8 in \cite{Wendland2005}. Note that the case $\ell =1$ we are dealing with constants (polynomials of degree zero), hence $-\psi(\|x-y\|^{2})$ is an  CND kernel. The case $\ell =0$ is of  PD kernels.  	
 
 \begin{theorem}\label{compleelltimes}  The following conditions are equivalent for a continuous functions $\psi: [0, \infty) \to \mathbb{R}$ 
 	\begin{enumerate}
 		\item[$(i)$] The kernel
 		$$
 		(x,y) \in \mathbb{R}^{d}\times \mathbb{R}^{d} \to \psi(\|x-y\|^{2}) \in \mathbb{R}
 		$$
 		is CPD$_{\ell}$  for every $d \in \mathbb{N}$. 
 		\item[$(ii)$] The function $\psi$ can be represented as
 		$$	
 		\psi(t)= \int_{(0,\infty)} (e^{-tr} - e_{\ell}(r)\omega_{\ell}(rt))\frac{1+r^{\ell}}{r^{\ell}} d\eta(r) + \sum_{k=0}^{\ell}a_{k}t^{k}
 		$$
 		where $\sigma$ is a nonnegative  measure in $\mathfrak{M}((0,\infty))$, with
 		$$
 		\omega_{\ell} (s):= \sum_{j=0}^{\ell-1}(-1)^{j}\frac{s^{j}}{j!}, \quad e_{\ell}(s):= e^{-s}\sum_{j=0}^{\ell-1}\frac{s^{j}}{j!},
 		$$
 		and $(-1)^{\ell}a_{\ell} \geq 0$. The representation is unique.
 		\item[$(iii)$] The function $\psi \in C^{\infty}((0,\infty))$ and the function  $(-1)^{\ell}\psi^{(\ell)}$ is  completely monotone, that is, $(-1)^{n+\ell}\psi^{(n+\ell)}(t) \geq 0$, for every $n\in \mathbb{Z}_{+}$ and $t>0$.
 	\end{enumerate}
 \end{theorem}

 A function that satisfies the equivalence on Theorem \ref{compleelltimes}  is called a  completely monotone function of order $\ell$  (CM$_{\ell}$). For instance, the functions
 \begin{enumerate}
 	\item[$i)$] $(-1)^{\ell}t^{a} $;
 	\item[$ii)$] $(-1)^{\ell}t^{\ell-1}\log (t)$;
 	\item[$iii)$] $(-1)^{\ell}( c+t )^{a}$;
 	\item[$iv)$] $e^{-rt}$, 
 \end{enumerate}
 are elements of CM$_{\ell}$, for $\ell-1 < a \leq \ell$  and  $c>0$.

 In  \cite{guella2023} the case $\ell=2$ was proved to have a  connection with radial PDI$_{2,n}^{\infty}$ kernels in $(\mathbb{R}^{d})_{2}$ for every $d \in \mathbb{N}$, that is, independence tests. In this Section we go one step further and relate a general $\ell$ with  PDI$_{\ell, n}^{\infty}$  functions.

 In order to prove Theorem \ref{ktimes}    we will need the following special representation of completely monotone functions with $\ell$ variables.
 
 \begin{lemma}\label{sumcm} Let $g:(0, \infty) \to \mathbb{R}$, the function   $ t \in (0, \infty)^{\ell} \to g(t_{1} + \ldots + t_{\ell})$ is completely monotone with $\ell$ variables if and only if $g$ is completely monotone with one variable. 
 \end{lemma}
 
 \begin{proof} For the converse, if $g$ is completely monotone, then
 	$$	
 	(-1)^{|\alpha|}\partial^{\alpha}[g(t_{1} + \ldots + t_{\ell}))] = (-1)^{|\alpha|}g^{(|\alpha|)} (t_{1} + \ldots + t_{\ell}) \geq 0.	
 	$$	
 	Conversely, if the function $g(t_{1} + \ldots + t_{\ell})$ is completely monotone with $\ell$ variables, then for any $c >0$ define $t_{i}= c/(\ell -1)$ to obtain that  
 	$$	 
 	(-1)^{k }g^{(k)} (t_{1} + c) = (-1)^{k}\partial_{1}^{k}[g(t_{1} + c))] \geq 0, \quad k \in \mathbb{Z}_{+}	
 	$$
 	which proves that $g$ is completely monotone.\end{proof}

 In particular, for  a function $g$ that satisfies Lemma \ref{sumcm}, we have that
 
 $$
 g(t_{1} + \ldots + t_{\ell})= \int_{[0, \infty)}e^{-r(t_{1} + \ldots + t_{\ell})}d\eta(r),
 $$
 that is, the measure that represents $g(t_{1} + \ldots + t_{\ell})$ has support on $\{(r\vec{1}), \quad  r \in [0, \infty)  \} \subset [0, \infty)^{\ell}$. 
 
 In the  proof of   Theorem    \ref{ktimes}, we use the Frechet functional equation, which is a generalization of the Cauchy functional equation and  is concerned with which functions $z: \mathbb{R} \to \mathbb{R} $ satisfies the following relation
 $$
 \sum_{F \subset \{1, \ldots, \ell\}}(-1)^{\ell-|F|}z(\overline{t_{F}})=0, \quad t_{1}, \ldots, t_{\ell} \in \mathbb{R}^{\ell} \text{ with } \overline{t_{\emptyset}}:=0, \quad  \overline{t_{F}}:= \sum_{i \in F}t_{i}.
 $$ 
 Under some very weak assumptions on the function $z$, which continuity is a special case,  it can be proved that $z$ must be a polynomial of degree less  or equal to $\ell-1$. For the proof of Theorem \ref{ktimes} we need a  consequence  of this result: assume that  $z$ is continuous but it is defined on $[0, \infty)$ and the functional equation only holds for $t_{1}, \ldots, t_{\ell} \in [0, \infty)$. Likewise, the function $z$ still is a polynomial of degree less than or equal to  $\ell-1$, and an argument can be found in \cite{kuczma1964equation} or in Theorem $13,6$ page $273$ in \cite{kuczma1968functional}. We leave to the reader the proof that for any $0 \leq k \leq \ell-1 $         
 \begin{equation}\label{eq1frechet}
 	\sum_{F \subset \{1, \ldots, \ell\}}(-1)^{\ell-|F|}(\overline{t_{F}})^{k}=0, \quad t_{1}, \ldots, t_{\ell} \in \mathbb{R}
 \end{equation}
 \begin{equation}\label{eq2frechet}
 	\sum_{F \subset \{1, \ldots, \ell\}}(-1)^{\ell-|F|}(\overline{t_{F}})^{\ell}=\ell!\prod_{i=1}^{\ell}t_{i}, \quad t_{1}, \ldots, t_{\ell} \in \mathbb{R}
 \end{equation}
 which can be proved by the multinomial Theorem.

 \begin{theorem}\label{ktimes} Let $\psi:[0, \infty) \to \mathbb{R}$ be a continuous function and $\ell \in \mathbb{N}$. The following conditions are equivalent:
 	\begin{enumerate}
 		\item [$(i)$] The function $(-1)^{\ell}\psi(t_{1}+ \ldots + t_{n})$ is PDI$_{\ell,n}^{\infty}$.
 		\item [$(ii)$] 	The kernel
 		$$
 		(x,y) \in \mathbb{R}^{d}\times \mathbb{R}^{d} \to \psi(\|x-y\|^{2}) \in \mathbb{R}
 		$$
 		is CPD$_{\ell}$  for every $d \in \mathbb{N}$. 
 		\item [$(iii)$] The function $\psi$ can be represented as
 		$$
 		\psi(t)= \sum_{k=0}^{\ell}a_{k}t^{k}  + \int_{(0,\infty)}(e^{-rt} - e_{\ell}(r)\omega_{\ell}(rt))\frac{(1+r)^{\ell}}{r^{\ell}} d\eta(r) 
 		$$
 		where  $(-1)^{\ell}a_{\ell} \geq 0$   and $\eta \in \mathfrak{M}([0, \infty))$ is a nonnegative  measure. The representation is unique.
 		\item [$(iv)$] The function $\psi$ is  a completely monotone function of order $\ell$, that is, $\psi \in C^{\infty}((0, \infty)) $ and $(-1)^{\ell}\psi^{(\ell)}$ is a completely monotone function.
 	\end{enumerate}
 \end{theorem}

 \begin{proof} Relations $(ii)$, $(iii)$  and $(iv)$ are equivalent by Theorem \ref{compleelltimes}  and the reparametrizatization does not affect the class of functions obtained as 
 	$$
 	\frac{1+r^{\ell}}{r^{\ell}} \leq \frac{(1+r)^{\ell}}{r^{\ell}} \leq \ell!\frac{1+r^{\ell}}{r^{\ell}}, \quad r \in (0, \infty).
 	$$
 	If relation $(ii)$ holds, pick  arbitrary  $d, m\in \mathbb{N}$,  points  $x_{i}^{1}, \ldots, x_{i}^{n} \in \mathbb{R}^{d}$, $1 \leq i \leq m$ 	and scalars $c_{\alpha}$, $\alpha \in \mathbb{N}_{m}^{n}$, that satisfies that the measure $\sum_{\alpha \in  \mathbb{N}_{m}^{n}}c_{\alpha}\delta_{x_{\alpha}} \in \mathcal{M}_{k}((\mathbb{R}^{d})_{n})$. Define the $m^{n}$ vectors $u^{\alpha}:=(x_{\alpha_{1}}, \ldots, x_{\alpha_{n}}) \in \mathbb{R}^{dn}$, and note that for every polynomial  $p: \mathbb{R}^{dn}  \to \mathbb{R}$   of degree less than or equal to  $\ell-1$, we have that
 	$$
 	\sum_{\alpha \in \mathbb{N}_{m}^{\ell}}c_{\alpha}p(u_{\alpha})=0
 	$$
 	because $p$ does not depend on the $\ell$ indexes among the possible $n$ of a vector $(z_{1}, z_{2}, \ldots , z_{n}) \in (\mathbb{R}^{d})^{n}$, see Equation \ref{integmu0n}, and consequently relation $(i)$ holds because
 	\begin{align*}
 		\sum_{\alpha, \beta \in \mathbb{N}_{m}^{\ell}}&c_{\alpha}c_{\beta}(-1)^{\ell}\left[ (-1)^{\ell}\psi(\|x_{1}^{\alpha_{1}}- x_{1}^{\beta_{1}}\|^{2} + \ldots + \|x_{n}^{\alpha_{n}}- x_{n}^{\beta_{n}}\|^{2})\right ]\\
 		&=\sum_{\alpha, \beta \in \mathbb{N}_{m}^{\ell}}c_{\alpha}c_{\beta}\psi(\|u^{\alpha}-u^{\beta}\|^{2})\geq 0.
 	\end{align*}
 	Now, we prove that  relation $(i)$ implies relation $(ii)$ on the case $n=\ell$ and from this we conclude the general case for an arbitrary $n\geq \ell$.\\
 	By specializing Lemma \ref{PDInsimpli} to the function $g(t):=(-1)^{\ell}\psi(t_{1} + \ldots +t_{\ell})$ we get that
 	$$
 	G(t):=\sum_{F\subset \{1, \ldots, \ell \}}(-1)^{|F|}\psi(\overline{t_{F}})=\int_{[0,\infty)^{n}}\prod_{i=1}^{n}(1-e^{-r_{i}t_{i}})\frac{1 +r_{i}}{r_{i}}d\eta(r) 
 	$$ 
 	is a Bernstein function of $\ell$ variables in $(0, \infty)^{\ell}$ that is zero at $\partial_{\ell-1}^{\ell}$. Hence, $\partial^{\vec{1}}G$ is completely monotone, but   $\partial^{\vec{1}}G$ only depends on the value $\sum_{i=1}^{\ell}t_{i}$. Indeed, for every fixed subset $F$ and $t_{\vec{1}}, t_{\vec{2}} \in [0, \infty)^{n}$  we have that
 	$$
 	\sum_{\alpha \in \mathbb{N}^{\ell}_{2} }(-1)^{|\alpha|} \psi(\overline{(t_{\alpha})_{F}}) = \int_{[0, \infty)^{\ell}}\psi\left (\overline{u_{F}}\right )d[\bigtimes_{i=1}^{\ell}(\delta_{t_{i}^{1}} - \delta_{t_{i}^{2}}) ](u_{1}, \ldots, u_{\ell}),
 	$$
 	and whenever  $|F| < \ell$	this sum is zero by Equation \ref{integmu0n}. Hence,  if  $h \in (0, \infty)^{\ell}$ define $t_{\vec{2}}:= t_{\vec{1}} + h$ 
 	\begin{align*}
 		\sum_{\alpha \in \mathbb{N}^{\ell}_{2} }(-1)^{|\alpha|} G(t_{\alpha})&=
 		\sum_{\alpha \in \mathbb{N}^{\ell}_{2} }(-1)^{|\alpha|} \left [\sum_{F\subset \{1, \ldots, \ell \}}(-1)^{|F|}\psi(\overline{(t_{\alpha})_{F}}) \right ]\\
 		&=\sum_{F\subset \{1, \ldots, \ell \}}(-1)^{|F|} \left [\sum_{\alpha \in \mathbb{N}^{\ell}_{2} }(-1)^{|\alpha|}\psi(\overline{(t_{\alpha})_{F}}) \right ]\\
 		&=(-1)^{\ell} \sum_{\alpha \in \mathbb{N}^{\ell}_{2} }(-1)^{|\alpha|} \psi(\overline{(t_{\alpha})_{\{1,\ldots ,\ell\}}})+0\\
 		& = \sum_{F\subset \{1, \ldots, \ell \}}(-1)^{|F|} \psi(\overline{t_{\vec{1}}} + \overline{h_{F}}),  	
 	\end{align*}
 	which concludes the proof because  
 	$$	
 	\partial^{\vec{1}}G(t_{\vec{1}})= \lim_{ t_{\vec{2}}  \to t_{\vec{1}}}\frac{\sum_{\alpha \in \mathbb{N}^{\ell}_{2} }(-1)^{|\alpha|} G(t_{\alpha})}{\prod_{i=1}^{\ell}(t_{i}^{2} - t_{i}^{1})}= \lim_{ h   \to \vec{0}}\frac{\sum_{ F\subset \{1, \ldots, \ell \}}(-1)^{|F|} \psi(\overline{t_{\vec{1}}} + \overline{h_{F}})}{\prod_{i=1}^{\ell}h_{i}}.	
 	$$	
 	By Lemma \ref{sumcm}  and proof of  Theorem \ref{bernssevndim}, the support of the measure  $\eta$ that represents $G $ is contained on the set $\{(r,r,\ldots, r) \in [0, \infty)^{\ell}, \quad r \in [0, \infty)\}$, hence we may assume that $\eta \in \mathfrak{M}([0, \infty))$. On the other hand, define the function 
 	$$
 	\varphi(t):=(-1)^{\ell}t^{\ell}\frac{\eta(\{0\})}{\ell!}  + \int_{(0,\infty)}(e^{-rt} - e_{\ell}(r)\omega_{\ell}(rt))\frac{(1+r)^{\ell}}{r^{\ell}} d\eta(r) 
 	$$
 	which is a well defined continuous function by Theorem \ref{compleelltimes}. The key step in the proof is to note that by Equation \ref{eq1frechet} and Equation \ref{eq2frechet}, we have that
 	\begin{align*}
 		(-1)^{\ell}&\sum_{F\subset \{1, \ldots, \ell \}}(-1)^{\ell-|F|}\varphi(\overline{t_{F}})\\
 		&= \prod_{i=1}^{\ell}t_{i}\eta(\{0\})  + \int_{(0, \infty)} \left [\prod_{i=1}^{\ell}(1-e^{-rt_{i}}) \right ]\frac{(1+r)^{\ell}}{r^{\ell}}d\eta(r) = G(t ) 
 	\end{align*}
 	But then, the function $z(t):=\psi(t) -(-1)^{\ell}\varphi(t)$ satisfies the Frechet functional equation on $[0, \infty)$, implying that $z$ is a polynomial of degree less  or equal to $\ell-1$, which concludes the proof that relation $(i)$ implies relation $(iii)$. \end{proof}

 \begin{remark}\label{difguo} A closer look at Theorem \ref{ktimes}, reveals that with it we can obtain that relation $iv)$ implies relation $iii)$ with different arguments then the one presented in \cite{Guo1993}.\\
 	Indeed, note that relation $iv)$ implies relation $i)$ by direct verification and we proved on the Theorem that   relation $i)$ implies relation $iii)$. 	  
 \end{remark}

 	\section*{Acknowledgments}
 
 This work was partially funded by São Paulo Research Foundation (FAPESP) under grants
 2021/04226-0 and 2022/00008-0

\bibliographystyle{siam}
\bibliography{References}

\begin{thebibliography}{10}

\bibitem{Albert2022}
{\sc M.~Albert, B.~Laurent, A.~Marrel, and A.~Meynaoui}, {\em {Adaptive test of
  independence based on HSIC measures}}, The Annals of Statistics, 50 (2022),
  pp.~858--879.

\bibitem{AlonsoMalaver2015}
{\sc C.~E. Alonso-Malaver, E.~Porcu, and R.~Giraldo}, {\em Multivariate and
  multiradial {S}choenberg measures with their dimension walks}, Journal of
  Multivariate Analysis, 133 (2015), pp.~251--265.

\bibitem{Bakirov2006}
{\sc N.~K. Bakirov, M.~L. Rizzo, and G.~J. Sz{\'e}kely}, {\em A multivariate
  nonparametric test of independence}, Journal of Multivariate Analysis, 97
  (2006), pp.~1742--1756.

\bibitem{Berg1984}
{\sc C.~Berg, J.~Christensen, and P.~Ressel}, {\em Harmonic analysis on
  semigroups: theory of positive definite and related functions}, vol.~100 of
  Graduate Texts in Mathematics, Springer, 1984.

\bibitem{Bochner2005}
{\sc S.~Bochner}, {\em Harmonic analysis and the theory of probability},
  Courier Corporation, 2005.

\bibitem{Boettcher2018}
{\sc B.~B{\"o}ttcher, M.~Keller-Ressel, and R.~Schilling}, {\em Detecting
  independence of random vectors: generalized distance covariance and gaussian
  covariance}, Modern Stochastics: Theory and Applications, 5 (2018),
  pp.~353--383.

\bibitem{Boettcher2019}
{\sc B.~{B{\"o}ttcher}, M.~{Keller-Ressel}, and R.~L. {Schilling}}, {\em
  Distance multivariance: New dependence measures for random vectors}, Annals
  of Statistics, 47 (2019), pp.~2757--2789.

\bibitem{Chakraborty2019}
{\sc S.~Chakraborty and X.~Zhang}, {\em Distance metrics for measuring joint
  dependence with application to causal inference}, Journal of the American
  Statistical Association, 114 (2019), pp.~1638--1650.

\bibitem{NIST:DLMF}
{\em {\it NIST Digital Library of Mathematical Functions}}.
\newblock F.~W.~J. Olver, A.~B. {Olde Daalhuis}, D.~W. Lozier, B.~I. Schneider,
  R.~F. Boisvert, C.~W. Clark, B.~R. Miller, B.~V. Saunders, H.~S. Cohl, and
  M.~A. McClain, eds.

\bibitem{Dueck2014}
{\sc J.~Dueck, D.~Edelmann, T.~Gneiting, and D.~Richards}, {\em {The affinely
  invariant distance correlation}}, Bernoulli, 20 (2014), pp.~2305--2330.

\bibitem{fernandez2003flexible}
{\sc R.~Fern{\'a}ndez-Casal, W.~Gonz{\'a}lez-Manteiga, and M.~Febrero-Bande},
  {\em Flexible spatio-temporal stationary variogram models}, Statistics and
  Computing, 13 (2003), pp.~127--136.

\bibitem{Feuerverger1993}
{\sc A.~Feuerverger}, {\em A consistent test for bivariate dependence}, 61,
  pp.~419--433.

\bibitem{Gretton2005}
{\sc A.~Gretton, O.~Bousquet, A.~Smola, and B.~Sch{\"o}lkopf}, {\em Measuring
  statistical dependence with hilbert-schmidt norms}, in International
  conference on algorithmic learning theory, Springer, 2005, pp.~63--77.

\bibitem{Gretton2008}
{\sc A.~Gretton, K.~Fukumizu, C.~H. Teo, L.~Song, B.~Sch{\"o}lkopf, and A.~J.
  Smola}, {\em A kernel statistical test of independence}, in Advances in
  neural information processing systems, 2008, pp.~585--592.

\bibitem{guella2023}
{\sc J.~C. Guella}, {\em Generalization of the hsic and distance covariance
  using pdi kernels}, Banach Journal of Mathematical Analysis, 17 (2023).

\bibitem{Guella2020}
{\sc J.~C. Guella and V.~A. Menegatto}, {\em Conditionally positive definite
  matrix valued kernels on {E}uclidean spaces}, Constructive Approximation, 52
  (2020), pp.~65--92.

\bibitem{Guo1993}
{\sc K.~Guo, S.~Hu, and X.~Sun}, {\em Conditionally positive definite functions
  and {L}aplace-{S}tieltjes integrals}, Journal of Approximation Theory, 74
  (1993), pp.~249--265.

\bibitem{ip2004structural}
{\sc E.~H. Ip, Y.~J. Wang, and Y.-n. Yeh}, {\em Structural decompositions of
  multivariate distributions with applications in moment and cumulant}, Journal
  of multivariate analysis, 89 (2004), pp.~119--134.

\bibitem{Janson2021}
{\sc S.~Janson}, {\em On distance covariance in metric and hilbert spaces},
  Latin American Journal of Probability and Mathematical Statistics, 18 (2021),
  pp.~1353--1393.

\bibitem{kuczma1964equation}
{\sc M.~Kuczma}, {\em Sur une {\'e}quation aux diff{\'e}rences finies et une
  caract{\'e}risation fonctionnelle des polyn{\^o}mes}, Fundamenta
  Mathematicae, 55 (1964), pp.~77--86.

\bibitem{kuczma1968functional}
{\sc M.~Kuczma}, {\em Functional equations in a single variable}, Pwn-Polish
  Scientific Pub, 1968.

\bibitem{lancaster1969chi}
{\sc H.~O. Lancaster}, {\em The chi-squared distribution}, (No Title),  (1969).

\bibitem{NEURIPS2023_74f11936}
{\sc Z.~Liu, R.~Peach, P.~A. Mediano, and M.~Barahona}, {\em Interaction
  measures, partition lattices and kernel tests for high-order interactions},
  in Advances in Neural Information Processing Systems, A.~Oh, T.~Naumann,
  A.~Globerson, K.~Saenko, M.~Hardt, and S.~Levine, eds., vol.~36, Curran
  Associates, Inc., 2023, pp.~36991--37012.

\bibitem{macdonald1998symmetric}
{\sc I.~G. Macdonald}, {\em Symmetric functions and Hall polynomials}, Oxford
  university press, 1998.

\bibitem{MartinezGomez2014}
{\sc E.~Mart{\'{\i}}nez-G{\'{o}}mez, M.~T. Richards, and D.~S.~P. Richards},
  {\em Distance correlation methods for discovering associations in large
  astrophysical databases}, The Astrophysical Journal, 781 (2014), p.~39.

\bibitem{Mirotin2013rohtua}
{\sc A.~R. Mirotin}, {\em Properties of bernstein functions of several complex
  variables}, Mathematical Notes, 93 (2013), pp.~257--265.

\bibitem{petersjonasjointindp}
{\sc N.~Pfister, P.~B{\"u}hlmann, B.~Sch{\"o}lkopf, and J.~Peters}, {\em
  Kernel-based tests for joint independence}, Journal of the Royal Statistical
  Society: Series B (Statistical Methodology), 80 (2018), pp.~5--31.

\bibitem{Poczos2012}
{\sc B.~Poczos, Z.~Ghahramani, and J.~Schneider}, {\em Copula-based kernel
  dependency measures}, in Proceedings of the 29th International Conference on
  Machine Learning (ICML-12), J.~Langford and J.~Pineau, eds., ICML '12, New
  York, NY, USA, 2012, Omnipress, pp.~775--782.

\bibitem{Schilling2012}
{\sc R.~L. Schilling, R.~Song, and Z.~Vondracek}, {\em Bernstein functions:
  theory and applications}, vol.~37, Walter de Gruyter, 2012.

\bibitem{schoenbradial}
{\sc I.~J. Schoenberg}, {\em Metric spaces and completely monotone functions},
  Annals of Mathematics, 39 (1938), pp.~811--841.

\bibitem{NIPS2013_076a0c97}
{\sc D.~Sejdinovic, A.~Gretton, and W.~Bergsma}, {\em A kernel test for
  three-variable interactions}, in Advances in Neural Information Processing
  Systems, C.~Burges, L.~Bottou, M.~Welling, Z.~Ghahramani, and K.~Weinberger,
  eds., vol.~26, Curran Associates, Inc., 2013, pp.~1--9.

\bibitem{sejdinovic2013equivalence}
{\sc D.~Sejdinovic, B.~Sriperumbudur, A.~Gretton, and K.~Fukumizu}, {\em
  Equivalence of distance-based and rkhs-based statistics in hypothesis
  testing}, The Annals of Statistics, 41 (2013), pp.~2263--2291.

\bibitem{streitberg1990lancaster}
{\sc B.~Streitberg}, {\em Lancaster interactions revisited}, The Annals of
  Statistics,  (1990), pp.~1878--1885.

\bibitem{Szekely2009}
{\sc G.~J. {Sz{\'e}kely} and M.~L. {Rizzo}}, {\em Brownian distance
  covariance}, The Annals of Applied Statistics, 3 (2009), pp.~1236--1265.

\bibitem{Szekely2013}
{\sc G.~J. Sz{\'e}kely and M.~L. Rizzo}, {\em The distance correlation t-test
  of independence in high dimension}, Journal of Multivariate Analysis, 117
  (2013), pp.~193--213.

\bibitem{Szekely2014}
\leavevmode\vrule height 2pt depth -1.6pt width 23pt, {\em {Partial distance
  correlation with methods for dissimilarities}}, The Annals of Statistics, 42
  (2014), pp.~2382--2412.

\bibitem{Szekely2007}
{\sc G.~J. Sz{\'e}kely, M.~L. Rizzo, and N.~K. Bakirov}, {\em {Measuring and
  testing dependence by correlation of distances}}, The Annals of Statistics,
  35 (2007), pp.~2769--2794.

\bibitem{Tjoestheim2022}
{\sc D.~Tj{\o}stheim, H.~Otneim, and B.~St{\o}ve}, {\em {Statistical
  Dependence: Beyond Pearson's rho }}, Statistical Science, 37 (2022),
  pp.~90--109.

\bibitem{Wendland2005}
{\sc H.~Wendland}, {\em Scattered data approximation}, vol.~17, Cambridge
  university press, 2005.

\bibitem{Yao2018}
{\sc S.~Yao, X.~Zhang, and X.~Shao}, {\em Testing mutual independence in high
  dimension via distance covariance}, Journal of the Royal Statistical Society:
  Series B (Statistical Methodology), 80, pp.~455--480.

\bibitem{Zhu2020}
{\sc C.~Zhu, X.~Zhang, S.~Yao, and X.~Shao}, {\em {Distance-based and
  RKHS-based dependence metrics in high dimension}}, The Annals of Statistics,
  48 (2020), pp.~3366--3394.

\end{thebibliography}

\end{document}